\documentclass[final]{siamltex}

\usepackage{amssymb}
\usepackage{amsmath, amsfonts}
\usepackage{graphicx}
\usepackage{setspace}
\usepackage{color}
\usepackage{enumitem}
\usepackage{float}
\usepackage{caption}
\usepackage{subfigure}

\usepackage[left = 100pt, right = 100pt, bottom = 50pt]{geometry}

\def\tf{\tilde{f}}   % interface state
\def\calT{{\cal T}}  % Tumbling rate: decoupled
\def\tfe{g}

\def\Na{\mathbb{N}}  % Positive integers
\def\Re{\mathbb{R}}  % Reals
\def\Ze{\mathbb{Z}}  % Integers

\def\ds{\displaystyle}
\def\BE{\begin{equation}}
\def\EE{\end{equation}}
\def\DT{\Delta t}
\def\DX{\Delta x}

\def\Dt{\partial_t}
\def\Dx{\partial_x}

\def\Dxx{\partial_{xx}}

\def\vp{\varphi}
\def\calV{\mathcal{V}}
\def\calW{\mathcal{W}}

\def\sign{{\rm sgn}}
\newcommand{\lbeq}[1]{{\label{OR:eq:#1}}}

\newcommand{\be}[1]{\begin{equation} \lbeq{#1}}
\newcommand{\ee}{\end{equation}}

\newtheorem{rmk}{Remark}

\def\V{{\mathbf V}}

\def\O{{\mathbf \Omega}}

\def\EL{{\mathcal L}}

\newcommand{\chS}{M}   % chemical signal
\newcommand{\chN}{N}
\newcommand{\nG}{K}    % size of Gauss quadrature
\newcommand{\nGG}{\mathcal{K}}
\def\lsum{{\ell\in\nGG}}    % summing over all discrete velocities
\def\ksum{{k \in\nGG}}    % summing over all discrete velocities

\title{A well-balanced scheme for chemotactic travelling waves at the mesoscopic scale}

\author{Vincent Calvez\footnote{CNRS \& Institut Camille Jordan, Universit\'e de Lyon 1, and Inria, project-team NUMED, Lyon, France, \tt{vincent.calvez@math.cnrs.fr}}
\and
Laurent Gosse\footnote{IAC, CNR, via dei Taurini, 19, 00185 Roma (Italia), 
\tt{l.gosse@ba.iac.cnr.it}}
\and
Monika Twarogowska\footnote{ Unit\'e de Math\'ematiques Pures et Appliqu\'ees, Ecole Normale Sup\'erieure de Lyon, and Inria, project-team NUMED, Lyon, France,  
\tt{monika.twarogowska@ens-lyon.fr}}
}

\begin{document}

\maketitle

\begin{abstract}
We investigate numerically a model consisting in a kinetic equation for the biased motion of bacteria following a run-and-tumble process, coupled with two reaction-diffusion equations for chemical signals. This model exhibits asymptotic propagation at a constant speed. In particular, it admits travelling wave solutions. To capture this propagation, we propose a well-balanced numerical scheme based on Case's elementary solutions for the kinetic equation, and $\mathcal{L}$-splines for the parabolic equations. We use this scheme to explore the Cauchy problem for various parameters. Some examples far from the diffusive regime lead to the co-existence of two waves travelling at different speeds. Numerical tests support the hypothesis that they are both locally asymptotically stable.  Interestingly, the exploration of the bifurcation diagram raises counter-intuitive features.
\end{abstract}



\pagestyle{myheadings} \thispagestyle{plain} \markboth{V. Calvez, L. Gosse \& M. Twarogowska}{Numerical capture of chemotactic solitary waves}

%%%%%%%%%%%%%%%%%%%%%%%%%%%%%%%%%%%%%%%%%%%%%%%%%%%%%%%%%%%%%%%%%%%%%%%%%%%%%%%%%%%%%%%%%%%
%%%%%%%%%%%%%%%%%%%%%%%%%%%%%%%%%%%%%%%%%%%%%%%%%%%%%%%%%%%%%%%%%%%%%%%%%%%%%%%%%%%%%%%%%%%
\section{Introduction}
%%%%%%%%%%%%%%%%%%%%%%%%%%%%%%%%%%%%%%%%%%%%%%%%%%%%%%%%%%%%%%%%%%%%%%%%%%%%%%%%%%%%%%%%%%%
%%%%%%%%%%%%%%%%%%%%%%%%%%%%%%%%%%%%%%%%%%%%%%%%%%%%%%%%%%%%%%%%%%%%%%%%%%%%%%%%%%%%%%%%%%%

%%%%%%%%%%%%%%%%%%%%%%%%%%%%%%%%%%%%%%%%%%%%%%%%%%%%%%%%%%%%%%%%%%%%%%%%%%%%%%%%%%%%%%%%%%%
\subsection{Concentration waves of chemotactic bacteria}
%%%%%%%%%%%%%%%%%%%%%%%%%%%%%%%%%%%%%%%%%%%%%%%%%%%%%%%%%%%%%%%%%%%%%%%%%%%%%%%%%%%%%%%%%%%

This work deals with numerical simulation of bacteria collective motion at the mesoscopic scale. In particular, we focus on wave propagation in the  long time asymptotics  (see Fig. \ref{fig:wave_introduction}).  
\begin{figure}[ht]
  \begin{center}
    \subfigure[]{
      \includegraphics[scale=0.45]{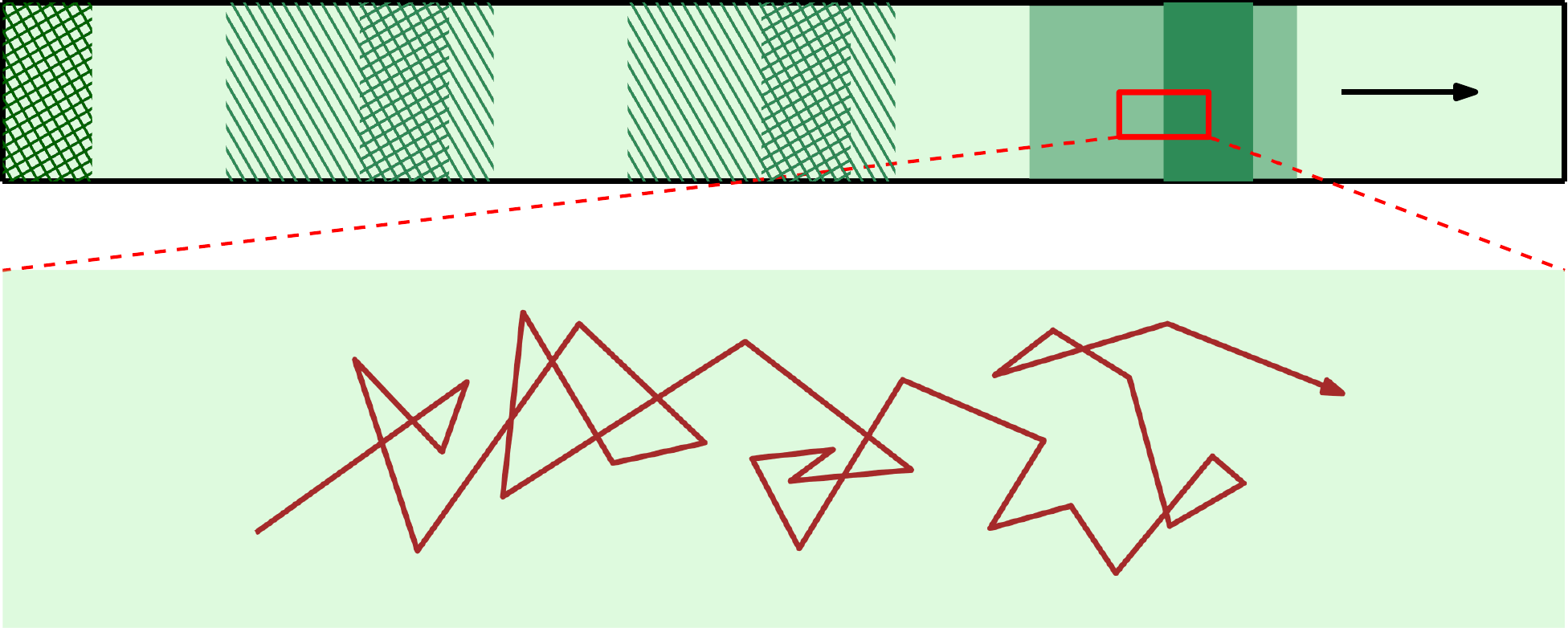} 
    }
    \subfigure[]{
      \begin{tabular}{cc}
        \includegraphics[scale=0.11]{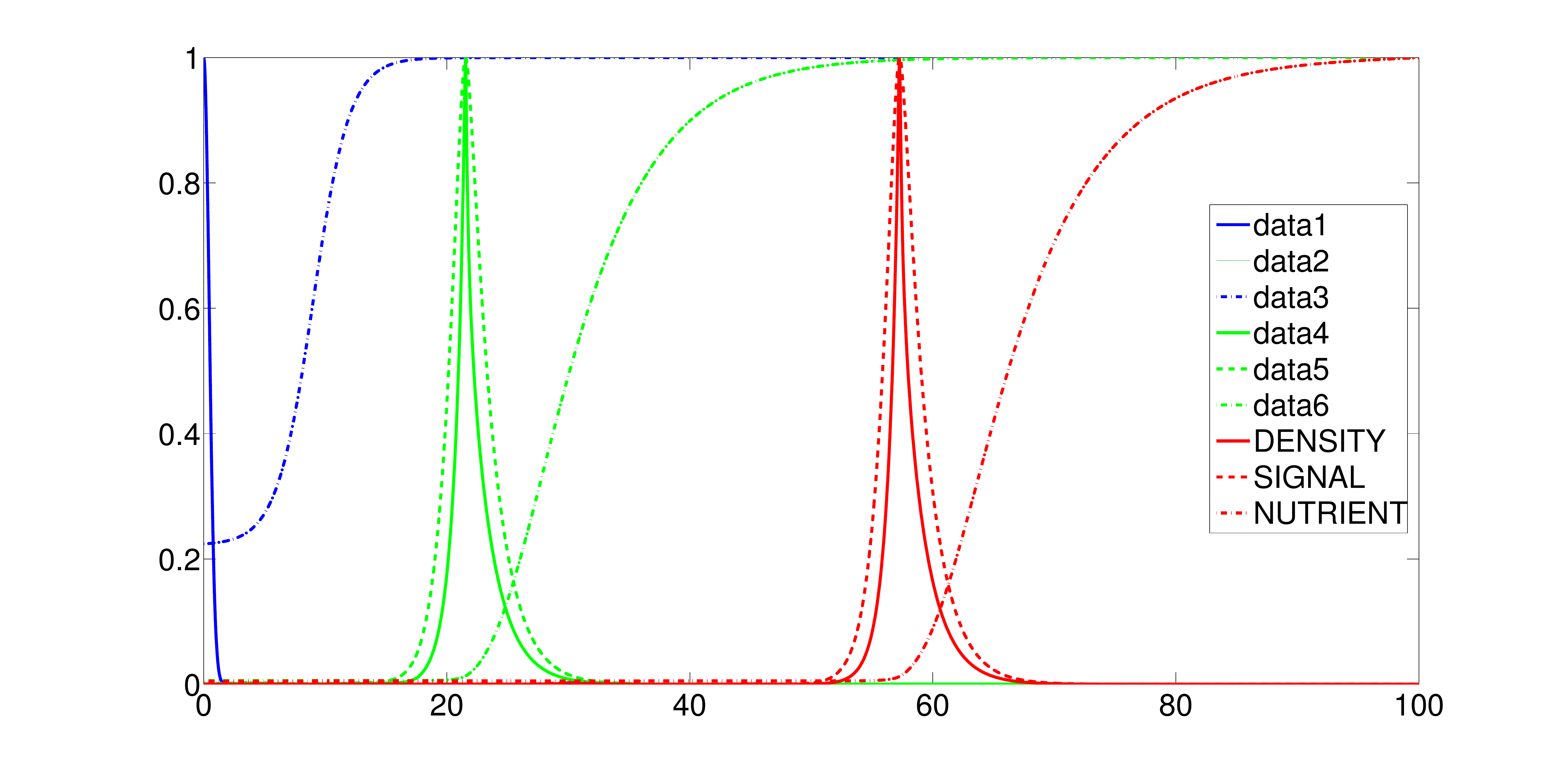}&\includegraphics[scale=0.11]{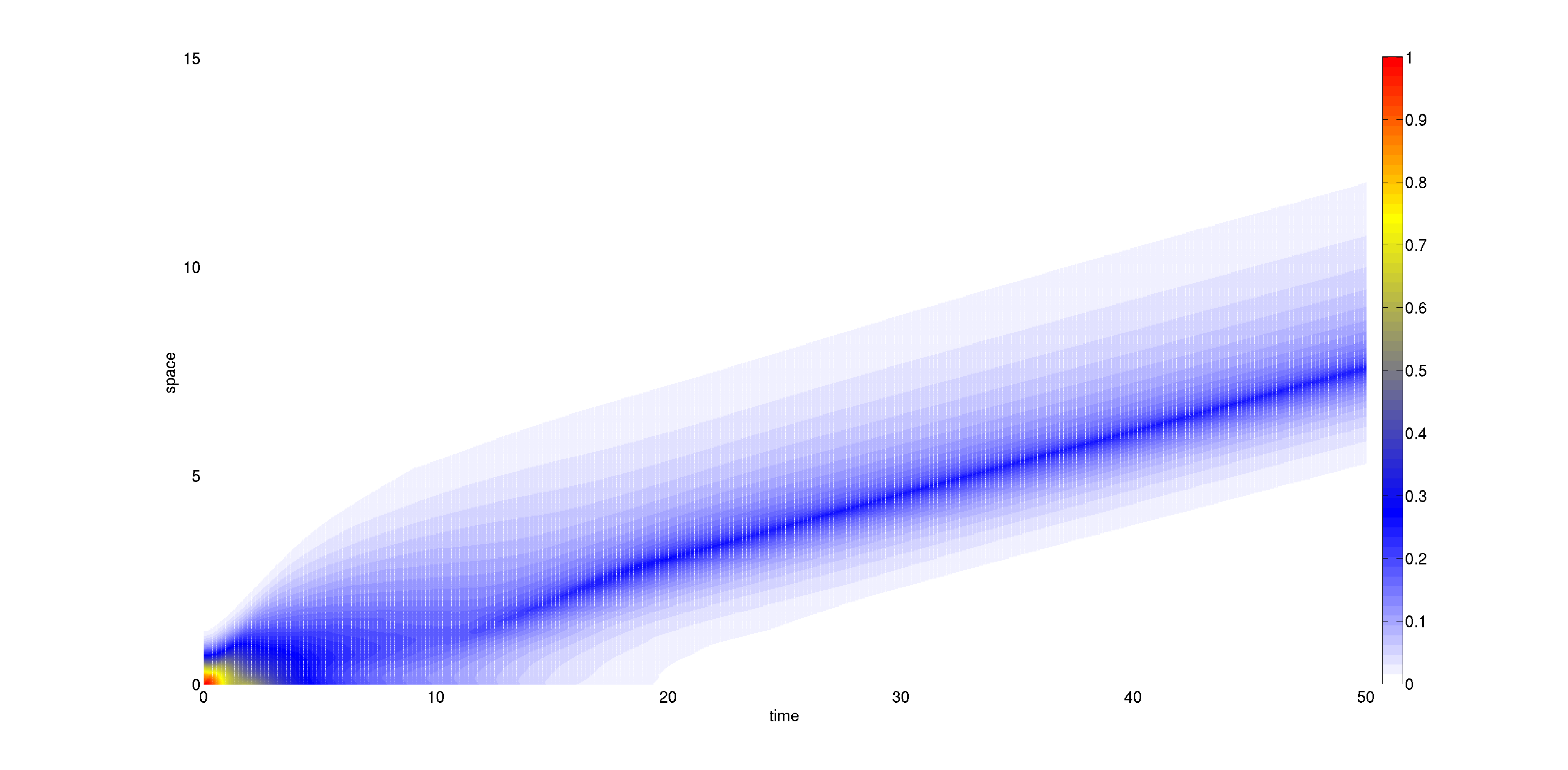}\\
      \end{tabular}
    }
  \end{center}
  \caption{(a) Cartoon of concentration waves of bacteria, as reported in \cite{adler_chemotaxis_1966,saragosti_directional_2011}: The population of bacteria is initially located on the left hand side of the channel after centrifugation. Shortly, a large fraction of the population detaches and propagate to the right side at constant speed. Individual trajectories follows a run-and-tumble process in first approximation: cells alternate between straight runs and fast reorientation events (tumbles). The duration of run phases is modulated by sensing temporal variations of the chemical gradients in the environment. We refer to \cite{Berg_book} for biological aspects of motions of {\em E. coli}. (b) Numerical simulation of a travelling wave emerging from system \eqref{eq:model}. (Left) Spatial distribution of density of cells and concentration of the signal and the nutrient. (Right) Density of a wave as a function of time and space.}
  \label{fig:wave_introduction}
\end{figure}
Bacteria perform run-and-tumble motion in a liquid medium \cite{berg_chemotaxis_1972,macnab_gradient-sensing_1972,Berg_book}. They alternate between run phases of ballistic motion and tumble phases of rotational diffusion. It is often assumed that tumbles are instantaneous reorientation events, where the cell changes velocity (from $v'$ to $v$, say). On the other hand, the duration of runs is modulated by temporal-sensing chemotaxis (chemokinesis). Accordingly, cells spend more time (in average) in favorable directions, for which the concentration of some molecular signal is increasing. This strategy allows them to navigate in heterogeneous environments. 

Remarkably, self-organized collective motion can emerge from this individual process. Here, we focus on concentration waves of bacteria {\em E. coli} in a capillary assay (or a micro-channel), as described in the seminal article by  Adler \cite{adler_chemotaxis_1966}, see \cite{tindall} for a review from the modelling perspective. The following 
model was
proposed in  \cite{saragosti_mathematical_2010,saragosti_directional_2011}, together with its validation on tracking experimental data: 
\begin{subequations}\label{eq:model}
\BE\label{kinetic}
\Dt f(t,x,v) + v\Dx f(t,x,v) = \int_V \mathcal{T}(t,x,v')f(t,x,v')  d\nu(v')- \mathcal{T}(t,x,v)f(t,x,v) .
\EE
The bacteria population is described by its density $f(t,x,v)$ in the position$\times$velocity space at time $t>0$. Here, we restrict to the one dimensional case $(x,v)\in \Re\times V$, as we seek planar travelling waves in the original three dimensional setting. Here, $V$ is the compact interval of admissible velocities, and $\nu$ is a symmetrical probability measure on this interval. 
The tumbling rate $\mathcal{T}(t,x,v)$ depends on time, space and velocity through several molecular signals, called  chemoattractants. 
Following \cite{salman_solitary_2006,xue_travelling_2010,saragosti_mathematical_2010,saragosti_directional_2011}, we make the hypothesis of two chemical species: an amino-acid signal $M$ released by the bacteria ({\em e.g.} aspartate, serine) and a nutrient $\chN$ consumed by the bacteria ({\em e.g.} oxygen, glucose). Assuming that both signals contribute additively to the tumbling rate $\mathcal{T}$, we assume that it is given by the following expression:  
\BE\label{tumbling}
\mathcal{T}(t,x,v)=1+\chi_{\chS} \cdot \phi\left(\left.\frac{D{\chS}}{Dt}\right|_{v}\right) + \chi_N\cdot \phi\left(\left.\frac{DN}{Dt}\right|_{v}\right), \qquad \phi(\cdot)=-\sign(\cdot),
\EE
where $\frac{D }{Dt}$ stands for the {\it material derivative} along the direction given by the velocity~$v$, {\it i.e.} $\frac{D }{Dt} = \Dt   + v \Dx  $.
The sign function has been chosen for at least three reasons: (i) There exists experimental evidence that bacteria can dramatically amplify small amplitudes during temporal sensing \cite{Berg_book, othmer_models_1988}. This motivates the choice of non linear functions such as the sign function. But see \cite{saragosti_directional_2011} where a more appropriate choice of sigmoidal function was proposed. (ii) Existence of travelling waves solution for this conservative problem rely on the specific choice of the sign function in \eqref{tumbling}. It would be highly relevant to replace it with $\phi\left(\left.\frac{D{\log \chS}}{Dt}\right|_{v'}\right)$, as suggested in \cite{tu_modeling_2008,kalinin_logarithmic_2009,zhu_frequency-dependent_2012,perthame_derivation_2015}, but the mathematical picture seems by far more complicated. (iii) It is a numerical challenge to cope with the lack of regularity of the sign function, and resulting consequences on the lack of regularity of density profiles.  

Chemoattractant concentrations evolve according to standard reaction-diffusion equations, with production, and uptake reaction terms, respectively: 
\BE\label{para-S-N}
\begin{cases}
\Dt \chS - D_{\chS} \Dxx \chS + \alpha \chS= \beta \rho,\medskip\\  \Dt N - D_N \Dxx N +\gamma \rho N= 0,
\end{cases}
\EE
\end{subequations}
where $D_{\chS}, D_{\chN}, \alpha,\beta,\gamma$ are positive constants, denoting respectively the diffusion coefficient of $\chS$, the diffusion coefficient of $\chN$, the rate of degradation of $\chS$, the rate of production of $\chS$ by the bacteria, and the rate of consumption of the nutrient $\chN$ by the bacteria. Also, $\rho$ denotes the spatial density of cells: 
\begin{displaymath}
\rho(t,x)=\int_V f(t,x,v)d\nu(v).
\end{displaymath}

Kinetic modeling of bacteria motion dates back to Stroock \cite{stroock_stochastic_1974} and Alt  \cite{alt_biased_1980}. We refer to \cite{othmer_models_1988,erban_signal_2005,dolak_kinetic_2005,chalub_model_2006,xue_macroscopic_2013,perthame_derivation_2015} for the description of the run-and-tumble process at multiple scales. In particular, \cite{xue_travelling_2010} and \cite{franz_travelling_2013} deals with the modelling of the same experiment with a similar model including an additional variable (the internal state of the bacteria). Also,  \cite{almeida_existence_2014,emako_traveling_2016} is concerned with the modelling of interactions between two strains within the same wave of propagation.
 
Kinetic models have been the basis for the derivation of macroscopic models of cell chemotaxis \cite{othmer_models_1988, Othmer_Hillen2,chalub_kinetic_2004,dolak_kinetic_2005,stevens_drift-diffusion_2005, erban_signal_2005,hwang_global_2006,xue_macroscopic_2013,Perthame}. Mathematical analysis of kinetic models for chemotaxis was performed in \cite{chalub_kinetic_2004,hwang_global_2005,hwang_global_2006,bournaveas_global_2008,bournaveas_global_2008-1} from the perspective of global existence and regularity of solutions. 
Numerical analysis of kinetic models for chemotaxis was performed in \cite{filbet_numerical_2014,Gosse_chemo,Gosse_Vauchelet}. In \cite{Yasuda}, the author proposed a Monte Carlo algorithm to simulate \eqref{eq:model} with the aim to resolve travelling waves. 
%Mesoscopic models describing the travelling wave phenomenon revealed the directional distribution of individuals, and in particular the spatially dependent biases in the trajectories. The above model \eqref{eq:model} has been validated with the experiment in which cells (approx. $10^5$ bacteria {\em E. coli}) were initially located on the left side of a micro-channel $500\mu m\times 100 \mu m\times 2cm$. After short time, a significant fraction of the population started to move towards the right side of the channel, at constant speed, within a constant profile. 

The constructions of travelling waves for system \eqref{eq:model} was investigated in \cite{calvez_existence}. It was proved that travelling wave solutions exist under certain conditions on the parameters. Furthermore, some careful analysis revealed that such waves are not unique in general, contrary to the macroscopic model obtained in the diffusion limit.

The main objectives of the present work are twofold: (i) We propose an efficient numerical scheme to approximate system \eqref{eq:model}, which is able to capture the waves over long period of time, despite their lack of regularity (ii) We explore some cases where several travelling waves co-exist, and investigate their stability, from a numerical perspective.

%%%%%%%%%%%%%%%%%%%%%%%%%%%%%%%%%%%%%%%%%%%%%%%%%%%%%%%%%%%%%%%%%%%%%%%%%%%%%%%%%%%%%%%%%%%
\subsection{Numerical simulations of travelling waves}
%%%%%%%%%%%%%%%%%%%%%%%%%%%%%%%%%%%%%%%%%%%%%%%%%%%%%%%%%%%%%%%%%%%%%%%%%%%%%%%%%%%%%%%%%%%
%\textcolor{red}{TODO} Introduction to numerical approximation\newline

Numerical approximation of problem \eqref{eq:model} is delicate because it requires an algorithm able to accurately reproduce attraction toward waves of constant velocity on large domains, along with reliable large-time behavior, non-linear coupling, and material derivatives representing real biological behavior. To be more precise, consider  the following points: 
%\textcolor{red}{TODO}(add other problems). 
\begin{itemize}
\item In order to preserve shape and speed of a travelling wave over large domains a numerical scheme has to balance correctly the transport and the tumbling operator. We propose a well-balanced approximation of the kinetic equation in the framework of scattering matrices. Their construction is based on the generalized Case's solutions for the stationary problem of \eqref{kinetic} which allows to preserve constant velocity profiles. The well-balanced technique reduces also the time-growth of numerical errors \cite{Amadori_Gosse}, which is extremely important due to the time scales of the problem.  
\item The coupling between the density $f$ and the concentrations $(M,N)$ plays a crucial role in maintaining the right direction of the propagation. Well-balanced discretizations for linear diffusive equations including lower-order terms were recently introduced in \cite{Gosse_Lsplines} and we apply these techniques to the parabolic part of the model \eqref{eq:model} for a better resolution of the time evolution of the concentrations $\chS,\chN$. However, this method gives also a more accurate coupling with the kinetic equation through the tumbling operator, because the space grids of the kinetic part and of the parabolic part are naturally tilted in the appropriate way. 
%[More precisely, the tumbling operator in the kinetic equation is defined at interfaces and the key point lies in the shift on the space grid between the reaction-diffusion system and the kinetic equation. These schemes contain it automatically, without additional approximations.]
\item  Material derivatives in the tumbling operator  account for the temporal variation of the concentrations along bacteria trajectories. Our first naive trial was not coherent with the underlying process. We realized that a basic upwind of this transport operator behaves in a nice way. 
%The simplest method consist in using its definition, then splitting the approximation in space and time. However, this approach isn't coherent with the true behavior of bacteria. A more natural method is based on upwinding with respect to the velocity. 
\end{itemize}
The proposed numerical scheme is compared with more classical time-splitting techniques. In particular, the resolution of the velocity profile and computation of the wave speed is verified and the advantage of the well-balanced approach is shown.

A source of global error in our simulations turns out to be the dissociation between kinetic and parabolic %discretizations. 
time steps.
This proceeds by stipulating that either material derivatives in \eqref{kinetic}, or the macroscopic density $\rho$ in \eqref{para-S-N}, are kept constant during each time step. Such a splitting assumption appears to be reasonable because \eqref{eq:model} is only a weakly nonlinear system.
\begin{rmk}\label{rem-zero-TW}
In spite of the weakness of the mean-field coupling, and the fact that strictly parabolic equations like (\ref{para-S-N}) are likely to react ``slowly'' to perturbations of $f$, the nature of the waves that we aim at capturing numerically puts such a ``splitting strategy'' in jeopardy. Indeed, chemotactic waves travelling at constant velocity result from subtle balancing involving all the equations in system (\ref{eq:model})-(\ref{tumbling})-(\ref{para-S-N}). Yet, as soon as a numerical algorithm proceeds by solving (\ref{eq:model})-(\ref{tumbling}) independently of (\ref{para-S-N}), the resulting kinetic equation only perceives $x$-dependent coefficients: in a bounded domain, such an equation does not admit travelling waves solutions, except the one with zero velocity.
\end{rmk}

%%%%%%%%%%%%%%%%%%%%%%%%%%%%%%%%%%%%%%%%%%%%%%%%%%%%%%%%%%%%%%%%%%%%%%%%%%%%%%%%%%%%%%%%%%%
\subsection{Organization of the paper}
%%%%%%%%%%%%%%%%%%%%%%%%%%%%%%%%%%%%%%%%%%%%%%%%%%%%%%%%%%%%%%%%%%%%%%%%%%%%%%%%%%%%%%%%%%%

This paper is organized as follows: Section~\ref{sec:theory} contains a description of theoretical results about the existence of such travelling waves. Then, Section~\ref{sec:scheme} is devoted to a detailed description of the components involved in our numerical approximation of model (\ref{eq:model})-(\ref{tumbling})-(\ref{para-S-N}). In particular, the treatment of (\ref{eq:model}) proceeds by applying techniques relying on Case's elementary solutions (see section 3.2), and the one handling parabolic equations (\ref{para-S-N}) relies on $\mathcal L$-splines discretization (see section 3.3). Accordingly, sections~\ref{sec:sim1} and \ref{sec:sim2} display numerical results of increasing complexity.

%%%%%%%%%%%%%%%%%%%%%%%%%%%%%%%%%%%%%%%%%%%%%%%%%%%%%%%%%%%%%%%%%%%%%%%%%%%%%%%%%%%%%%%%%%%
%%%%%%%%%%%%%%%%%%%%%%%%%%%%%%%%%%%%%%%%%%%%%%%%%%%%%%%%%%%%%%%%%%%%%%%%%%%%%%%%%%%%%%%%%%%
\section{Existence theory of chemotactic solitary waves}\label{sec:theory}
%%%%%%%%%%%%%%%%%%%%%%%%%%%%%%%%%%%%%%%%%%%%%%%%%%%%%%%%%%%%%%%%%%%%%%%%%%%%%%%%%%%%%%%%%%%
%%%%%%%%%%%%%%%%%%%%%%%%%%%%%%%%%%%%%%%%%%%%%%%%%%%%%%%%%%%%%%%%%%%%%%%%%%%%%%%%%%%%%%%%%%%

We summarize below the result obtained in \cite{calvez_existence} concerning the existence of travelling waves. Firstly, the problem is recast in the moving frame variable $z = x-ct$, where $c$ denotes the wave speed, which is the main unknown in our problem. It writes
\begin{equation}
\begin{cases}
  \displaystyle   (v-c) \partial_z f(z,v) =  \int  \calT(z,v'-c)f(z,v') \, d\nu(v') - \calT(z,v-c)f(z,v)
\medskip\\
 - c\partial_z \chS(z)  -  D_{\chS} \partial_{zz} \chS(z)  + \alpha \chS(z) = \beta \rho(z) \medskip\\
- c\partial_z \chN(z)  - D_{\chN} \partial_{z}^2 \chN(z) + \gamma \rho(z) \chN(z) = 0
\end{cases},\label{eq:TW}
\end{equation}
where the tumbling rate $\calT$ can take only four possible values depending on the sign of the gradients,
\begin{eqnarray}
\calT(z,v'-c) &=&  1 - \chi_{\chS} \sign((v'-c) \partial_z \chS(z)) - \chi_{\chN}\sign((v'-c) \partial_z \chN(z)) \nonumber \\
&=&  1 \pm \chi_{\chN} \pm \chi_{\chN}\, . \label{eq:T}
\end{eqnarray}
Under some restriction on the reaction-diffusion parameters $(D_{\chS},\alpha)$, there exist $c$, and functions $(f,\chS,\chN)$ solutions of the travelling wave problem \eqref{eq:TW}, see the precise statement in Theorem \ref{theo:kin TW} below. The conditions which are imposed on the parameters to guarantee existence of a travelling wave solution are linked to the asymptotic behavior of the solution of the linear stationary problem
\begin{equation}\label{eq:decoupled}
\displaystyle   (v-c) \partial_z f(z,v) =  \int   T(z,v'-c)f(z,v') \, d\nu(v') - T(z,v-c)f(z,v),
\end{equation}
for a given $c$, in a given field of chemical concentrations $\chS(z),\chN(z)$ which satisfy the following sign rules:
\begin{equation}
\label{eq:ansatz}
\begin{cases}
 \text{$(\forall z<0)\;   \partial_z \chS(z)>0$, and $(\forall z>0)\; \partial_z \chS(z)<0$}\,, \medskip\\
\text{$(\forall z)\;  \partial_z \chN(z)>0$}.
\end{cases}
 \end{equation}
The tumbling rate $T$ associated with such given concentrations $\chS(z),\chN(z)$ can take only four possible values $ 1 \pm \chi_{\chS} \pm \chi_{\chN} $, according to the rule of signs depicted in Figure \ref{fig:sign convention}. We introduce the notation
%\begin{displaymath}
%\begin{cases}
$$
T_+ = 1 + (\chi_{\chS}-\chi_{\chN})\sign(v),\qquad
T_+ = 1 - (\chi_{\chS}+\chi_{\chN})\sign(v).
$$
%\end{cases}
%\end{displaymath}
Under mild conditions on the measure $\nu$ (essentially bounded below by a positive constant on its support),  the density $f(z,v)$ decays exponentially fast on both sides of the origin $z=0$. There exist positive exponents $\lambda_-, \lambda_+$, and velocity distributions $F_-, F_+$ such that
\begin{equation}\label{eq:asympt}
f(z,v) \underset{z\to -\infty}\sim e^{\lambda_-  z } F_-(v)\, , \quad f(z,v) \underset{z\to +\infty}\sim e^{-\lambda_+  z } F_+(v)\, .
\end{equation}
The meaning of the equivalence in \eqref{eq:asympt} is made precise in \cite{calvez_confinement,calvez_existence}, in terms of some $L^2$ weighted space.
The pairs $(\lambda_-,F_-)$ and $(\lambda_+,F_+)$ are given by the expressions,
\begin{multline}\label{eq:lambda-}
\bullet \, F_-(v) = \dfrac{1}{T_-(v-c) + \lambda_-  (v-c)}
\, , \\ \text{where $\lambda_-$ is the smallest positive root of }
\int \dfrac{v-c}{T_-(v) + \lambda (v-c)}\, dv  = 0\, ,
\end{multline}
and
\begin{multline}\label{eq:lambda+}
\bullet \, F_+(v) = \dfrac{1}{T_+(v-c) - \lambda_+  (v-c)}\, , \\ \text{where $\lambda_+ $ is the smallest positive root of } \int \dfrac{v-c}{T_+(v-c) - \lambda(v-c)}\, dv = 0 \, .
\end{multline}
On the one hand, the existence of a positive root $\lambda_+$ in the latter equation \eqref{eq:lambda+} is guaranteed if $c$ is larger than $c_*$, where $c_*$ is defined as the unique velocity such that
\begin{equation}
\int \dfrac{v-c_*}{T_+(v-c_*)}\, dv = 0\, .
\end{equation}
On the other hand, the existence of a positive root $\lambda_-$ in equation \eqref{eq:lambda-} is guaranteed provided $c$ is less than $c^*$, where $c^*$ is defined as the unique velocity such that
\begin{equation}
\int \dfrac{v-c^*}{T_-(v-c^*)}\, dv = 0\, .
\end{equation}
\begin{figure}[t]
\begin{center}
\includegraphics[width = .6\linewidth]{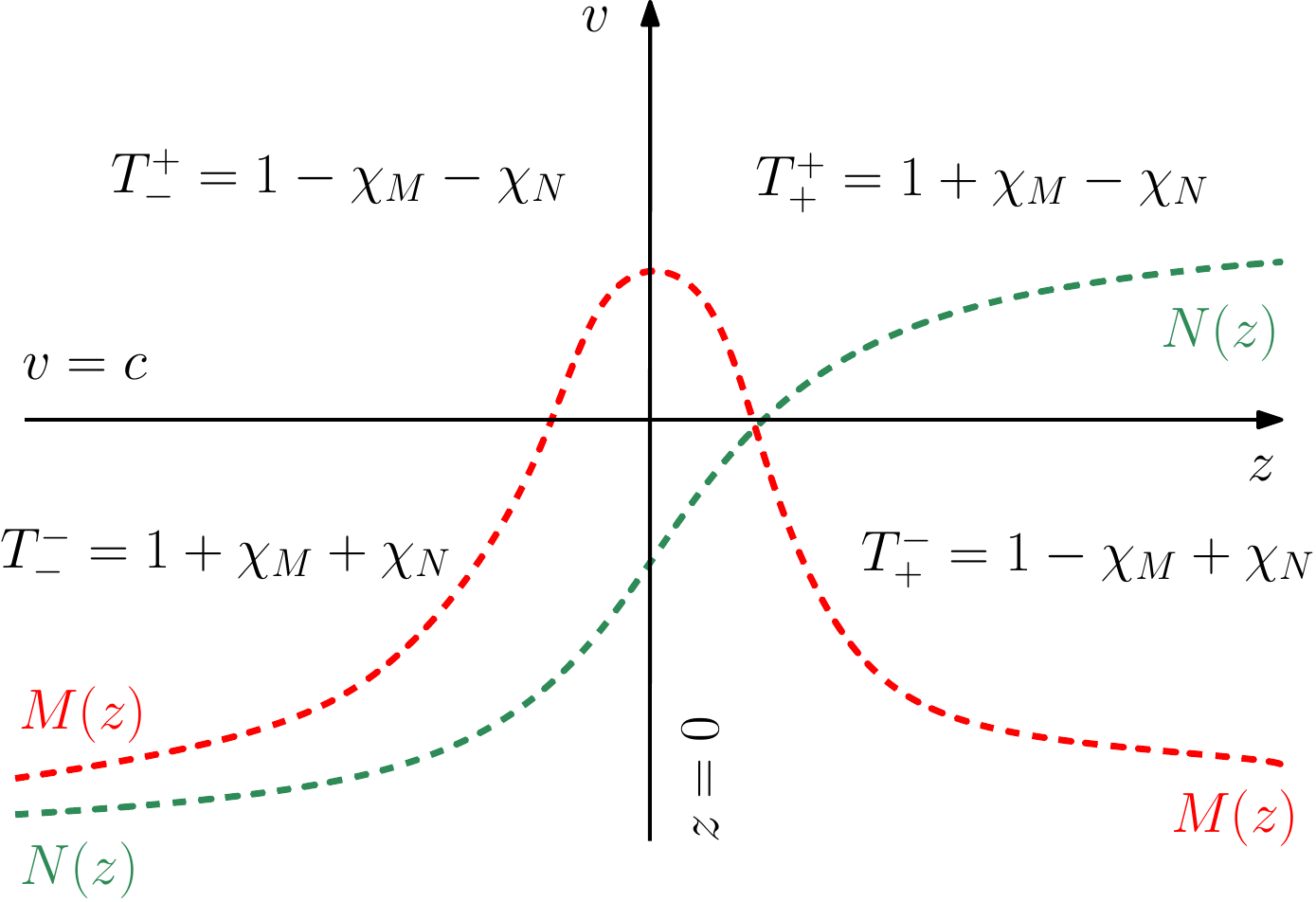}
\caption{Expression of the tumbling rate $T_\pm^\pm$ \eqref{eq:T} depending on the signs of $z$ and $v-c$. The expected profiles of $\chS$ and $\chN$ are plotted in dashed lines, in order to get the correct value of $T_\pm^\pm$ at a glance. Note the dependency with respect to $c$, as some signs change relatively to $v-c$ from bottom to top in the $(z,v)$ plane. }
\label{fig:sign convention}
\end{center}
\end{figure}
We now state precisely the conditions for existence of travelling wave solutions.
\begin{theorem}[\cite{calvez_existence}]
Assume $(\chi_{\chS}, \chi_{\chN})\in (0,1/2)\times[0,1/2)$. Assume that $\nu$ is absolutely continuous with respect to Lebesgue's measure: $d\nu(v) = \omega(v)dv$, where the p.d.f. $\omega$ belongs to $L^p$ for some $p>1$. Assume in addition that the reaction-diffusion parameters $\alpha, D_{\chS}$ satisfy the following conditions:
\begin{subequations}\label{eq:parameter conditions}
\begin{equation} \label{eq:cond 1} \dfrac{c_* + \sqrt{(c_*)^2 + 4\alpha D_{\chS}}}{c_* + \sqrt{(c_*)^2 + 4\alpha D_{\chS}} + 2D_{\chS} \lambda_-(c_*)} \leq \text{explicit constant}\,, \end{equation}
\begin{equation}  \label{eq:cond 2}
\dfrac{-c^* + \sqrt{(c^*)^2 + 4\alpha D_{\chS}}}{-c^* + \sqrt{(c^*)^2 + 4\alpha D_{\chS}} + 2D_{\chS} \lambda_+(c^*) }
\leq \text{explicit constant}\, ,
\end{equation}
\begin{equation}  \label{eq:cond 3} \text{Either $\chi_{\chN}\geq \chi_{\chS}$, or}\quad \dfrac{ \sqrt{\alpha/ D_{\chS}} +   \lambda_+(0)}{\sqrt{\alpha / D_{\chS}} +  \lambda_-(0)} \leq \text{explicit constant}\, . \end{equation}
\end{subequations}
Then, there exist a non-negative velocity $c$, and a set of nonnegative functions, 
$$
(f,\chS,\chN)\in \left( L^1\cap L^\infty (\Re\times V)\right)\times \mathcal C^2(\Re)\times \mathcal C^2(\Re),
$$
being solution of the travelling wave problem \eqref{eq:TW}-\eqref{eq:T}.
\label{theo:kin TW}
\end{theorem}
\begin{proof} (Sketch) The strategy of proof is inspired from the macroscopic limit of \eqref{eq:TW} in the diffusive regime \cite{saragosti_mathematical_2010}. It is based on a fixed-point argument on the signal concentration $\chS$. It is assumed {\em a priori} that $\chN$ is increasing ($\partial_z \chN$ is positive), and that $\chS$ is unimodal ($\partial_z \chS$ changes sign only once), with a unique maximum located at $z = 0$, see \eqref{eq:ansatz}. The resulting tumbling rate is deduced according to the rule of signs in Figure \ref{fig:sign convention}. One of the main results contained in \cite{calvez_existence} states that the spatial density  $\rho(z) = \int f(z,v)\, d\nu(v)$, obtained from \eqref{eq:decoupled}, is unimodal too. Furthermore, it reaches its maximum point at $z=0$, as for $\chS$. As a by-product, $\widetilde \chS$, defined as the solution the following reaction-diffusion (elliptic) equation,
\begin{equation}\label{eq:signal_elliptic}
 - c\partial_z \widetilde \chS (z)  -  D_{\chS} \partial_{zz} \widetilde \chS (z)  + \alpha \widetilde \chS (z) = \beta \rho(z)\, 
\end{equation}
is unimodal, which is consistent with the preliminary assumption \eqref{eq:ansatz}. However, its maximum point may not coincide with $z=0$, except if $c$ is chosen appropriately, see Figure \ref{fig:variation}. If $c$ is chosen such that  the maximum point of $\widetilde \chS$ is located at $z=0$, then $\widetilde \chS = \chS$ can be chosen consistently to solve the full coupled system \eqref{eq:TW}.
\begin{figure}[t]
\begin{center}
\includegraphics[width = .45\linewidth]{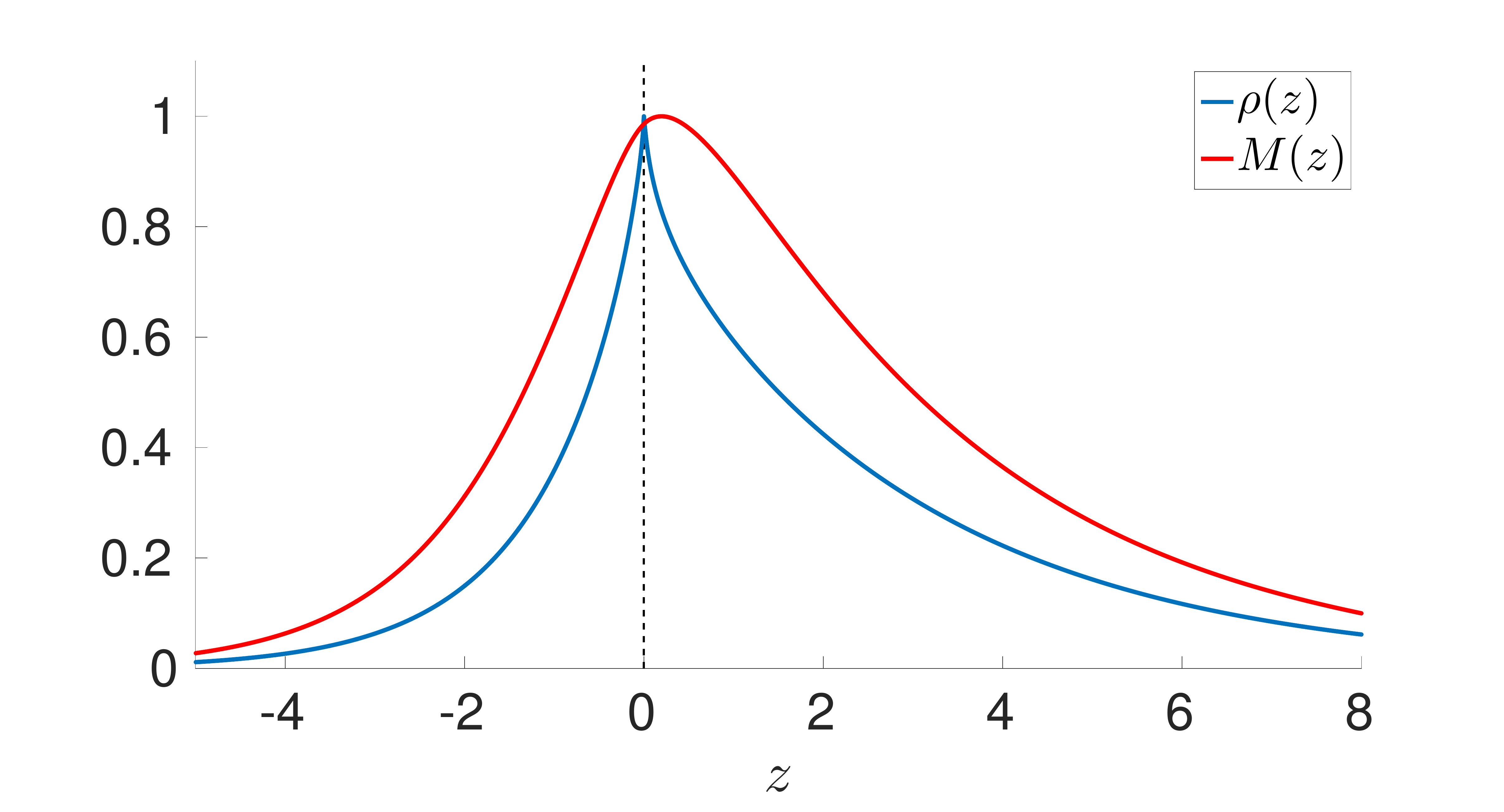}\;
\includegraphics[width = .45\linewidth]{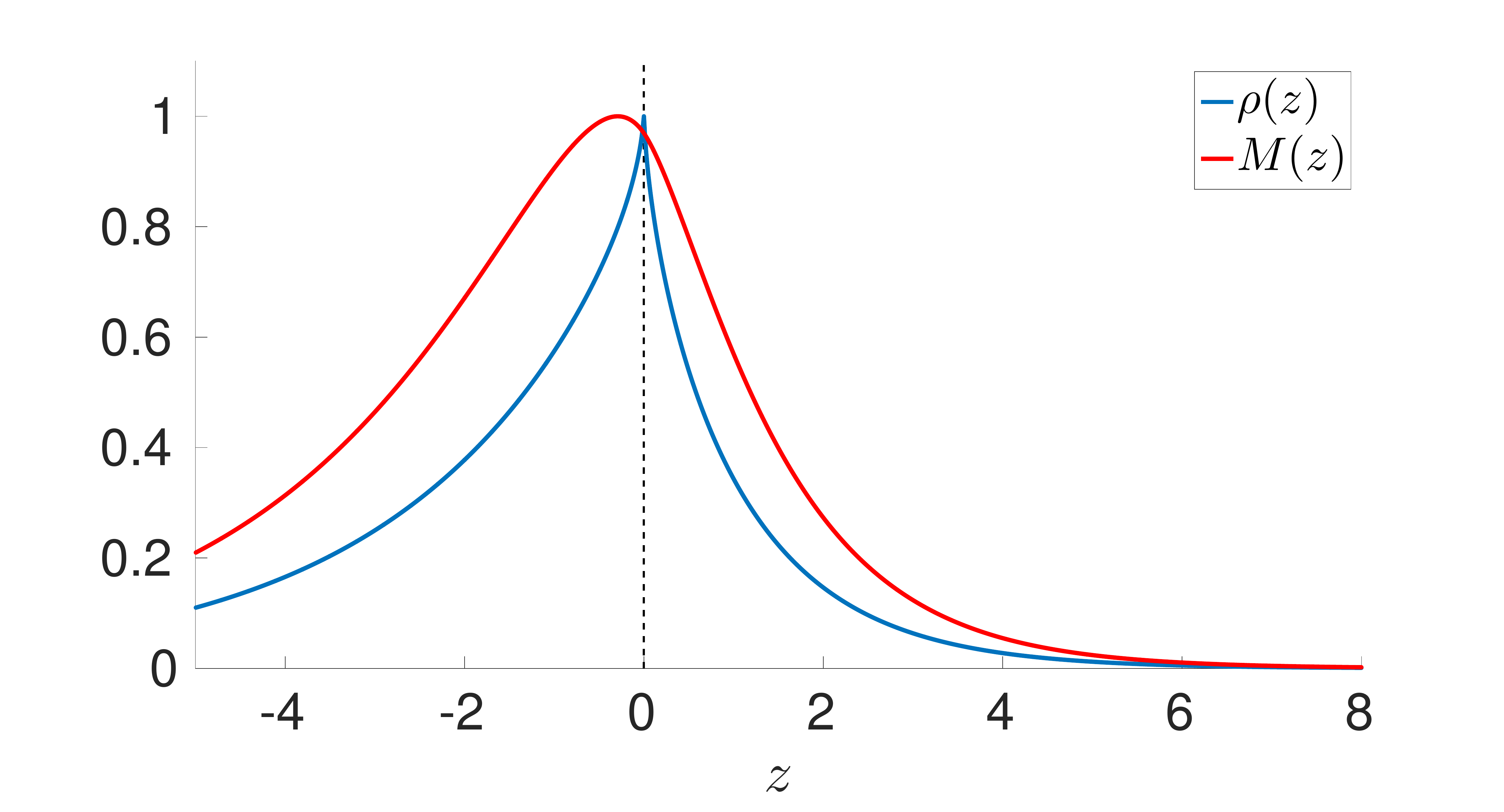}\\
\includegraphics[width = .6\linewidth]{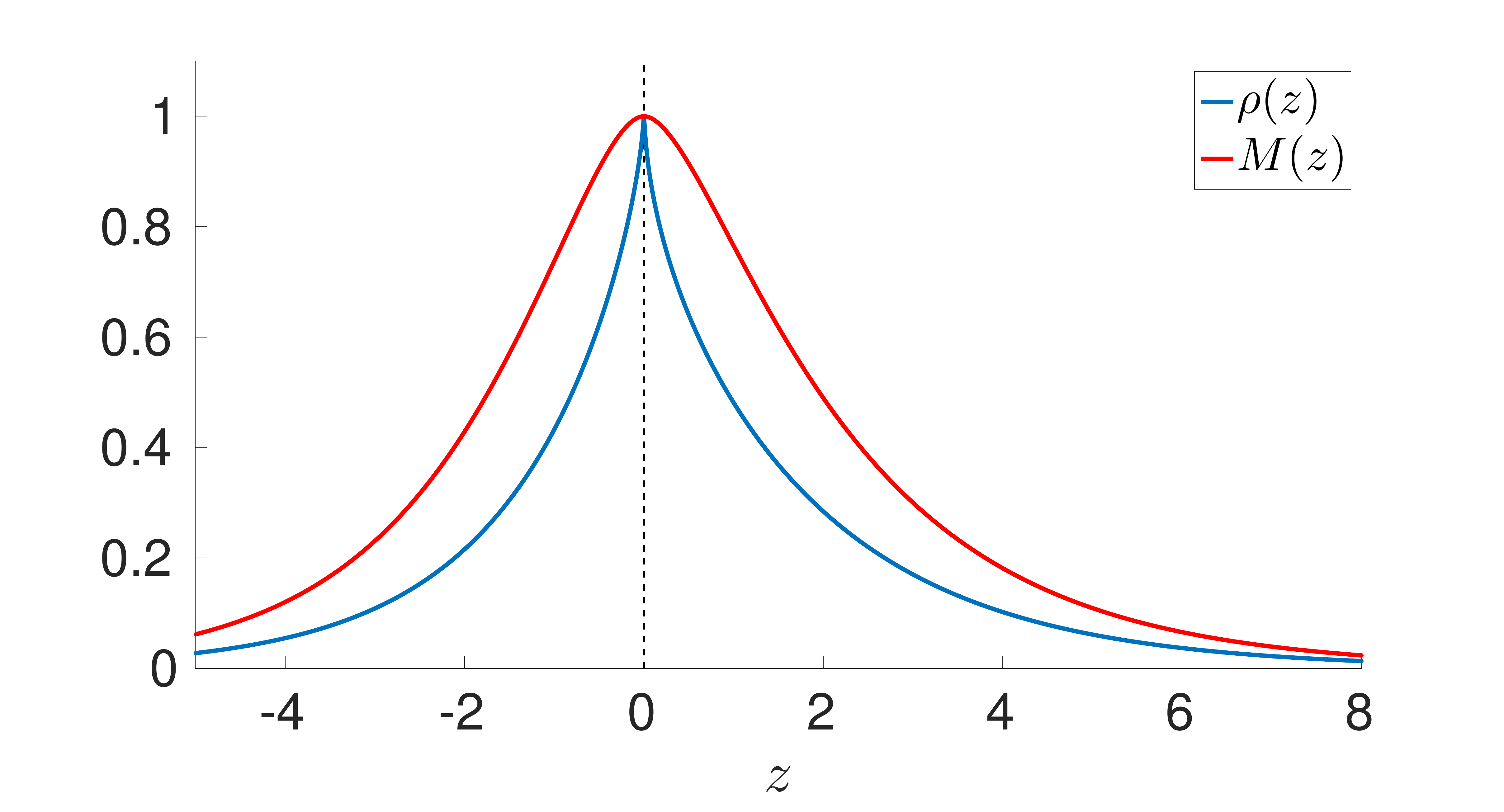}
\caption{Illustration of the solution $\widetilde\chS$ of the elliptic equation \eqref{eq:signal_elliptic}, where $\rho$ is deduced from the kinetic equation in \eqref{eq:TW}. (Left) Here, $c$ is chosen too small. As a result, the spatial density $\rho$ is tilted towards the right side. Accordingly, the maximum value of $\widetilde\chS$ is shifted to the right as well. (Right) The opposite conclusion holds if $c$ is chosen too large. (Bottom) Using a continuity argument, there must exist a value of $c$ for which the two maximum points coincide. However, one should use any kind of intuition with caution, as there is no monotonicity with respect to $c$ in general.}
\label{fig:variation}
\end{center}
\end{figure}
This motivates the following definition of an auxiliary function $c\mapsto \Upsilon(c)$, 
\begin{equation}\label{eq:gamma_C}
\boxed{
(c_*,c^*) \ni c \mapsto 
\Upsilon(c) = \partial_z \widetilde \chS(0).}
\end{equation}
Since $\widetilde \chS$ is unimodal, we have the following simple observation: if $\Upsilon(c)>0$, then the maximum of $\widetilde \chS$ is reached for $z>0$, whereas if $\Upsilon(c)<0$, then the maximum of $\widetilde \chS$ is reached for $z<0$.
These monotonicity considerations pave the way for a variational procedure:
\begin{enumerate}
\item  On the one hand, condition \eqref{eq:cond 1} guarantees that $\Upsilon(c)$ is positive as $c\searrow c_*$;
\item On the other hand, \eqref{eq:cond 2} guarantees that $\Upsilon(c)$ is negative as $c\nearrow c^*$;
\item Moreover, it is continuous\footnote{Continuity of the function $\Upsilon$ requires some regularity in the kinetic problem \eqref{eq:decoupled}. In particular, the velocity distribution is supposed to be absolutely continuous with Lebesgue's measure for that purpose. Continuity fails in the case of a discrete measure supported on a finite number of velocities, as in the numerical scheme discussed later in this article.} with respect to $c$;
\item Therefore, there exists $c\in (c_*,c^*)$ for which the maximum of $\widetilde \chS$ is located at $z=0$. The third condition \eqref{eq:cond 3} guarantees that $c$ can be chosen to be positive, which is a requirement for checking that $\chN$ is increasing afterwards.
\end{enumerate}
\end{proof}
It is a natural question to ask whether conditions \eqref{eq:parameter conditions} can be removed. Another issue is about uniqueness of the travelling wave. As a side result of the present numerical investigation, we found that there may exist several values of $c$ for which travelling wave exists, for some regions of the parameter set which violate \eqref{eq:parameter conditions}. This is in contradiction with the macroscopic limit of \eqref{eq:TW} in the diffusive regime, for which there exists a unique wave speed.
We presume that there  exist some parameter values for which no travelling exist,  see \cite{ours}.

Theorem \ref{theo:kin TW} is concerned with a continuum of velocities. This is a crucial assumption to ensure continuity of the auxiliary function $\Upsilon$. The case of a discrete number of velocities has been developed in \cite{ours}, where it is shown that the function $\Upsilon$ is well-defined.

%%%%%%%%%%%%%%%%%%%%%%%%%%%%%%%%%%%%%%%%%%%%%%%%%%%%%%%%%%%%%%%%%%%%%%%%%%%%%%%%%%%%%%%%%%%
%%%%%%%%%%%%%%%%%%%%%%%%%%%%%%%%%%%%%%%%%%%%%%%%%%%%%%%%%%%%%%%%%%%%%%%%%%%%%%%%%%%%%%%%%%%
\section{Approximation of the weakly nonlinear system}\label{sec:scheme}
%%%%%%%%%%%%%%%%%%%%%%%%%%%%%%%%%%%%%%%%%%%%%%%%%%%%%%%%%%%%%%%%%%%%%%%%%%%%%%%%%%%%%%%%%%%
%%%%%%%%%%%%%%%%%%%%%%%%%%%%%%%%%%%%%%%%%%%%%%%%%%%%%%%%%%%%%%%%%%%%%%%%%%%%%%%%%%%%%%%%%%%

Hereafter, an efficient numerical strategy for model \eqref{eq:model} is presented so as to reduce as much as possible the time-growth of accumulating errors. For such purposes, well-balanced (WB) schemes were proved convenient, see \cite{Amadori_Gosse}. A uniform Cartesian grid in the $(t,x)$-variables is defined through grid parameters $\DX, \DT>0$ so that $x_j =j \DX$ for convenient $j \in \Ze$, $t^n=n\DT$, $n \in \Na$. By convention, our control cells are $C_j=(x_{j-\frac 1 2},x_{j+\frac 1 2})$. We shall also consider staggered cells $(x_{j-1},x_j)$.

%%%%%%%%%%%%%%%%%%%%%%%%%%%%%%%%%%%%%%%%%%%%%%%%%%%%%%%%%%%%%%%%%%%%%%%%%%%%%%%%%%%%%%%%%%%
\subsection{Review of kinetic well-balanced schemes}
%%%%%%%%%%%%%%%%%%%%%%%%%%%%%%%%%%%%%%%%%%%%%%%%%%%%%%%%%%%%%%%%%%%%%%%%%%%%%%%%%%%%%%%%%%%

Equations \eqref{eq:model} constitute a weakly nonlinear system. Especially, parabolic equations (\ref{para-S-N}) are expected to react slowly to (macroscopic) density fluctuations. Thus, it makes sense to stipulate that either material derivatives $\frac{D\chS}{Dt}, \frac{D\chN}{Dt}$ in (\ref{kinetic}), or the macroscopic density $\rho$ in (\ref{para-S-N}), are constant during a time-step $\DT>0$.
\begin{figure}[t]
\begin{center}
\includegraphics[width = .52\linewidth]{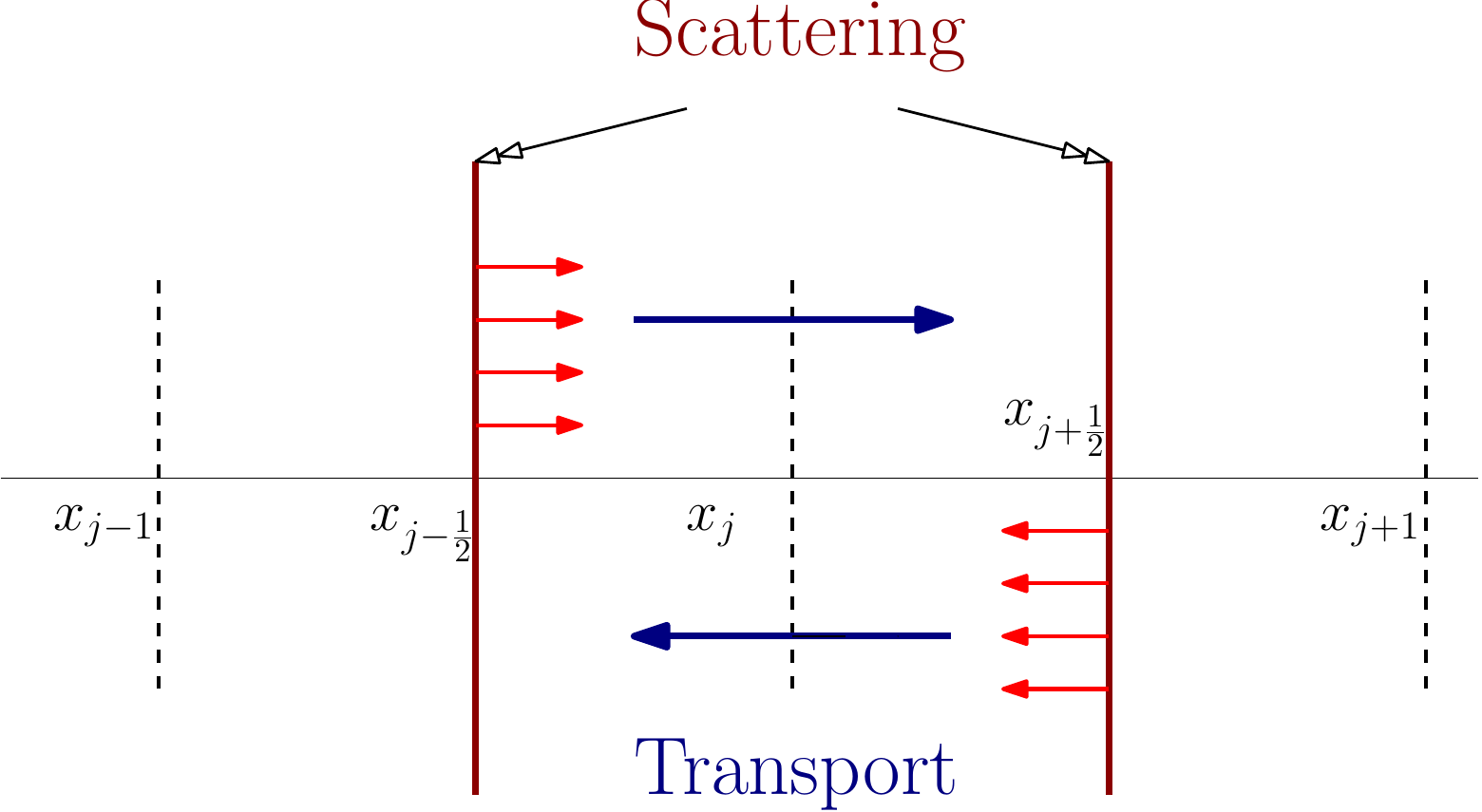}\quad
\includegraphics[width = .38\linewidth]{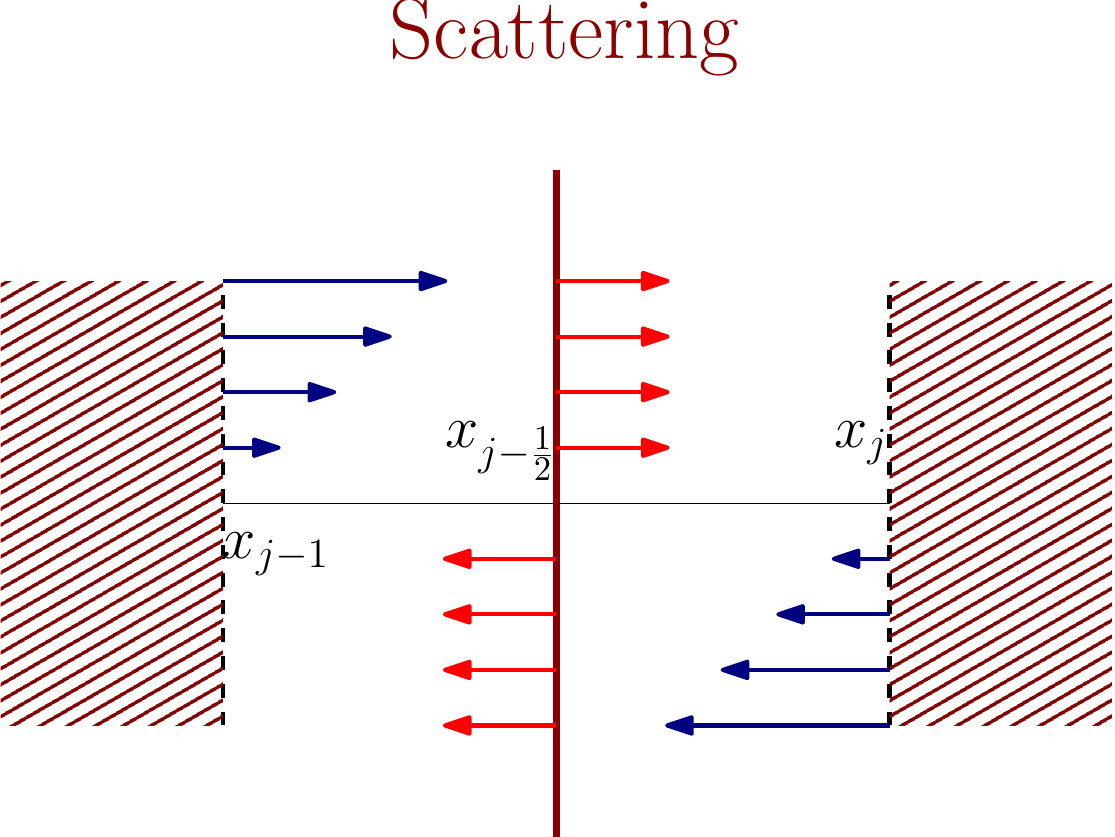}
\caption{Localization of the scattering events (tumbling) in well-balanced method for models like (\ref{kinetic}).  (Left) According to \eqref{eq:transport}, transport is solved in the cell $C_j$ using the approximated values of $f$ at the interfaces. (Right) The value at each interface is reconstructed by solving the stationary boundary-value problem \eqref{eq:scatt} in each staggered cell.}
\label{fig:scheme}
\end{center}
\end{figure}
Well-balanced schemes for $(1+1)$-dimensional linear collisional models were thoroughly presented in \cite[Part II]{Gosse_book}. Roughly speaking, a WB treatment of models like (\ref{kinetic}) consists in (formally) concentrating the scattering events at fixed locations of the computational grid (see Figure \ref{fig:scheme}),
\begin{equation}
\Dt f + v \Dx f = \DX \sum_{j \in \Ze} \left( \int_V \calT(t,x,v')f(t,x,v') d\nu(v') - \calT(t,x,v)f(t,x,v) \right) \delta(x-x_{j-\frac 1 2}).\label{eq:Dirac}
\end{equation}
Applying the nowadays standard Godunov procedure to the former (and more singular) equation, a time-marching numerical scheme is derived. Scattering events are rendered through supplementary jump relations at each $x_{j-\frac{1}{2}}$: according to \cite{Gosse_red},
\begin{equation}\label{eq:transport}
f_j^{n+1}(v )  = 
\begin{cases} 
f^n_j(v )-\dfrac{\DT}{\DX}v \left(f^n_j( v  )-\tf^n_{j-\frac 1 2}( v  )\right),  \quad&\text{if $v >0$} \medskip\\
f^n_j(v ) - \dfrac{\DT}{\DX}v  \left(\tf^n_{j+\frac 1 2}(v )-f^n_j(v )\right), \quad&\text{if $v <0$}
\end{cases},
\end{equation}
where shorthand notation was used: $f^n_j(v ) \simeq f(t^n,x_j,v )$, and $\tf^n_{j-\frac 1 2}( v  )$ (resp. $\tf^n_{j+\frac 1 2}( v )$) denote the approximated value of $f$ at the interface $x_{j-\frac 12}$ (resp. $x_{j+\frac 12}$). The latter approximations take into account the scattering operator, as explained below.
The presence of ``Dirac collision terms'' in \eqref{eq:Dirac} induces a static discontinuity at each interface separating two control cells $C_{j-1}$ and $C_j$. Hence, the approximated value $\tf^n_{j\pm\frac 1 2}$ which appears in \eqref{eq:transport}. The way to relate these interface states %values of $(\tf^n_{j\pm\frac 1 2})$ at interface,
and the values of $f_{j}$ at the center of the control cells is done through a scattering matrix. Conceptually, well balanced schemes are constructed to be exact on equilibrium states. This motivates to compute either analytically, or numerically the following (forward/backward) boundary-value problem (BVP) on each staggered cell $(x_{j-1},x_j)$ (see Figure \ref{fig:scheme}),
\begin{subequations}\label{eq:BVP}
\begin{equation}\label{eq:scatt}
v \Dx \tfe =\int_V \calT^n_{j-\frac 1 2}(v')\tfe(x,v') d\nu(v')- \calT^n_{j-\frac 1 2}(v)\tfe(x,v), \qquad x \in (x_{j-1},x_j),
\end{equation}
where $\calT^n_{j-\frac 1 2}(v)$ is ``frozen'' in space and time. Such a quantity stands for a reliable approximation of the tumbling mechanism at each interface $x_{j-\frac 1 2}$, as described in Section~\ref{sec:tumbling}. This BVP is complemented with inflow data 
\begin{equation}\label{eq:inflow data}
\begin{cases}(\forall v>0)\;& \tfe(x_{j-1},v) = f^n_{j-1}(v),  \medskip\\
(\forall v<0)\;& \tfe(x_{j},v) = f^n_{j}(v). 
\end{cases}
\end{equation}
\end{subequations}

\begin{rmk}
The well-balanced scheme using the equilibrium equation \eqref{eq:scatt} to derive the interface states is exact on stationary solutions, that is waves with speed $c=0$. So, if $c \not =0$, it is not endowed with an exact balance between transport and collision operators, as already explained in Remark \ref{rem-zero-TW}. Moreover, the wave speed $c$ is not known a priori and the steady equation \eqref{eq:decoupled} isn't easy to be used for deriving a scheme which remains consistent with (\ref{eq:model}). 
\end{rmk}  

To deal with the integral contribution in \eqref{eq:scatt}, a numerical quadrature in the $v$-variable is characterized by nodes and weights, which are symmetrical with respect to zero: 
\begin{equation}\label{quad}
\V=\{v_{-\nG},...,v_{-1},v_{1},..., v_{\nG}\} \in (-1,1)^{2\nG}, \quad \O=\{\omega_{-\nG},..., \omega_{-1},\omega_{1}, ..., \omega_{\nG}\}\in \Re^{2\nG}_+,
\end{equation}
such that $v_{k}\neq 0$, $v_{k}<v_{k+1}$ for all $k$. We denote by $\ksum = [-\nG,\nG]\setminus\{0\}$ the set of indices. We assume that the weights are normalized, such that $\sum_{\ksum}\omega_k = 1$.  For notational purposes, we also define $\calV \in (\Re_+^*)^{\nG}$  such that $\V= (-\calV)\cup \calV $, and $\calW=[\calV, \calV] \in \Re^{2\nG}$. The integral term in  \eqref{eq:scatt} is approximated by,
\begin{equation}\label{eq:quadra}
\sum_\ksum\omega_k T^n_{j-\frac12}(v_k)\tfe(x,v_k).
\end{equation}
To complete the scheme \eqref{eq:transport}, we seek the linear transformation that relates the outgoing states $\tf(x_{j-\frac12},v)$ with the incoming states \eqref{eq:inflow data} through the BVP \eqref{eq:BVP}. In the discrete velocity setting \eqref{quad}, it is given by a $2\nG\times 2\nG$ $S$-matrix which is denoted by $S^n_{j-\frac 1 2}$. In short, the BVP \eqref{eq:BVP} boils down to the following linear relation,
\begin{equation}\label{eq:outgoing_incoming}
\begin{pmatrix}
\tf(x_{j-\frac12},\calV) \\
\tf(x_{j-\frac12}, - \calV)  
\end{pmatrix} = S^n_{j-\frac 1 2}
\left( \begin{array}{c}f_{j-1}^{n}(\calV) \\ f_{j}^{n}(-\calV) \end{array}\right).
\end{equation}
Plugging this relation into the transport part on each cell \eqref{eq:transport}, after some shift in the indices, we arrive at our numerical scheme written in a concise way,
\BE\label{god}
\left( \begin{array}{c}f_j^{n+1}(\calV) \\ f_{j-1}^{n+1}(-\calV) \end{array}\right)=
\left(1-|\calW| \frac{\DT}{\DX}\right)
\left( \begin{array}{c}f_j^{n}(\calV) \\ f_{j-1}^{n}(-\calV) \end{array}\right)
+|\calW| \frac{\DT}{\DX} S^n_{j-\frac 1 2}
\left( \begin{array}{c}f_{j-1}^{n}(\calV) \\ f_{j}^{n}(-\calV) \end{array}\right).%}
\EE

Define the diagonal matrix, 
\begin{displaymath}
\Gamma=\mbox{diag}(\omega_k|v_k|), \qquad k=1, ..., 2\nG,
\end{displaymath} 
the next Lemma states basic properties of the scheme  (\ref{god}).
\begin{lemma}
For any nonnegative initial data $0\leq f^0(x,v) \in L^1(\Re \times V)$, $V$ standing for $(-1,1)$, the scheme (\ref{god}) preserves both non-negativity and total mass (hence a uniform $L^1$ bound) of $f^n_j(\pm \calV)$ as soon as 
\begin{eqnarray}\label{eq:cfl}
\max(\V)\cdot \DT \leq \DX,&\qquad \mbox{(CFL condition),} \\
\label{eq:stoch}\forall j, \quad 
\Gamma\, S_{j-1/2}^n \, \Gamma^{-1}&\, \mbox{  is left-stochastic.}
\end{eqnarray} 
\end{lemma}
\begin{proof}
Under condition (\ref{eq:cfl}), (\ref{god}) realizes a nonnegative combination of its inputs as soon as the entries of $S_{j-1/2}^n$ are nonnegative. Yet, assume moreover that (\ref{eq:stoch}) holds, we  compare with the standard upwind scheme for free transport,
\begin{displaymath}
\forall n\geq 0, \qquad \DX \sum_{i \in \Ze}\rho_i^{n+1}=
\DX \sum_{j \in \Ze, k \in \{1, ... \nG\}}  \omega_k (f^{n+1}_j(v_k)+f^{n+1}_{j-1}(-v_k)).
\end{displaymath}
Yet, by requiring (\ref{eq:stoch}), one secures that,
\begin{displaymath}
\begin{pmatrix}
\omega_k |v_k|\, \tf_{j-\frac12}(v_k)_{k=1, ..., \nG} \\
\omega_k |v_k|\, \tf_{j-\frac12}(-v_k)_{k=1, ..., \nG} \\
\end{pmatrix} = \Gamma\, S_{j-1/2}^n \, \Gamma^{-1}
\left( \begin{array}{c}
\omega_k |v_k|\, f_{j-1}^{n}(v_k)_{k=1, ..., \nG} \\ 
\omega_k |v_k|\, f_{j}^{n}(-v_k)_{k=1, ..., \nG} 
\end{array}\right),
\end{displaymath}
which brings the following (current-preservation) property,
\begin{displaymath}
\sum_{k \in \{1, ... \nG \}}\omega_k |v_k| (\tilde f_{j-\frac 1 2}(v_k)+ \tilde f_{j-\frac 1 2}(-v_k))
= \sum_{k \in \{1, ... \nG \}}\omega_k |v_k|(f_{j-1}(v_k)+ f_{j}(-v_k)),
\end{displaymath}
so that, from the expression (\ref{god}), it follows that
\begin{eqnarray*}
\DX \sum_{j \in \Ze}\rho_j^{n+1}&=&\sum_{j \in \Ze, k \in \{1, ... \nG \}}  \omega_k (\DX-v_k\DT)(f^{n}_j(v_k)+f^{n}_{j-1}(-v_k)) \\ 
&&\qquad +\ \DT \sum_{j \in \Ze, k \in \{1, ... \nG \}}  \omega_k\ v_k (\tilde f_{j-\frac 1 2}(v_k)+ \tilde f_{j-\frac 1 2}(-v_k)) \\
&=& \sum_{j \in \Ze, k \in \{1, ... \nG \}}  \omega_k (\DX-v_k\DT)(f^{n}_j(v_k)+f^{n}_{j-1}(-v_k)) \\ &&\qquad +\ \DT \sum_{j \in \Ze, k \in \{1, ... \nG \}}  \omega_k\ v_k (f_{j-1}(v_k)+ f_{j}(-v_k)) \\
&=&\DX \sum_{j \in \Ze}\rho_j^{n},
\end{eqnarray*}
\end{proof}

This Lemma is somehow the complementary of \cite[Proposition 1]{Gosse_red}; it furnishes an easy-to-check (sufficient) condition for both positivity- and mass-preservation for a scheme written in the form (\ref{god}). However, except for elementary two-stream models, like the ones studied in \cite{Gosse_Vauchelet}, it is usually difficult to check the ``left-stochastic'' property in practice.
In the next section, we propose two choices for deriving $S$-matrices $S^n_{j-\frac 1 2}$.

%%%%%%%%%%%%%%%%%%%%%%%%%%%%%%%%%%%%%%%%%%%%%%%%%%%%%%%%%%%%%%%%%%%%%%%%%%%%%%%%%%%%%%%%%%%
\subsection{Derivation of two  different $S$-matrices}\label{sec:scattering_matrix}
%%%%%%%%%%%%%%%%%%%%%%%%%%%%%%%%%%%%%%%%%%%%%%%%%%%%%%%%%%%%%%%%%%%%%%%%%%%%%%%%%%%%%%%%%%%

Our first option builds on the original analysis of ``Case's elementary solutions'' devoted to the accurate description of solutions to the stationary BVP \eqref{eq:BVP}, which involves exponentially damped modes. A second option  draws onto second-order finite-differences, see \cite[Chap. 10]{Gosse_chemo}.
\subsubsection{Case's elementary solutions}\label{sec:S_case}
Following \cite{Aamodt_Case,Barichello_ADO1999,c1}, solutions of \eqref{eq:BVP} are sought as a (finite) combination of Case's elementary modes with separated variables: 
\begin{equation}\label{eq:case}
G(x,v) = \exp(-\lambda x)\vp(v).
\end{equation} 
Plugging $G$ into \eqref{eq:BVP}, approximated with the velocity quadrature \eqref{eq:quadra}, we get an equation for the pair $(\lambda, \varphi)$:
\begin{equation}
(\forall k)\quad \left (T^n_{j-\frac 1 2}(v_k)-\lambda  v_k \right ) \vp(v_k) = \sum_\lsum\omega_\ell T^n_{j-\frac12}(v_\ell)\vp(v_\ell).  
\end{equation}
Thus, each ``constant of separation'' $\lambda$ is a solution to the following equation: 
\begin{equation}
\sum_\ksum\omega_k \dfrac{T^n_{j-\frac12}(v_k)}{T^n_{j-\frac 1 2}(v_k)-\lambda  v_k } = 1.
\end{equation}
Clearly, $\lambda = 0$ is a special solution, which results from mass conservation. The corresponding eigenvector $\varphi$ is:
\begin{equation}
\overline{\varphi}(v) = \frac{1}{T^n_{j-\frac12}(v)}.
\end{equation} 
Other solutions are given by the equivalent relation
\begin{equation}
\sum_\ksum\omega_k \left( \frac{T^n_{j-\frac 1 2}(v_k)}{v_k}-\lambda\right )^{-1} = 0.
\end{equation}
By studying the variations of the left-hand-side with respect to $\lambda$, one deduces the existence of exactly $2\nG-1$ distinct solutions which are interlaced, like
\begin{multline}\label{eq:entrelacing}
\frac{T^n_{j-\frac 1 2}(v_{-1})}{v_{-1}} < \lambda_{-\nG+1} < \frac{T^n_{j-\frac 1 2}(v_{-1})}{v_{-2}} <  \lambda_{-\nG+2}  < (\dots) < \frac{T^n_{j-\frac 1 2}(v_{-\nG})}{v_{-\nG}} < \lambda_0 \\ <   \frac{T^n_{j-\frac 1 2}(v_{\nG})}{v_{\nG}} < \lambda_1 <  \frac{T^n_{j-\frac 1 2}(v_{\nG-1})}{v_{\nG-1}} < \lambda_2 < (\dots) < \lambda_{\nG-1} < \frac{T^n_{j-\frac 1 2}(v_{1})}{v_{1}}.
\end{multline} 
\begin{rmk}
There is some subtlety hidden there, because the values $T^n_{j-\frac 1 2}(v_{k})/v_{k}$ do not necessarily respect the order of $(1/v_k)$. This would be the case if  $T^n_{j-\frac 1 2}(v)$ depends only on the sign of $v$, as assumed implicitly in \eqref{eq:entrelacing}, {\em e.g.} in a stationary chemical field $\chS(t,x) = \chS(x)$, for which  $\frac{D \chS}{Dt} =   v' \Dx \chS $. In full generality, the value of $\calT$   depends on the sign of two material derivatives, which are obviously affine with respect to $v$. Consequently, there exist two cutting values, where each contribution in $\calT$ changes sign. Both values on each sides of these cuts respect the order of $(1/v_k)$, simply because $\calT$ is piecewise constant. 
\end{rmk}

Notice that  the sign of $\lambda_0\neq 0$ is determined by the sign of the mean flux
\begin{equation}
\sum_\ksum\omega_k   \frac{v_k}{T^n_{j-\frac 1 2}(v_k)} .
\end{equation}
As a conclusion, the solutions of  the approximated BVP can be written in a general form as a combination of $2\nG$ independent modes:
\begin{equation}\boxed{
\tfe(x,v) =   \frac{\overline{A}}{T^n_{j-\frac12}(v)} + \sum_{\ell = -\nG+1}^{\nG-1} \frac{A_\ell}{T^n_{j-\frac12}(v) - \lambda_\ell v} \exp(-\lambda_\ell x),}
\end{equation}
where the $\lambda_\ell$'s are eigenvalues of a rank-1 perturbation of a diagonal matrix,
\begin{displaymath}
P = \mathrm{diag}\left(T_{j-\frac 1 2}(\V)\V^{-1}\right)-(\O T_{j-\frac 1 2}(\V)) \otimes (\V^{-1})^T, 
\end{displaymath}
associated to an eigenvector $\vp_\ell(v)$. 
The above formulation enables to build the scattering matrix \eqref{eq:outgoing_incoming}, which relates inflow data to values at the middle of the staggered cell ({\em i.e.} the interface between $C_{j-1}$ and $C_j$), see Figure \ref{fig:scheme}.  
Indeed, the degrees of freedom $\overline{A},(A_\ell)$ are obtained by solving the following system of linear equations:
\begin{subequations}\label{eq:expansion_cases}
\begin{equation}
\begin{cases}
(\forall k\in [1,\nG])\quad  f^n_{j-1}(v_k) = \displaystyle  \frac{\overline{A}}{T^n_{j-\frac12}(v_k)} + \sum_{\ell = -\nG+1}^{\nG-1} \frac{A_\ell}{T^n_{j-\frac12}(v_k) - \lambda_\ell v_k} \exp(-\lambda_\ell x_{j-1}) \medskip,\\
(\forall k\in [-\nG,-1])\quad  f^n_{j}(v_k) = \displaystyle  \frac{\overline{B}}{T^n_{j-\frac12}(v_k)} + \sum_{\ell = -\nG+1}^{\nG-1} \frac{B_\ell}{T^n_{j-\frac12}(v_k) - \lambda_\ell v_k} \exp(-\lambda_\ell x_{j}), 
\end{cases}
\end{equation}
which in the matrix form writes as
\begin{displaymath}
\left( \begin{array}{c}f_{j-1}(\calV) \\ f_{j}(-\calV) \end{array}\right)=M
\left( \begin{array}{c}\mathbf{A}\\ \mathbf{B}\end{array}\right).
\end{displaymath}
Next, values at any interface are obtained using the following reconstruction
\begin{equation}
\begin{cases}
(\forall k\in [1,\nG])\quad  \tf^n_{j-\frac12}(v_k) = \displaystyle  \frac{\overline{A}}{T^n_{j-\frac12}(v_k)} + \sum_{\ell = -\nG+1}^{\nG-1} \frac{A_\ell}{T^n_{j-\frac12}(v_k) - \lambda_\ell v_k} \exp(-\lambda_\ell x_{j-\frac12}) \medskip,\\
(\forall k\in [-\nG,-1])\quad  \tf^n_{j-\frac12}(v_k) = \displaystyle  \frac{\overline{B}}{T^n_{j-\frac12}(v_k)} + \sum_{\ell = -\nG+1}^{\nG-1} \frac{B_\ell}{T^n_{j-\frac12}(v_k) - \lambda_\ell v_k} \exp(-\lambda_\ell x_{j-\frac12}). 
\end{cases}
\end{equation}
\end{subequations}
This definines a complementary matrix $\tilde{M}$ such that
\begin{displaymath}
\left( \begin{array}{c}\tf_{j-\frac12}(\calV) \\ \tf_{j-\frac12}(-\calV) \end{array}\right)=\tilde{M}
\left( \begin{array}{c}\mathbf{A}\\ \mathbf{B}\end{array}\right).
\end{displaymath}
Eliminating the vector of coefficients $\mathbf{A}$, $\mathbf{B}$ we obtain \eqref{eq:outgoing_incoming} with 
\begin{equation}\label{eq:S_productWB}
\boxed{
S=\tilde{M}M^{-1}.}
\end{equation}
\begin{rmk}
The numerical procedure described above relies on both the computation of $2\nG-1$ eigenvalues, and the resolution of a $2\nG \times 2\nG$ linear system, at each time-step, at each interface. However, based on the very simple structure of the scattering operator studied here, namely $T$ can take only four possible values, $ 1 \pm \chi_{\chS} \pm \chi_{\chN}$, we can reduce that task to a relatively small number of cases. Indeed, one should discuss the possible cutting values where each contribution in $T$ changes sign, making just $2\times 2\times(2\nG+1)$ possibilities. All in all, this scheme requires the pre-computation of $8\nG+4$ sets of $2\nG-1$ eigenvalues $\lambda_\ell$, and linear systems resolutions. 
\end{rmk}

The previous analysis extends  to the case of a frame moving at speed $c$: $z = x-ct$ (see \cite[Section 7]{calvez_existence}). This makes sense when seeking wave propagation phenomena.  This construction shows some similarities with the one presented in \cite{emako_tang}.

\subsubsection{Second-order finite-difference approximation}\label{sec:S_approx}
 Oppositely, relying on \cite[\S10.4]{Gosse_book}, a simpler approach consists in discretizing the stationary problem in both velocity and space variables with a second order finite-difference approximation: for any $k$, 
\begin{align*}
v_k\frac{\tf_{j-\frac 1 2}(v_k)- f_{j-1}(v_k)}{\DX}& = - T_{j-\frac 1 2}(v_k)\frac{\tf_{j-\frac 1 2}(v_k)+ f_{j-1}(v_k)}{2}\\
&+\sum_\lsum\omega_\ell T_{j-\frac 1 2}(v_\ell)\frac{\tf_{j-\frac 1 2}(v_\ell)+f_{j-1}(v_\ell)}{2}\quad\textrm{ if }v_{k}>0,\\
v_k\frac{f_{j}(v_k)-\tf_{j-\frac 1 2}(v_k)}{\DX}& = - T_{j-\frac 1 2}(v_k)\frac{f_{j}(v_k)+\tf_{j-\frac 1 2}(v_k)}{2}\\
&+\sum_\lsum\omega_\ell T_{j-\frac 1 2}(v_\ell)\frac{f_{j}(v_\ell)+\tf_{j-\frac 1 2}(v_\ell)}{2}\quad\textrm{ if }v_{k}<0.
\end{align*}
Considering $f_{j-1/2}(v_k)$ as unknowns we obtain a linear system at each interface
\begin{equation}\label{eq:interface_variables}
  Q_{j-1/2}\left(
  \begin{array}{c}
    f_{j-1/2}(\calV)\\
    f_{j-1/2}(-\calV)
  \end{array}\right) = \tilde{Q}_{j-1/2}
  \left(\begin{array}{c}
    f_{j-1}(\calV)\\
    f_{j}(-\calV)
  \end{array}\right),
\end{equation}
where $Q_{j-1/2},\tilde{Q}_{j-1/2}$ are $2\nG\times 2\nG$ space and time dependent matrices:
 \begin{subequations}\label{eq:Smatrix_approx}
     \begin{align}
      Q_{j-1/2}&=\textrm{diag} \left(|\calW|+\frac{\Delta x}{2}T_{j-1/2}(\calW)\right)-\frac{\Delta x}{2}(\O T_{j-1/2}(\calW))\otimes \mathbf{1}_{2\nG\times 1},\\
      \tilde{Q}_{j-1/2}&=\textrm{diag} \left(|\calW|-\frac{\Delta x}{2}T_{j-1/2}(\calW)\right)+\frac{\Delta x}{2}(\O T_{j-1/2}(\calW))\otimes \mathbf{1}_{2\nG\times 1}.
     \end{align}
 \end{subequations}
\begin{lemma}\label{lem:Smatrix_approx}
Under the sufficient ``non-resonance'' condition,
   \BE\label{eq:condition_dv}
       v_{min}:=\min_{K=1, ..., \nG}(|v_k|)> \Delta x\cdot \left(\chi_{\chS}+\chi_{\chN}\right)\,,
   \EE
the matrix $Q$ is invertible, so the scattering matrix is 
\begin{equation}\label{eq:S_productApprox}
S_{j-1/2}=Q_{j-1/2}^{-1}\tilde{Q}_{j-1/2}.
\end{equation}
\end{lemma}
\begin{proof}
The matrix $Q=(q)_{m,n}$, $m,n\in\{1,...,2\nG\}$ is invertible if it is strictly diagonally dominant, that is, for each row $m\in\{1,...,2\nG\}$ the following must hold
\begin{displaymath}
   |q_{m,m}|-\sum_{n=1, n\neq m}^{2\nG}|q_{m,n}|>0. 
\end{displaymath}
Using the explicit formula for $Q$ the above condition becomes
\begin{align*}
  \Gamma_k
  =\left||v_k|+\frac{\Delta x}{2}T(v_k)-\frac{\Delta x}{2}\omega_k T(v_k)\right|
   -\frac{\Delta x}{2}\left(\sum_\lsum|\omega_\ell T(v_k)|-|\omega_kT(v_k)|\right).
\end{align*}
Since $T>0$ and $\omega_k\leq 1$ we have $T(v_k)-\omega_k T(v_k)\geq 0$, so omitting the moduli yields
\begin{align*}
   \Gamma_k & = |v_k| + \frac{\Delta x}{2}\sum_\lsum\omega_\ell T (v_k)-\frac{\Delta x}{2}\sum_\lsum\omega_\ell T(v_\ell)\\ 
   & = |v_k| + \frac{\Delta x}{2}\sum_\lsum\omega_\ell\left(T(v_k)-T(v_\ell)\right) \geq \min_k(|v_k|) -  \Delta x\cdot \left(\chi_{\chS}+\chi_{\chN}\right). 
\end{align*}
\end{proof}
%

%%%%%%%%%%%%%%%%%%%%%%%%%%%%%%%%%%%%%%%%%%%%%%%%%%%%%%%%%%%%%%%%%%%%%%%%%%%%%%%%%%%%%%%%%%%
\subsection{Well-balanced schemes for reaction-diffusion equations}
%%%%%%%%%%%%%%%%%%%%%%%%%%%%%%%%%%%%%%%%%%%%%%%%%%%%%%%%%%%%%%%%%%%%%%%%%%%%%%%%%%%%%%%%%%%

Well-balanced discretizations for linear diffusive equations including lower-order terms were recently introduced in \cite{Gosse_Lsplines}, extending previous works mostly devoted to stationary models. In particular, it was shown that in the vanishing viscosity limit, usual well-balanced schemes for the remaining hyperbolic equations were recovered.

As both the parabolic equations showing up in (\ref{para-S-N}) are very similar, the treatment of a generic dissipative diffusion-reaction model for a generic unknown $u(t,x)$ will be presented hereafter.
We seek a numerical approximation of Cauchy problem,
\begin{equation}\label{eq:par_general}
\Dt u - D \Dxx u + p(x) u = q(t,x)
\end{equation}
with initial and boundary data, so that the resulting scheme recovers the collection of points $u(x_j)$, for $x_j$ being the nodes of the grid. The main idea of the scheme is to derive numerical flux functions using the properties of the ''$\EL$-spline'' interpolation of the data $u_j^n$ that is solving the stationary problem of \eqref{eq:par_general} in $(x_{j-1},x_j)$
\begin{equation}\label{eq:par_steady}
-D\Dxx v + pv=q
\end{equation}
and imposing the $C^1$ regularity conditions at interfaces. In $(x_{j-1},x_{j})$, both $p$ and $q$ are assumed to be constant, each ''local profile'' is obtained by a standard variation of constants technique involving exponential functions. 
\begin{itemize}
\item Let $r_{j-\frac 1 2}=\sqrt{{p_{j-1/2}}/{D}}>0$ in $(x_{j-1},x_j)$ and $\bar{q}_{j-\frac 1 2}=q_{j-\frac 1 2}/p_{j-\frac 1 2}$. Exponentials $\{\exp(r_{j-\frac1 2}x),\exp(-r_{j-\frac1 2}x)\}$ $\left(\textrm{resp. }\{\exp(r_{j+\frac 1 2}x),\exp(-r_{j+\frac1 2}x)\}\right)$ form a fundamental basis for an operator $\EL:v\rightarrow-D\Dxx v+pv$ in $(x_{j-1},x_j)$ $\left(\textrm{resp. }(x_{j},x_{j+1})\right)$. Let $v$ be a steady state solution to \eqref{eq:par_general}. Then in $(x_{j-1},x_j)$ it is written as
\begin{equation}
  v(x) = \left<
  \left( \begin{array}{c} A \\ B \end{array}\right),
  \left( \begin{array}{c} \exp(r_{j-\frac 1 2} x) \\ \exp(-r_{j-\frac 1 2} x) \end{array}\right)
  \right>+\bar{q}_{j-\frac 1 2},
\end{equation}
where $\left<\cdot,\cdot\right>$ denotes the canonical scalar product and $A,B$ are the inegration constants. Solutions to \eqref{eq:par_steady} at $x_{j-1},x_{j}$ rewrite as
\begin{equation}\label{eq:par_steady_solution}
\left( \begin{array}{c}
v_{j-1}^n - \bar q_{j-\frac 1 2} \\ v_j^n - \bar q_{j-\frac 1 2} 
\end{array}\right) =\underbrace{\left( \begin{array}{cc}
\exp(r_{j-\frac 1 2} x_{j-1}) & \exp(-r_{j-\frac 1 2} x_{j-1}) \\
\exp(r_{j-\frac 1 2} x_{j}) & \exp(-r_{j-\frac 1 2} x_{j}) \\
\end{array}\right)}_{Z_{j-\frac 1 2}}
\left( \begin{array}{c}
A \\ B \end{array}\right).
\end{equation}
The determinant $|Z_{j-\frac 1 2}|= -2\sinh(r_{j-\frac 1 2} \DX)\not =0$, so the matrix is invertible.
\item A solution $v_j$ belongs to a unique stationary curve defined on $(x_{j-1},x_{j+1})$, such that $u(x_{j\pm 1})=v_{j\pm 1}$, if at the node $x_j$ the $C^1$ regularity is assured.  Inside $(x_{j-1},x_j)$ we have
\begin{equation}
  v'(x) = \left<
  \left( \begin{array}{c} A \\ B \end{array}\right),
  r_{j-\frac 1 2} \left( \begin{array}{c} \exp(r_{j-\frac 1 2} x) \\ -\exp(-r_{j-\frac 1 2} x) \end{array}\right)
  \right>,
\end{equation}
so that, the $C^1$ smoothness at $x_j$ reads,
\begin{align}\label{eq:par_c1condition}
R_{j-\frac 1 2}^n:=& r_{j-\frac 1 2}\left<Z_{j- \frac 1 2}^{-1}\left( \begin{array}{c}
v_{j-1}^n - \bar q_{j-\frac 1 2} \\ v_j^n - \bar q_{j-\frac 1 2} 
\end{array}\right), \left( \begin{array}{c}
\exp(r_{j-\frac 1 2} x_j) \\ -\exp(-r_{j-\frac 1 2} x_j)
\end{array}\right) \right> \\
 =&r_{j+\frac 1 2}\left<Z_{j+ \frac 1 2}^{-1}\left( \begin{array}{c}
v_{j}^n - \bar q_{j+\frac 1 2} \\ v_{j+1}^n - \bar q_{j+\frac 1 2} 
\end{array}\right), \left( \begin{array}{c}
\exp(r_{j+\frac 1 2} x_j) \\ -\exp(-r_{j+\frac 1 2} x_j)
\end{array}\right) \right>=:L_{j+\frac 1 2}^n, \nonumber
\end{align}
where relation \eqref{eq:par_steady_solution} was used to replace the integration constants $A,B$.
\end{itemize}
The time-marching strategy consists in defining the discrete time derivative as the defect of $C^1$ smoothness at each $x_j$, that is, a difference of normal derivatives,
\begin{equation}\label{normal}
u_{j}^{n+1} = u_{j}^{n}-\frac{D\DT}{\DX}\left[L_{j+\frac 1 2}^n-R_{j-\frac 1 2}^n\right],
\end{equation}
where $L_{j+\frac 1 2}$ $\left(\textrm{resp. }R_{j-\frac 1 2}\right)$ is the right (resp. left) hand side of \eqref{eq:par_c1condition}. By developing these terms, the scheme for the time evolution of concentrations $\chS$, $\chN$ rewrites as
\begin{subequations}\label{eq:scheme_parabolic}
\begin{itemize}
\item signal $\chS$:  $p=\alpha$, $q = \beta\rho$, $r_{j+1/2}=\sqrt{{p}/{D}}\equiv r$ 
  \begin{align}
% \nonumber   \chS_{j}^{n+1} = \chS_{j}^{n} + \frac{\Delta t}{\Delta x}\frac{\sqrt{\alpha D_{\chS}}}{\sinh(r\Delta x)}\left(\chS_{j+1}^{n}-2\cosh(r\Delta x)\chS_j^n+\chS_{j-1}^{n}\right)\\
%    +\frac{\Delta t}{\Delta x}\frac{\beta\sqrt{D_{\chS}}}{\sqrt{\alpha}}\frac{\cosh(r\Delta x)-1}{\sinh(r\Delta x)}\left(\rho^n_{j+1/2}+\rho^n_{j-1/2}\right),
 \nonumber   \chS_{j}^{n+1} = \chS_{j}^{n} + \frac{\Delta t\, \sqrt{\alpha D_{\chS}}}{\Delta x\sinh(r\Delta x)}\left\{ \Big(\chS_{j+1}^{n}-2\cosh(r\Delta x)\chS_j^n+\chS_{j-1}^{n}\Big) \right.\\
    +\left. \frac{\beta}{\alpha}\frac{\cosh(r\Delta x)-1}{\sinh(r\Delta x)}\left(\rho^n_{j+1/2}+\rho^n_{j-1/2}\right) \right\},
  \end{align}  
where, presently, $\rho^n_{j \pm 1/2}$ can be defined as an arithmetic average.
\item nutrient $\chN$: $p = \gamma\rho$, $q=0$, $r_{j+1/2}=\sqrt{{\gamma\rho_{j+1/2}}/{D_{N}}}$
  \begin{align}
\nonumber    \chN_{j}^{n+1} = \chN_{j}^{n} +& \frac{\Delta t\, \sqrt{\gamma D_{\chN}}}{\Delta x}\left\{\frac{\sqrt{\rho_{j+1/2}}}{\sinh(r_{j+1/2}\Delta x)}\left( \chN_{j+1}^{n} -\cosh (r_{j+1/2}\Delta x) \chN_{j}^n \right) \right.\\
   & \left. -\frac{\sqrt{\rho_{j-1/2}}}{\sinh(r_{j-1/2}\Delta x)}\left( \cosh(r_{j-1/2}\Delta x)\chN_{j}^{n} - \chN_{j-1}^{n}\right)    \right\}.
  \end{align}  
\end{itemize}
\end{subequations}
Consistency is established like in \cite[Theorem 6.2]{Gosse_Lsplines}, essentially by performing Taylor expansions in every hyperbolic trigonometric function while sending $\DX \to 0$. The scheme is just a linear 3-point space-discretization, although it is ``exponential-fit'' like Scharfetter-Gummel's, so that implementing a Crank-Nicolson ($\theta$-method with $\theta=\frac 1 2$) time-integration is easy and produces second order accuracy. 

%%%%%%%%%%%%%%%%%%%%%%%%%%%%%%%%%%%%%%%%%%%%%%%%%%%%%%%%%%%%%%%%%%%%%%%%%%%%%%%%%%%%%%%%%%%
\subsection{Approximation of material derivatives}\label{sec:tumbling}
%%%%%%%%%%%%%%%%%%%%%%%%%%%%%%%%%%%%%%%%%%%%%%%%%%%%%%%%%%%%%%%%%%%%%%%%%%%%%%%%%%%%%%%%%%%
The scattering matrix, evaluated at each time-step and interface, requires a good approximation of material derivatives (\ref{tumbling}) inside the ``sign'' functions. We present here two methods, for which accuracy is addressed numerically in the next sections. For brevity,  only the concentration of $\chS$ is considered, as identical formulas apply to $N$ as well:
\begin{itemize}
\item Method MD-1: One of the choices to approximate $D\chS/Dt|_{j+1/2}^{n}$ is to define a piecewise constant approximation of $\chS(t,x)$ centered at the grid nodes,
\begin{displaymath}
  \forall j,n\in\mathbb{Z}\times\mathbb{N},\quad \chS_{j}^{n}=\chS(t^n=n\Delta t,x_j=j\Delta x),
\end{displaymath}
and use the definition of the material derivative $D\chS/Dt = \partial_t \chS+v\partial_x \chS$ with 
\begin{subequations}\label{eq:md1}
\begin{align}
  (\Dx\chS )_{j+1/2}^n &:= \frac{\chS_{j+1}^n-\chS_{j}^n}{\Delta x},\\
  (\Dt\chS )_{j+1/2}^n &:= \frac{1}{2}\left(\frac{\chS_{j+1}^n-\chS_{j+1}^{n-1}}{\Delta t}+\frac{\chS_{j}^n-\chS_{j}^{n-1}}{\Delta t}\right).&&
\end{align}
\end{subequations}
The space derivative is well defined at interfaces, but, not the approximation of the time derivative:  averaging might introduce additional errors.
\item Method MD-2:  Another way is to approximate directly the material derivative  
\begin{equation}
\frac{D\chS}{Dt}(x,t) = \frac{1}{\Delta t}(\chS(x,t)-\chS(x-v\Delta t,t-\Delta t)).
\end{equation}
The CFL condition \eqref{eq:cfl} assures that $x-v\Delta t\in(0,\Delta x)$. This definition is more coherent with the behavior of bacteria which  measure variations of  concentration of a chemical along their trajectory. As $DM/Dt$ has to be defined at the interfaces,  it becomes natural to define also $\chS$ on the interfaces, instead of on the grid nodes,
\begin{displaymath}
  \forall j,n\in\mathbb{Z}\times\mathbb{N},\quad \chS_{j}^{n}=\chS(t^n=n\Delta t,x_j=(j+1/2)\Delta x).
\end{displaymath}
Using the linear combination of values at the cell boundaries with the upwinding with respect to the velocity to approximate $M(x-v\DX,t-\DT)$ we obtain
\begin{equation}\label{eq:md2}
  \left.\Delta t\frac{D\chS}{Dt}\right|_{j+1/2}^{n} = \chS_{j+1/2}^{n}-
  \begin{cases}
\displaystyle  \left(1-\Delta t\frac{v}{\Delta x}\right)\chS_{j+1/2}^{n-1}+\Delta t\frac{v}{\Delta x}\chS_{j-1/2}^{n-1}&\text{if $v>0$}\medskip\\
\ds  \left(1+\Delta t\frac{v}{\Delta x}\right)\chS_{j+1/2}^{n-1}-\Delta t\frac{v}{\Delta x}\chS_{j+3/2}^{n-1}&\text{if $v<0$}
  \end{cases}
\end{equation}
\end{itemize}
\begin{rmk}
We note that using the well-balanced approximation of the parabolic equations \eqref{para-S-N} allows to avoid additional approximations in the method MD-2. More precisely, if the concentrations of the signal $\chS$ and the nutrient $\chN$ are computed at interfaces, then the values of $\rho_{j+1/2}$ in the scheme \eqref{eq:scheme_parabolic} coincide with the grid nodes.  
 \end{rmk} 

%%%%%%%%%%%%%%%%%%%%%%%%%%%%%%%%%%%%%%%%%%%%%%%%%%%%%%%%%%%%%%%%%%%%%%%%%%%%%%%%%%%%%%%%%%%
\subsection{Simple centered, time-splitting (TS) approach}
%%%%%%%%%%%%%%%%%%%%%%%%%%%%%%%%%%%%%%%%%%%%%%%%%%%%%%%%%%%%%%%%%%%%%%%%%%%%%%%%%%%%%%%%%%%
For comparison purposes, we present an alternative, standard scheme based on time-splitting. A main difference with respect to well-balanced techniques lies in a ``time localization'' of transport and tumbling terms (see \cite{Gosse_M3AS}). Processes are separated, so that the scheme splits into two distinct phases. In case of the kinetic equation \eqref{kinetic}-(\ref{tumbling}), we have the following discretization:
\begin{enumerate}
\item transport with velocity $v_k$ is solved by the classical upwind algorithm,
\begin{subequations}\label{eq:ts_kinetic}
\begin{align}
 \nonumber f_{j}^{*}(|v_{k}|) &= f_{j}^{n}(|v_{k}|)-|v_{k}|\frac{\Delta t}{\Delta x}\left(f_{j}^{n}(|v_{k}|)-f_{j-1}^{n}(|v_{k}|)\right),\\
  f_{j}^{*}(-|v_{k}|) &= f_{j}^{n}(-|v_{k}|)+|v_{k}|\frac{\Delta t}{\Delta x}\left(f_{j+1}^{n}(-|v_{k}|)-f_{j}^{n}(-|v_{k}|)\right),
\end{align}
It corresponds to scheme \eqref{god} with the identity as the scattering matrix. 
\item tumbling is an ordinary differential equation solved by explicit integration,
\BE
f_{j}^{n+1}(v_k) = (1-\Delta t T_{j}^{n}(v_k))f_{j}^{*} + \Delta t\sum_{l=1}^{2\nG}\omega_l T^{n}_{j}(v_l)f^{*}_{j}(v_l),
\EE
\end{subequations}
where $T_j(\cdot)=T(x_j,\cdot)$ is computed at each node, hence the word ``centered''.
\end{enumerate}
A similar method for a generic reaction-diffusion equation \eqref{eq:par_general} reads,
\begin{equation}\label{eq:ts_parabolic}
u^{n+1}_j=u^n_j + \frac{D\Delta t}{\Delta x}\left(\left[\frac{u^n_{j+1}-u^n_j}{\DX}\right]- \left[\frac{u_j^n-u^n_{j-1}}{\DX}\right] \right)- \DT( p_j\, u^n_j - q^n_j).
\end{equation}
As usual, the CFL restriction for linear stability is:
$$
\left(\frac{2D}{\DX^2} -\|\max(0,p)\|_\infty \right)\DT \leq  1.
$$
Such discretization of diffusive terms corresponds to (\ref{normal}), where $q=p=0$ is forced into  (\ref{eq:par_general}). Corresponding $\mathcal L$-spline interpolation reduces to piecewise-linear, because $\mathcal L$ is just the second derivative, $v \mapsto -D\Dxx v$, which fundamental system is $\{1,x\}$.

%{\jg Monika a cod\'e Crank Nicolson}

%%%%%%%%%%%%%%%%%%%%%%%%%%%%%%%%%%%%%%%%%%%%%%%%%%%%%%%%%%%%%%%%%%%%%%%%%%%%%%%%%%%%%%%%%%%
%%%%%%%%%%%%%%%%%%%%%%%%%%%%%%%%%%%%%%%%%%%%%%%%%%%%%%%%%%%%%%%%%%%%%%%%%%%%%%%%%%%%%%%%%%%
\section{First numerical assessments}\label{sec:sim1}
%%%%%%%%%%%%%%%%%%%%%%%%%%%%%%%%%%%%%%%%%%%%%%%%%%%%%%%%%%%%%%%%%%%%%%%%%%%%%%%%%%%%%%%%%%%
%%%%%%%%%%%%%%%%%%%%%%%%%%%%%%%%%%%%%%%%%%%%%%%%%%%%%%%%%%%%%%%%%%%%%%%%%%%%%%%%%%%%%%%%%%%

We study the accuracy of numerical schemes presented in the previous section. First, we focus on analyzing the properties of two types of scattering matrices: one based on Case's elementary solutions and another derived from a finite difference approximation, see Section~\ref{sec:S_case} and \ref{sec:S_approx} respectively. Then, we compare how different numerical approaches resolve the momentum of a travelling waves  used to compute its mean velocity. In particular, we compare:
\begin{itemize}
\item WB-WB: well-balanced for both kinetic \eqref{kinetic} and parabolic \eqref{para-S-N} equations
\item WB-TS: well-balanced for kinetic equation, time-splitting for parabolic ones
\item TS-TS: time-splitting for both kinetic and parabolic equations
\end{itemize}
and two possible approximations of the material derivative, described in Section~\ref{sec:tumbling}: one based on definition  (MD-1), another using upwinding (MD-2), see Table~\ref{tab:schemes}.

\begin{table}
\begin{center}
\begin{tabular}{|c|c|c|}
\hline 
                                  & WB & TS\\
\hline\hline 
Kinetic equation~\eqref{kinetic}  & \eqref{eq:outgoing_incoming}-\eqref{god}-\eqref{eq:expansion_cases}&\eqref{eq:ts_kinetic}\\
Parabolic system~\eqref{para-S-N} & \eqref{eq:scheme_parabolic} & \eqref{eq:ts_parabolic}\\
\hline\hline 
& MD - 1 & MD - 2 \\
\hline\hline 
Tumbling operator \eqref{tumbling} & \eqref{eq:md1} & \eqref{eq:md2}\\
\hline
\end{tabular}
\end{center}
\caption{Reference (names and equations) for the numerical methods used in the simulations.} 
\label{tab:schemes}
\end{table}

%%%%%%%%%%%%%%%%%%%%%%%%%%%%%%%%%%%%%%%%%%%%%%%%%%%%%%%%%%%%%%%%%%%%%%%%%%%%%%%%%%%%%%%%%%%
\subsection{General setting}
%%%%%%%%%%%%%%%%%%%%%%%%%%%%%%%%%%%%%%%%%%%%%%%%%%%%%%%%%%%%%%%%%%%%%%%%%%%%%%%%%%%%%%%%%%%

If not otherwise specified, system \eqref{eq:model}-(\ref{para-S-N}) is posed on $[0,L]\times[-1,1]$ with specular boundary conditions for the kinetic equation, and 
\begin{displaymath}
\partial_x \chS(t,x)|_{x=0}=\partial_x \chS(t,x)|_{x=L}=0,\quad \partial_x \chN(t,x)|_{x=0}=0, \chN(t,x=L) = \bar{\chN}
\end{displaymath}
for the parabolic system, where $\bar{\chN}$ is an arbitrary positive constant. The macroscopic density and velocity are approximated using the quadrature: for all $j \in \Ze,\, n \geq 0$,
\begin{subequations}\label{eq:speed}
\begin{equation}
\rho_j^n =\sum_{\ksum}\omega_k f^n_{j,k}, \qquad u_j^n = \frac{\sum_{\ksum}\omega_kv_kf^n_{j,k}}{\sum_{\ksum}\omega_kf^n_{j,k}},
\end{equation}
while the velocity $c$ is set as the average value of a truncated macroscopic velocity,
\begin{equation}
  c = \left< u_{j}^{n}\cdot\mathbf{1}_{\rho_{j}^{n}>10\%\max_{j}({\rho_{j}^n})}\right>,
\end{equation}
\end{subequations}
where $\mathbf{1}_{A}$ stands for the indicator function of a set $A$. Such a truncation allows to avoid the influence of the numerical noise at low macroscopic densities.

%%%%%%%%%%%%%%%%%%%%%%%%%%%%%%%%%%%%%%%%%%%%%%%%%%%%%%%%%%%%%%%%%%%%%%%%%%%%%%%%%%%%%%%%%%%
\subsection{Properties of the two $S$-matrices}\label{sec:sim_condS}
%%%%%%%%%%%%%%%%%%%%%%%%%%%%%%%%%%%%%%%%%%%%%%%%%%%%%%%%%%%%%%%%%%%%%%%%%%%%%%%%%%%%%%%%%%%

The core part of the well-balanced scheme (\ref{god}) for equation \eqref{kinetic} is the scattering matrix. When it derives from finite differences \eqref{eq:Smatrix_approx}-\eqref{eq:S_productApprox}, then the minimal grid velocity must be bounded from below like (\ref{eq:condition_dv}). It forbids ''too slow particles'', which increase both the stiffness of the linear system \eqref{eq:interface_variables} and the condition number of the $S$-matrix. A high condition number yields amplification of errors present in the incoming states. $S$-matrices based on the Case's solutions \eqref{eq:expansion_cases}-\eqref{eq:S_productWB} are free from the condition (\ref{eq:condition_dv}), however, their stability for numerous discrete velocities is still not entirely clear. 
\begin{figure}[t]
  \begin{center}
    \begin{tabular}{c}
      \includegraphics[scale=0.2]{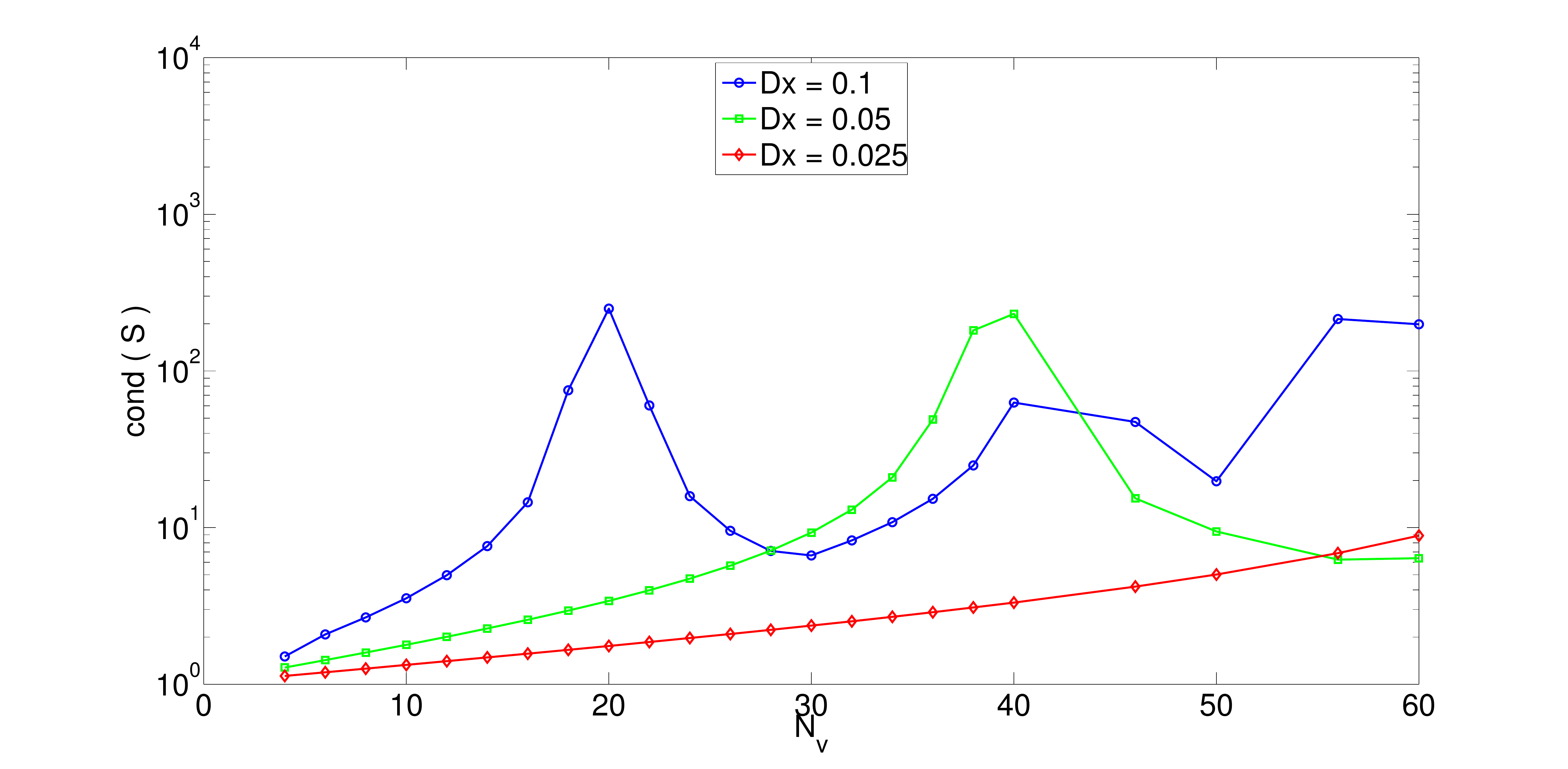}\\
      \includegraphics[scale=0.2]{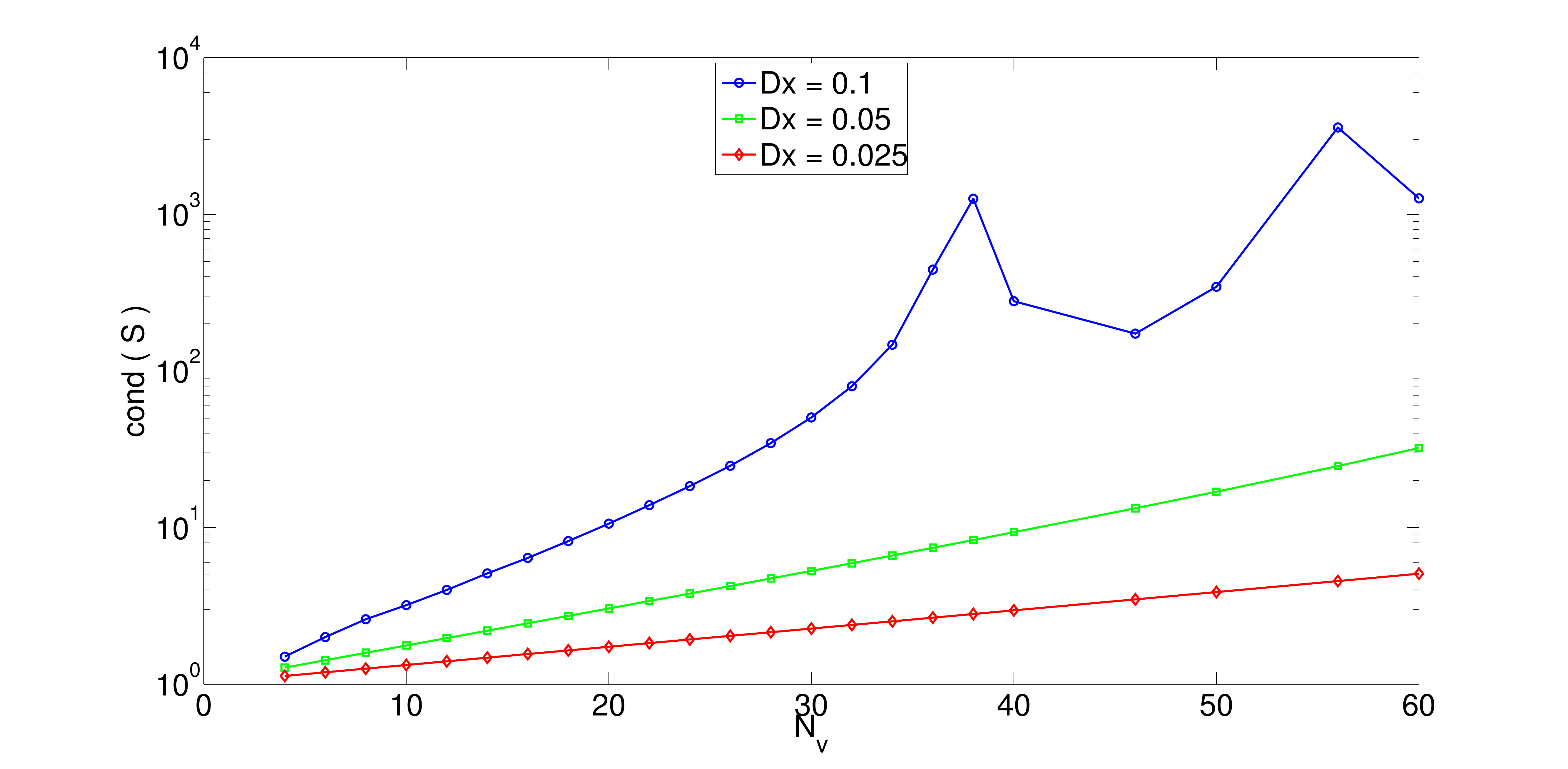}
    \end{tabular}
  \end{center}
  \caption{Condition number of  $S$-matrices as a function of $\nG$ for three grids, $\Delta x = 0.1,0.05,0.0025$.  Comparison between the scattering matrix based on finite differences \eqref{eq:S_productApprox} (top) and the scattering matrix based on the  Case's special functions \eqref{eq:S_productWB} (bottom).}
\label{fig:fullModel_condS}
\end{figure}
Figure~\ref{fig:fullModel_condS} displays condition numbers of both $S$-matrices described in Section~\ref{sec:scattering_matrix} depending on the number of points in a Gauss-Legendre quadrature for three different spatial grids, $\Delta x=0.1,0.05,0.025$. Results for the $S$-matrix based on finite differences are displayed over the ones involving Case's solutions method. In both cases, the condition number increases with the number of discrete velocities $N_v$, which implies  smaller values of $v_{\min}$. The sensitivity to slow particles moving at $v_{min}$ appears clearly weaker for $S$-matrices built on Case's solutions.

%%%%%%%%%%%%%%%%%%%%%%%%%%%%%%%%%%%%%%%%%%%%%%%%%%%%%%%%%%%%%%%%%%%%%%%%%%%%%%%%%%%%%%%%%%%
\subsection{Comparison with a time-splitting algorithm}\label{sec:sim_steady_states}
%%%%%%%%%%%%%%%%%%%%%%%%%%%%%%%%%%%%%%%%%%%%%%%%%%%%%%%%%%%%%%%%%%%%%%%%%%%%%%%%%%%%%%%%%%%

We chose two tests: approximation of the asymptotic states for the aggregation model, that is when $\chi_{\chN}=0$, and approximation of the wave speed $c$ for the full problem.

%%%%%%%%%%%%%%%%%%%%%%%%%%%%%%%%%%%%%%%%%%%%%%%%%%%%%%%%%%%%%%%%%%%%%%%%%%%%%%%%%%%%%%%%%%%
\subsubsection{Breaking the symmetry}

Without nutrient $\chN$, no travelling wave exist, so that macroscopic density  peaks symmetrically at $x=0$. Due to compensation phenomena in the velocity integral, see \cite{calvez_existence}, this maximum is produced {\it despite kinetic densities peak at slightly different locations}. Capturing efficiently such a subtle velocity repartition, see also \cite[Fig. 10.6]{Gosse_book}, is a first requisite.  
\begin{figure}[t]
    \centering
    \subfigure[\emph{Time independent signal $T=1+\chi_{\chS}\textrm{sign}(v \cdot x)$: WB-WB  (top), TS-TS (bottom)}]{
    \begin{tabular}{cc}
      \includegraphics[scale=0.11]{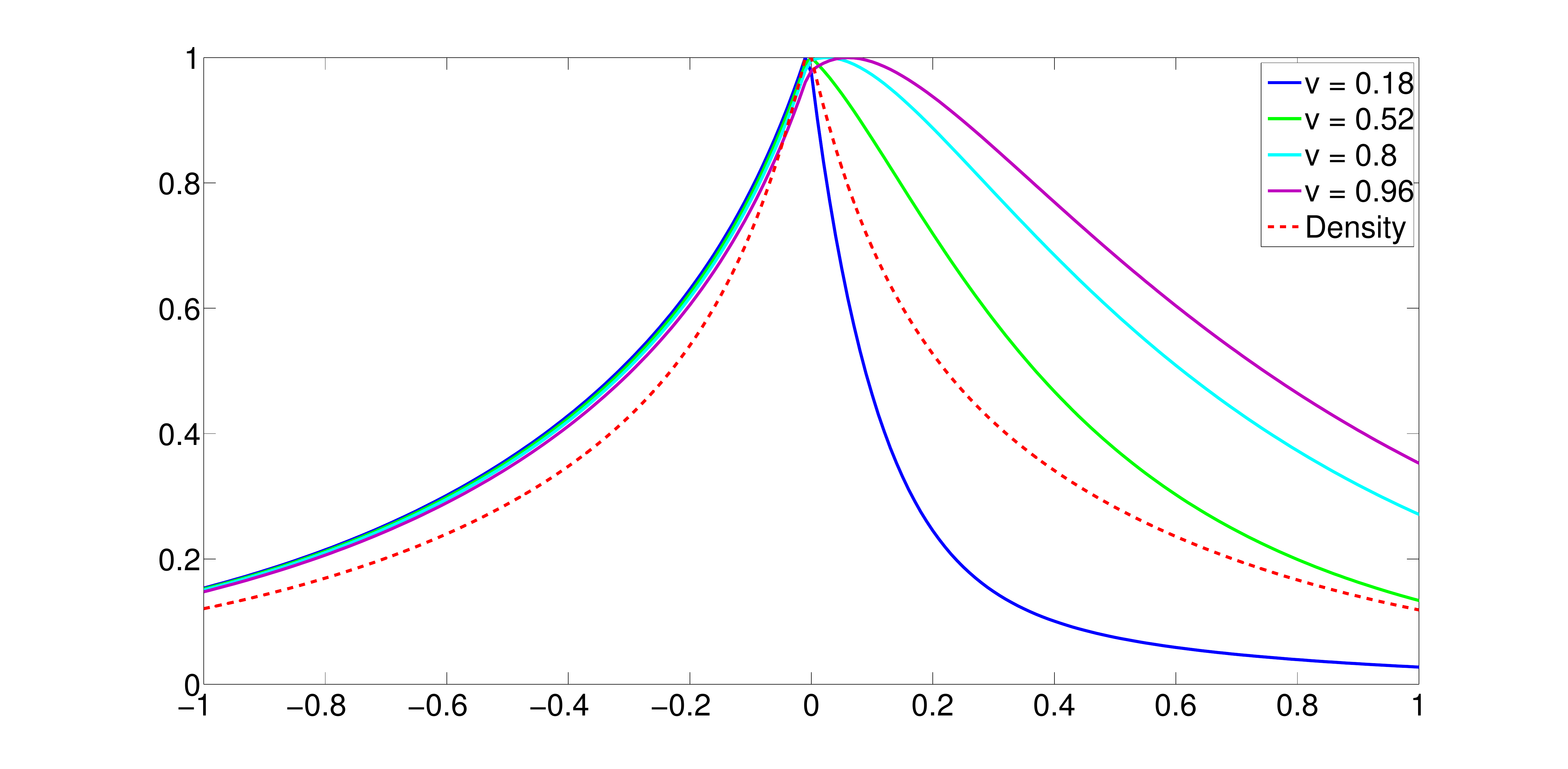}&\includegraphics[scale=0.11]{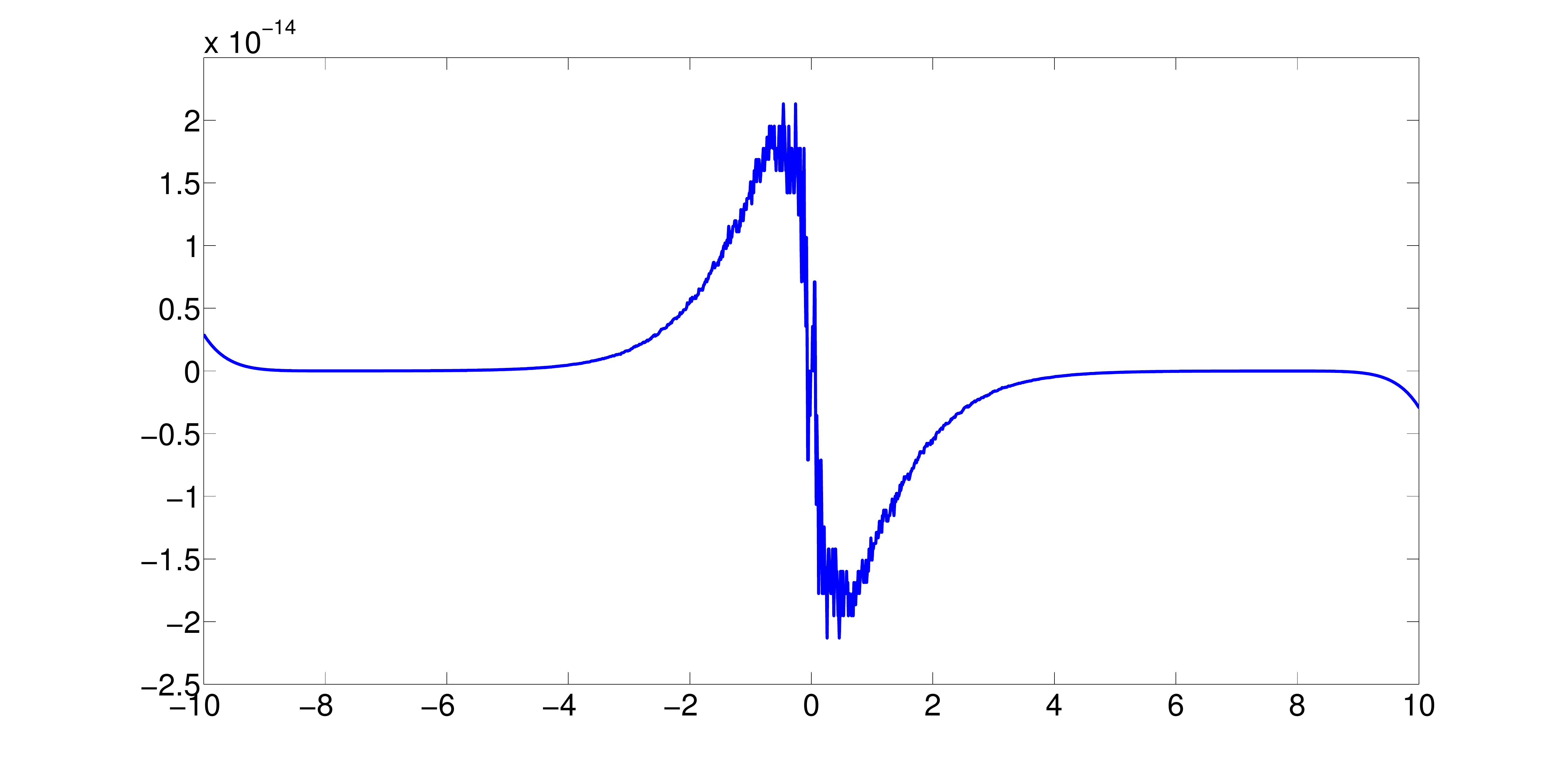}\\
      \includegraphics[scale=0.11]{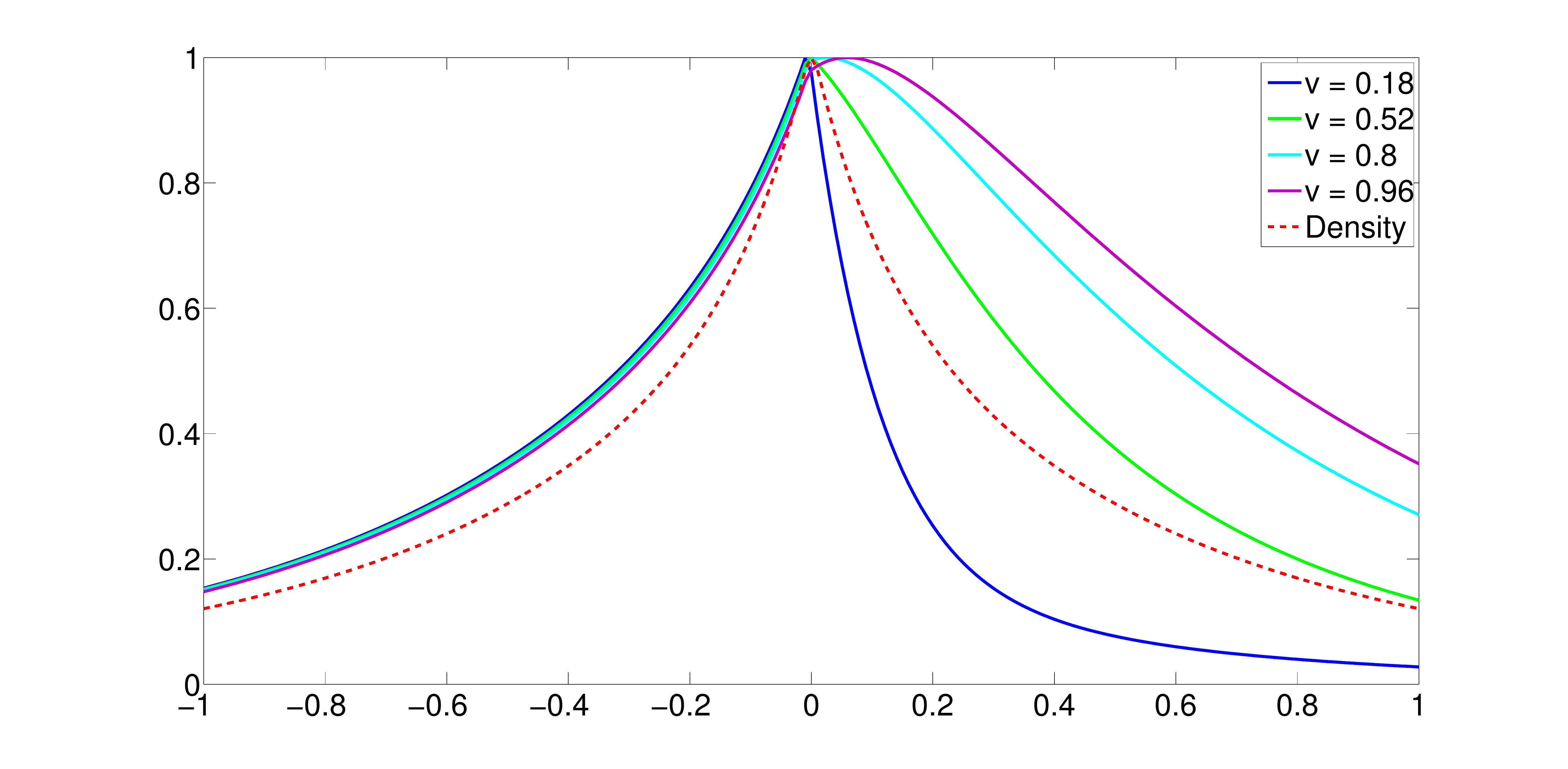}&\includegraphics[scale=0.11]{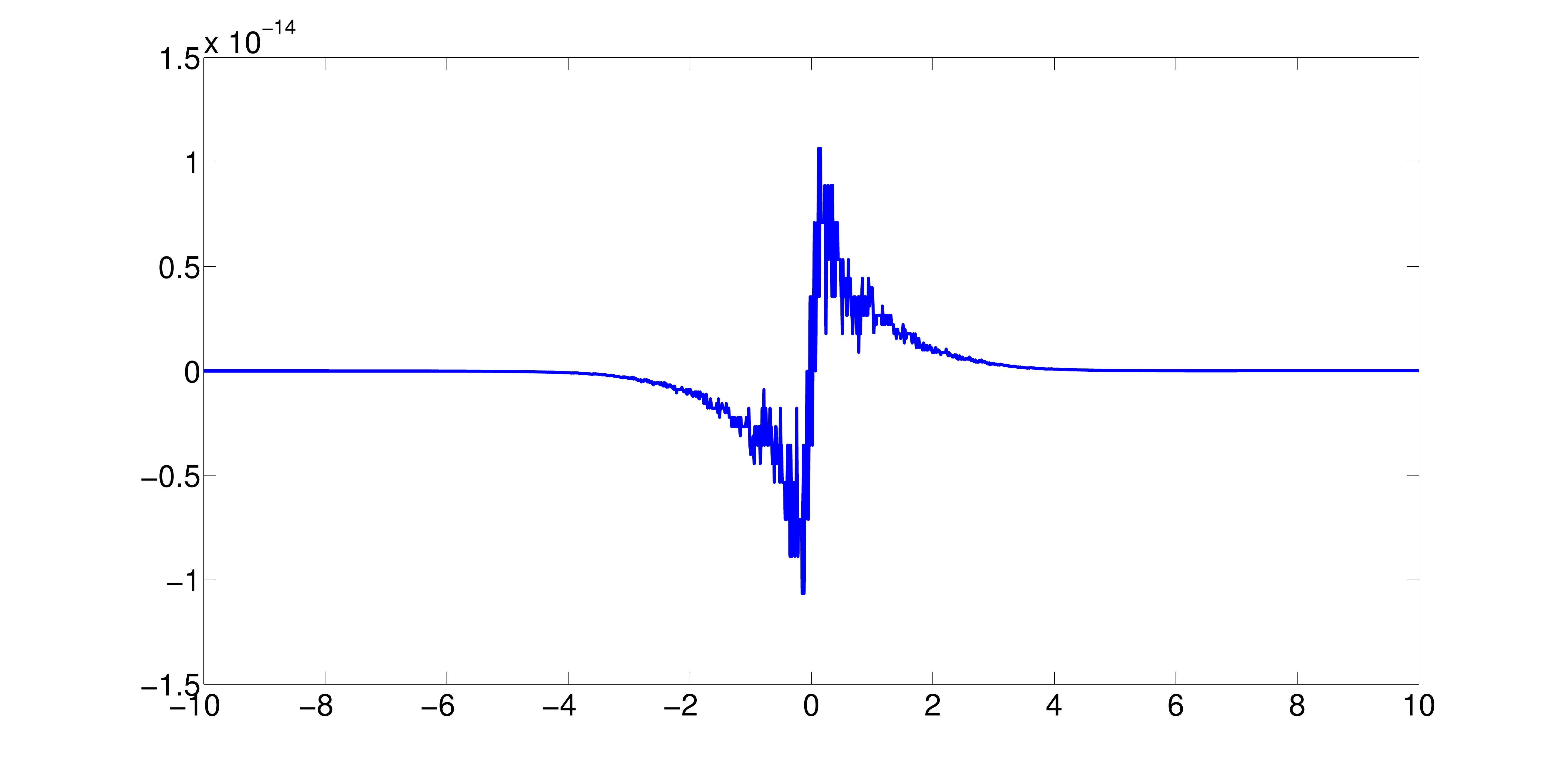}
    \end{tabular}

         \label{fig:compensation_fixedSignal}
    }
    \quad
    \subfigure[\emph{Time dependent signal, $T=1-\chi_{\chS}\textrm{sign}\left(\frac{D\chS}{DT}\right)$: WB-WB (top), TS-TS (bottom)}]{ 
   \begin{tabular}{cc}
      \includegraphics[scale=0.11]{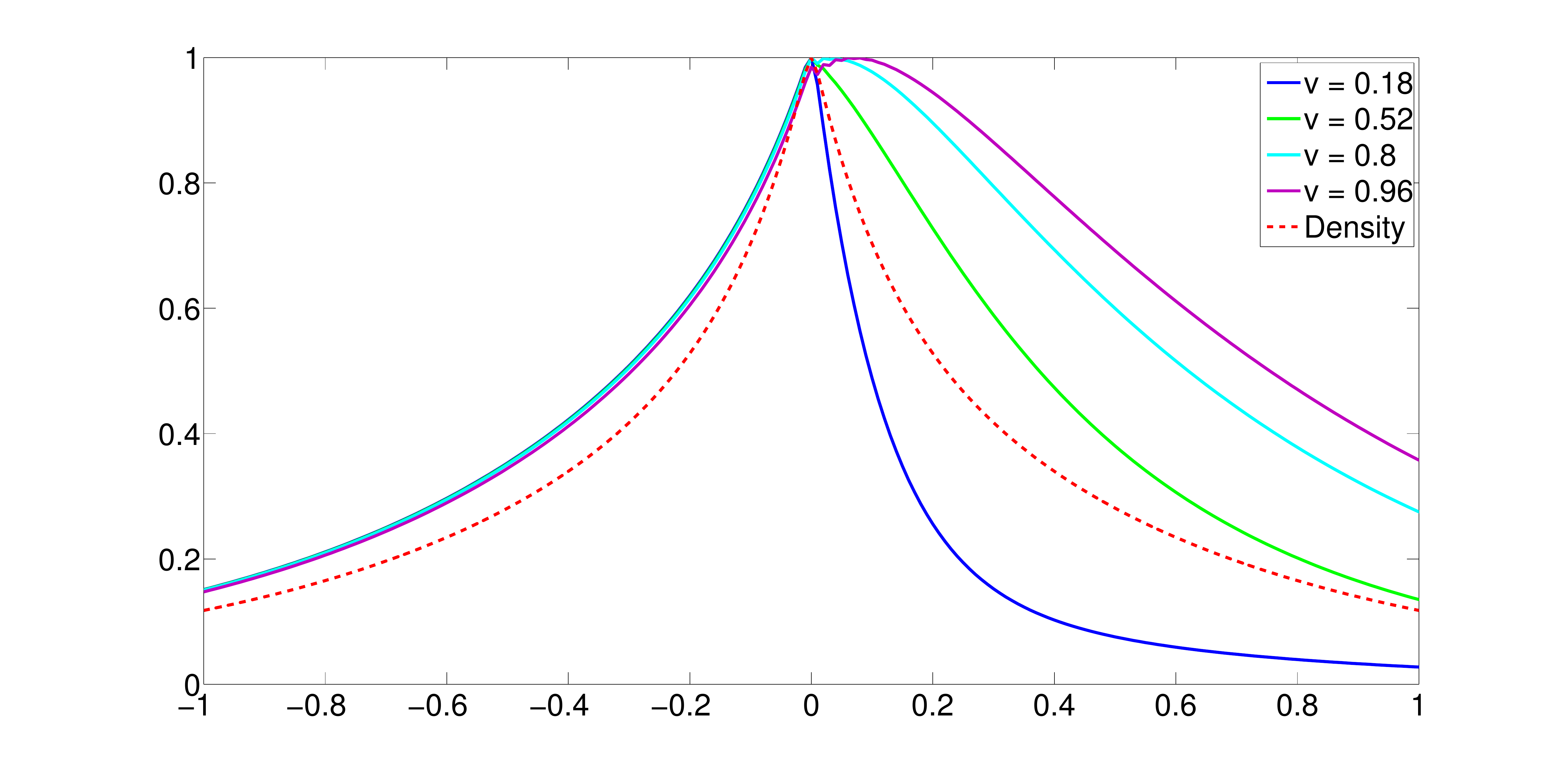}&\includegraphics[scale=0.11]{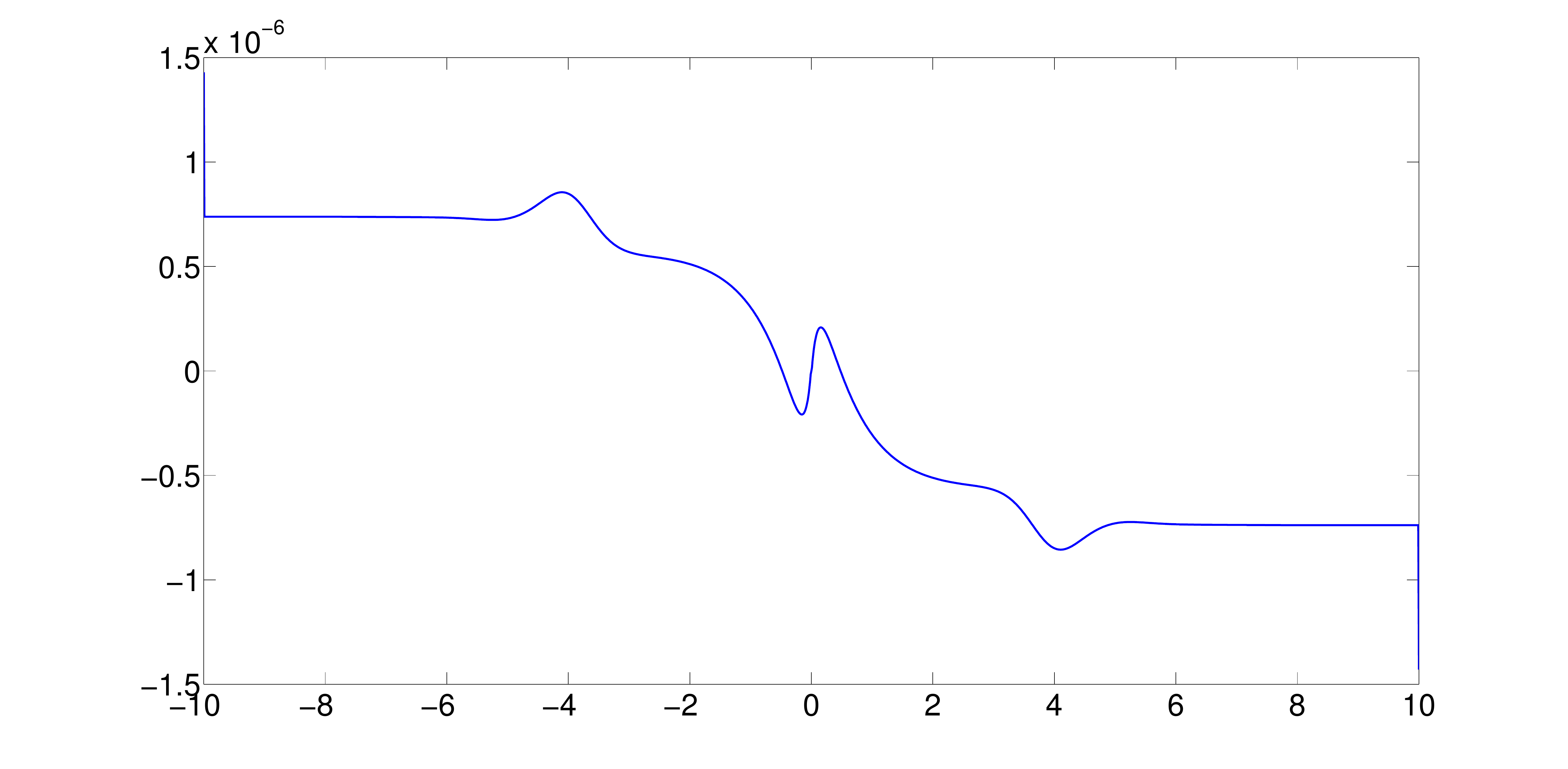}\\
      \includegraphics[scale=0.11]{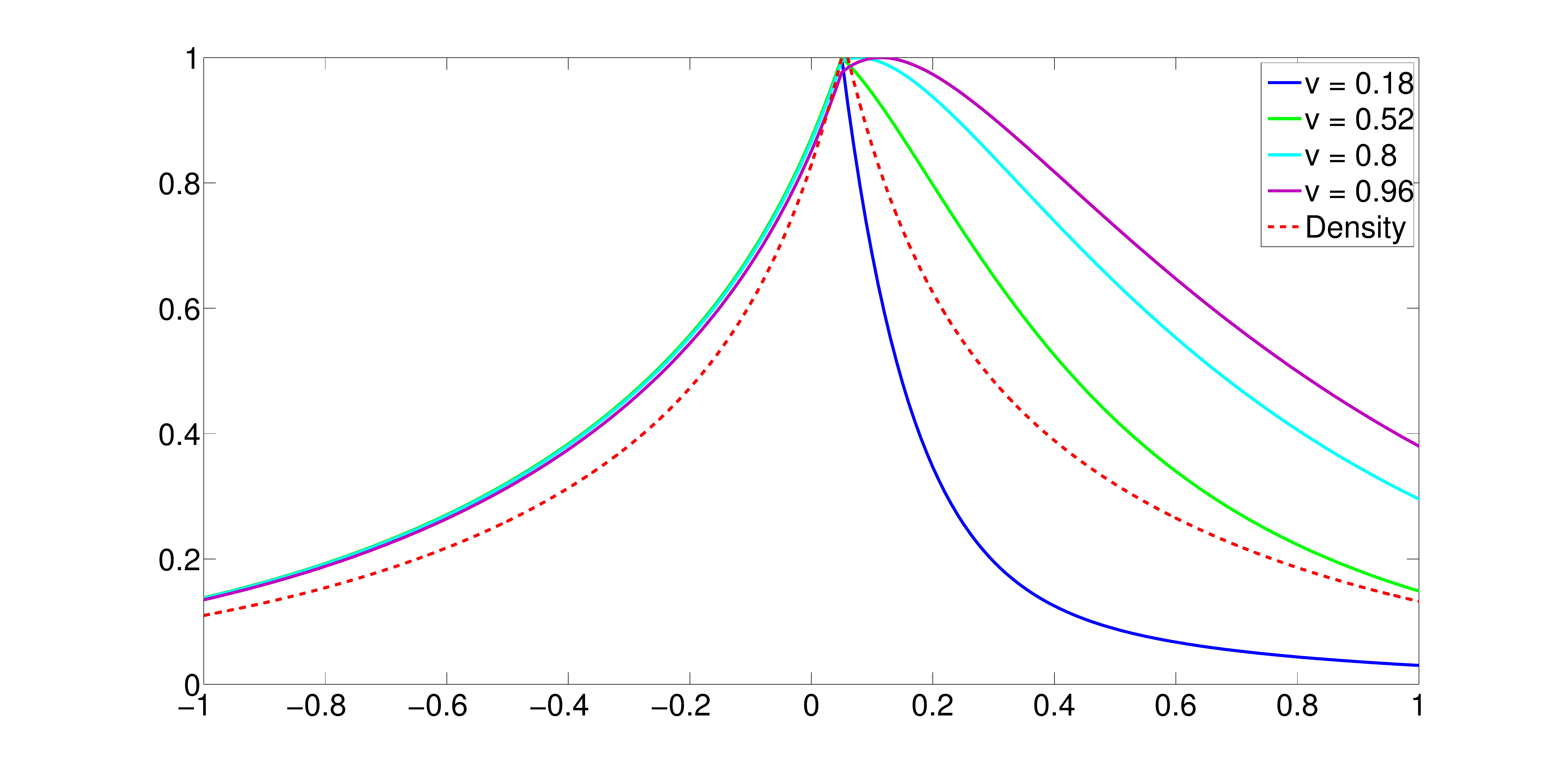}&\includegraphics[scale=0.11]{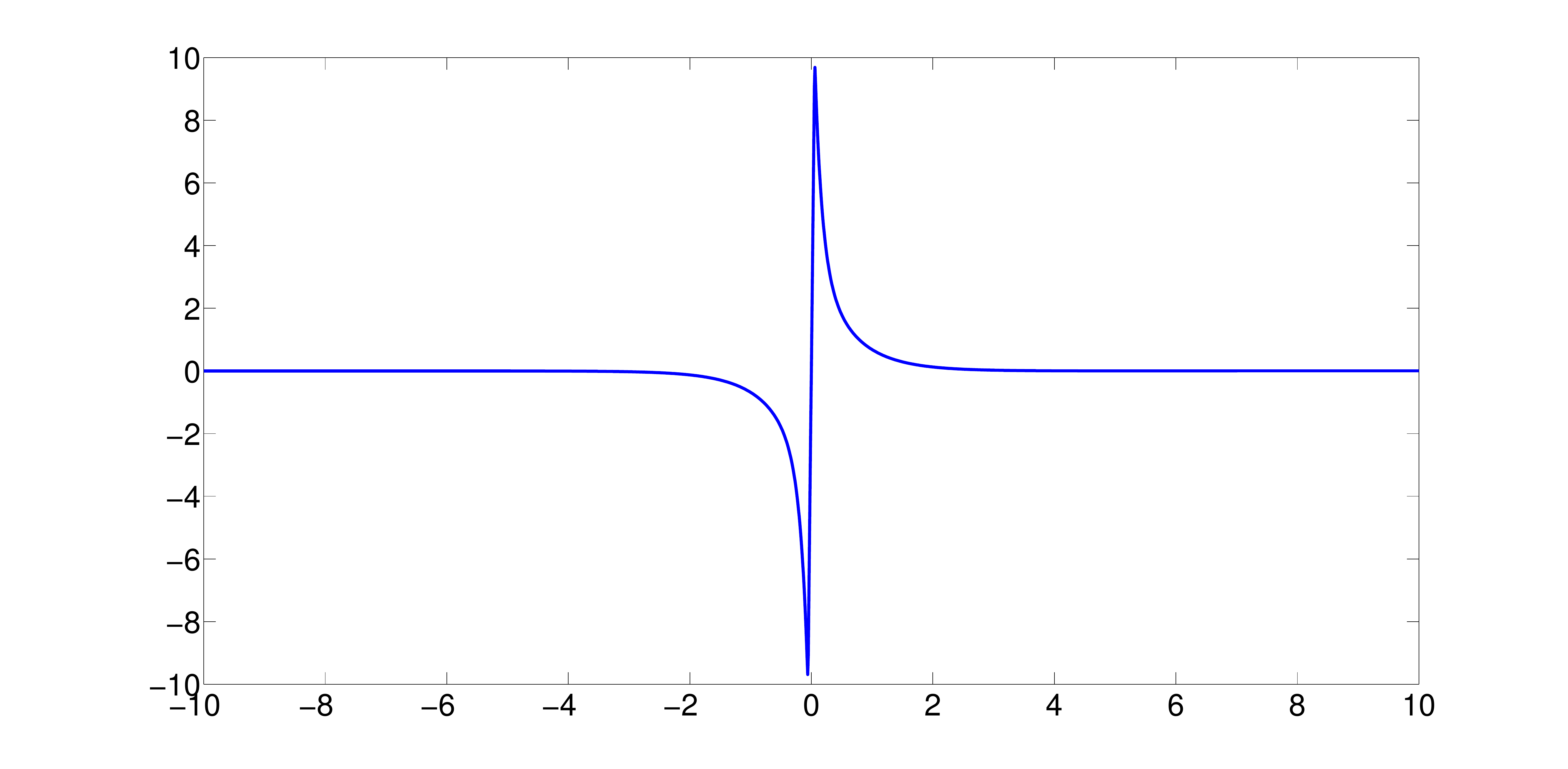}
    \end{tabular}
        \label{fig:compensation_timeSignal}
    }
    \caption{Symmetry of solutions to the aggregation model $(\chN=0)$ with $\chi_{\chS}=D_{\chS}=\alpha=\beta=1$, $f(0,x,v)=10 e^{-x^2-v^2}, \chS(0,x)=0$: (left) spatial distribution at steady-state of kinetic density $f(t,x,v)$ for positive velocities, macroscopic density $\rho(t,x)$  and (right) the error of symmetry $\Delta\rho=\rho(x)-\rho(-x)$ for WB-WB method  (top) and TS-TS (bottom) with $\nG=8$ point Gauss-Legendre quadrature and $\Delta x=0.01$.}
\label{fig:compensation}
\end{figure}
Accordingly, \eqref{eq:model} is considered with $\chi_{\chN}=0$, all parameters equal to one and initial data $f(0,x,v)=10 e^{-x^2-v^2}$. The velocity space is discretized by a $K=8$ point Gauss quadrature and $\Delta x=0.01$.  Figure~\ref{fig:compensation_fixedSignal} shows, in the left column, the kinetic density $f(t,x,v_k)$ for positive velocities (for negative velocities, profiles are symmetric) and the macroscopic density $\rho$ at steady-state for the WB-WB method (top) and TS-TS one (bottom); in the right column,  the symmetry breaking error $\Delta\rho=\rho(x)-\rho(-x)$.  For an aggregated signal, which is time-independent, i.e. $T=1+\chi_{\chS}\textrm{sign}(v\cdot x)$, both methods yield a macroscopic density correctly peaking at $x=0$ ($\Delta \rho\sim 10^{-14}$). In a time-dependent case $T=1-\chi_{\chS}\textrm{sign}\left(\frac{D\chS}{Dt}\right)$, see  Figure~\ref{fig:compensation_timeSignal}, only well-balanced gives accurate results with $\Delta\rho\sim 10^{-6}$ compared to $\Delta\rho\sim 10$ for the TS-TS method. 

%%%%%%%%%%%%%%%%%%%%%%%%%%%%%%%%%%%%%%%%%%%%%%%%%%%%%%%%%%%%%%%%%%%%%%%%%%%%%%%%%%%%%%%%%%%
\subsubsection{Approximation of the wave speed}

In biological experiments involving travelling pulses of \emph{E. coli} \cite{saragosti_mathematical_2010} bacteria were initially located at one end of the micro-channel filled with nutrient. They consumed the nutrient and moved towards its higher concentration. At some point an aggregate of bacteria was formed and traveled with a constant speed within a constant profile. In order to simulate this behavior it is necessary for a numerical scheme to be accurate enough when computing velocities, over large distances and for large times. 
We analyze the accuracy of different methods in approximating the speed of travelling waves emerging from \eqref{eq:model}  defined on $[0,L]\times[-1,1]$ with the following parameters
\[
\chi_{\chS}=0.48, \chi_{\chN}=0.44, D_{\chS}=0.5, \alpha=40, D_{\chN}=\beta=\gamma=1
\]
and the initial data
\[
f^0(x,v) = 3e^{-2x^2},\quad \chS^0(x)=0,\quad \chN^0(x)=400\left(\frac{\pi}{2}+\tanh\left(\frac{x}{3}-3\right)\right).
\]  
The simulations are performed on a mesh with $\Delta x=0.05$, $\V=\{-1,-0.5, 0.5,1\}$ and homogeneous weights $w_k=0.5$. Figure~\ref{fig:velocity_tw_formation} presents the comparison of the spatial distributions of the macroscopic velocity at time $t=100$ for MD-1 (on top) and MD-2 (on bottom) and WB-WB, WB-TS, TS-TS methods. The well-balanced method for the kinetic equation produce small oscillations in the back part of the wave, but the speed $c$ can be computed with good accuracy. The full time-splitting method (TS-TS) cannot balance correctly the macroscopic flux near the wave maximum resulting in a jump in the velocity profile. Also, our choice to approximate material derivatives through the upwinding technique (MD-2) improves significantly the accuracy. 
\begin{figure}[t]
  \begin{center}
    \begin{tabular}{c}
      \includegraphics[scale=0.2]{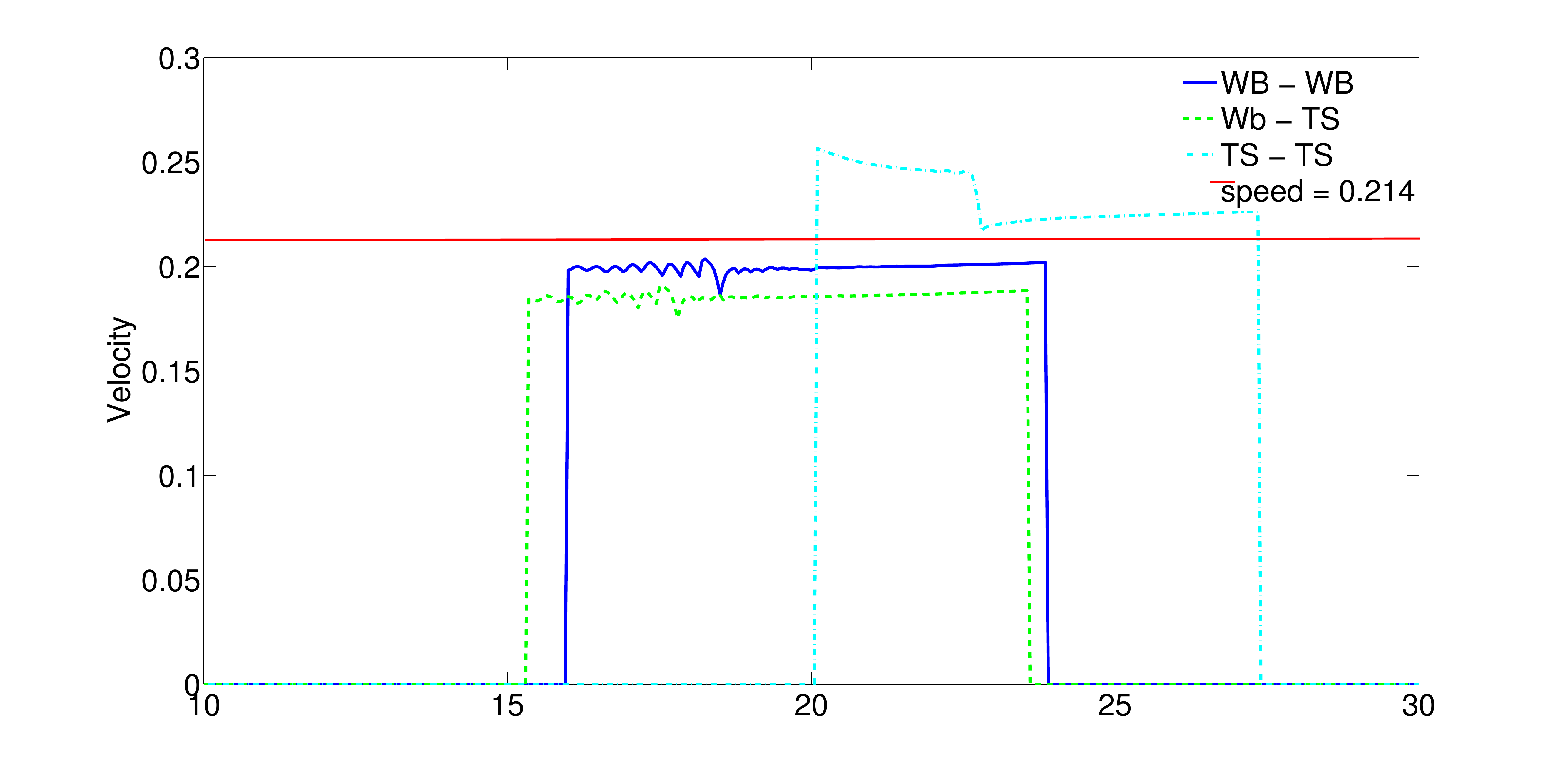}\\
      \includegraphics[scale=0.2]{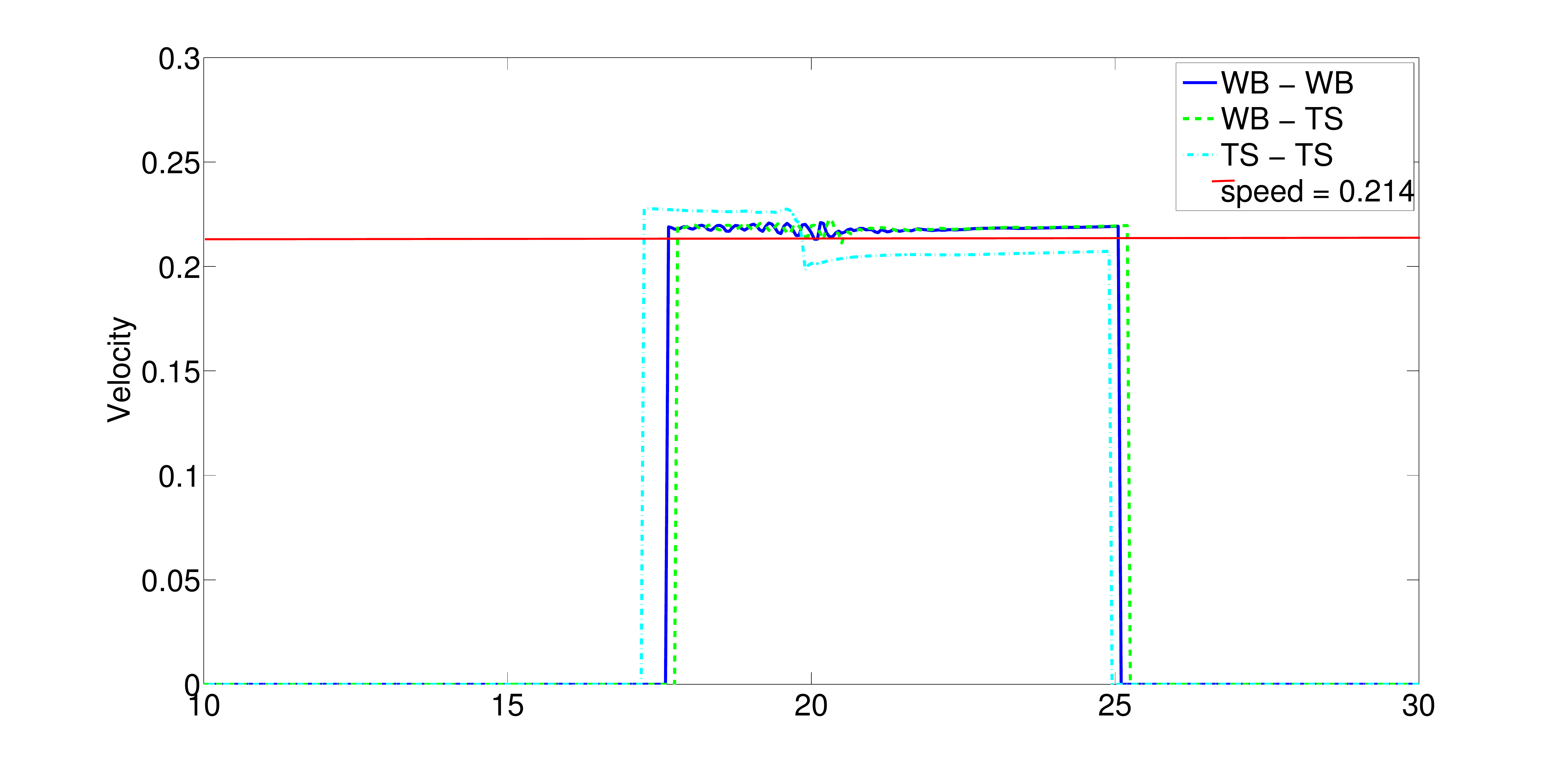}
    \end{tabular}
  \end{center}
  \caption{Comparison of the macroscopic velocities at $t=100$ obtained by different options of numerical methods, see Table \eqref{tab:schemes} for the details of the choices. (Top) Material derivative as in MD-1 \eqref{eq:md1}. (Bottom) Material derivative as in MD-2 \eqref{eq:md2}.}
  \label{fig:velocity_tw_formation}
\end{figure}

%%%%%%%%%%%%%%%%%%%%%%%%%%%%%%%%%%%%%%%%%%%%%%%%%%%%%%%%%%%%%%%%%%%%%%%%%%%%%%%%%%%%%%%%%%%
%%%%%%%%%%%%%%%%%%%%%%%%%%%%%%%%%%%%%%%%%%%%%%%%%%%%%%%%%%%%%%%%%%%%%%%%%%%%%%%%%%%%%%%%%%%
\section{Bi-stability of travelling waves}\label{sec:sim2}
%%%%%%%%%%%%%%%%%%%%%%%%%%%%%%%%%%%%%%%%%%%%%%%%%%%%%%%%%%%%%%%%%%%%%%%%%%%%%%%%%%%%%%%%%%%
%%%%%%%%%%%%%%%%%%%%%%%%%%%%%%%%%%%%%%%%%%%%%%%%%%%%%%%%%%%%%%%%%%%%%%%%%%%%%%%%%%%%%%%%%%%

Quantitative spectral analysis in the discrete velocity case \cite{calvez_existence} showed that the function $\Upsilon(c)=\partial_z\tilde{\chS}(z=0,c)$ \eqref{eq:signal_elliptic} may not be monotonically decreasing, but can have positive jumps, see \cite[Section 7]{calvez_existence}. As a consequence the condition $\Upsilon(c)=0$ can be satisfied by two different speeds $c$ corresponding to a slow and a fast wave, see Figure~\ref{fig:gammaC_vmin05}. This phenomenon depends on the velocity grid and model parameters. 

Below, we address the local stability of each of these waves. 
\begin{figure}[t]
  \begin{center}
    \includegraphics[scale=0.2]{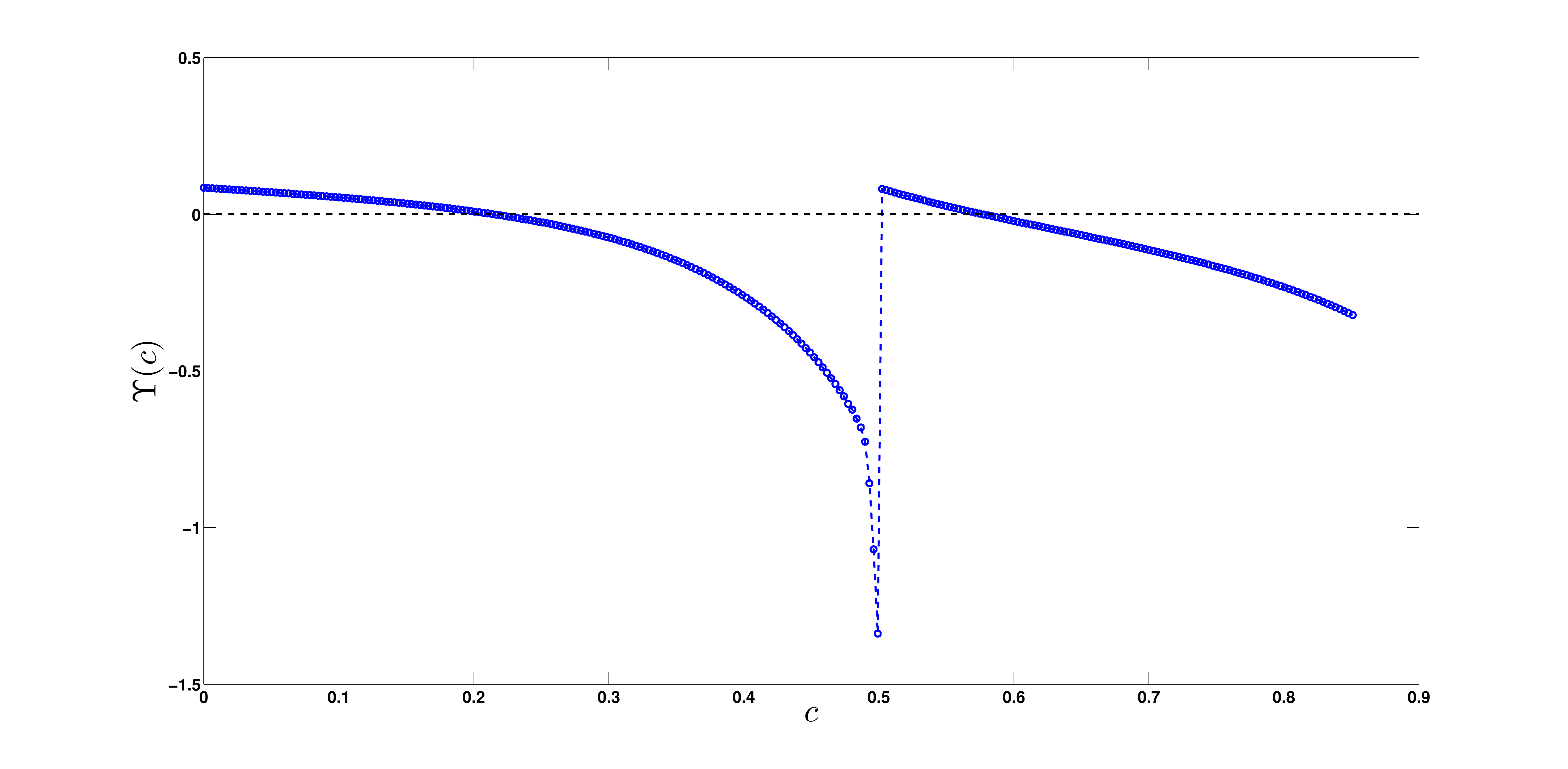}
  \end{center}
  \caption{Plot of the auxiliary function $\Upsilon(c)$ \eqref{eq:gamma_C}, for values of the parameter $c$ ranging from $0$ to $c^*$. Here,  $\V=\{-1,-0.5,0.5,1\}$, weights are uniform, and $\chi_{\chS}=0.48$, $\chi_{\chN}=0.44$, $\alpha=40$ and $D_{\chS}=0.5$. One clearly observe the co-existence of two wave speeds $c_s < c_f$ such that $\Upsilon(c_s) = \Upsilon(c_f) = 0$. Notice the positive jump at the transition $c = 0.5$ for which the problem is singular.}
  \label{fig:gammaC_vmin05}
\end{figure}
We opt for a very basic velocity set having only four values $\V=\{-1,-v_{\min},v_{\min},1\}$. We initialize the cell density profile with a stationary profile in the shifted frame $x-ct$ as solution of \eqref{eq:decoupled} for different values of parameter $c$. Initial distributions of the signal $\chS$ and the nutrient $\chN$ are the solutions to the stationary equations of the parabolic part of the model in the shifted frame.
 
Two problems are studied:
\begin{itemize}
\item For $v_{\min}=0.5$ and an appropriate choice of other parameters there exist two travelling waves, Figure~\ref{fig:gammaC_vmin05}. We show that they are both locally stable.
\item  A numerical bi-stability diagram is sought, for various values of $v_{\min}$.  
\end{itemize}

During the course of our analysis, we noticed that it is very challenging to preserve waves with high speed. Such waves are very narrow in the central part due to large values of dominant eigenmodes $\lambda$ in Case's solutions. To maintain the stability of fast waves a sufficient refinement in space, of order $\lambda^{-1}$ is required. Figure~\ref{fig:stability_dx} presents both macroscopic density and velocity of a fast wave corresponding to $v_{\min}=0.5$ at times $t=5$ and $t=30$ for two different grids $\Delta x = 0.05,\, 0.018$. The wave is not maintained on the coarse mesh.  For a fast wave corresponding to $v_{\min}=0.6$, the grid is such that $\Delta x=0.004$. 

\begin{figure}[t]
  \begin{center}
    \begin{tabular}{cc}
      \includegraphics[scale=0.1]{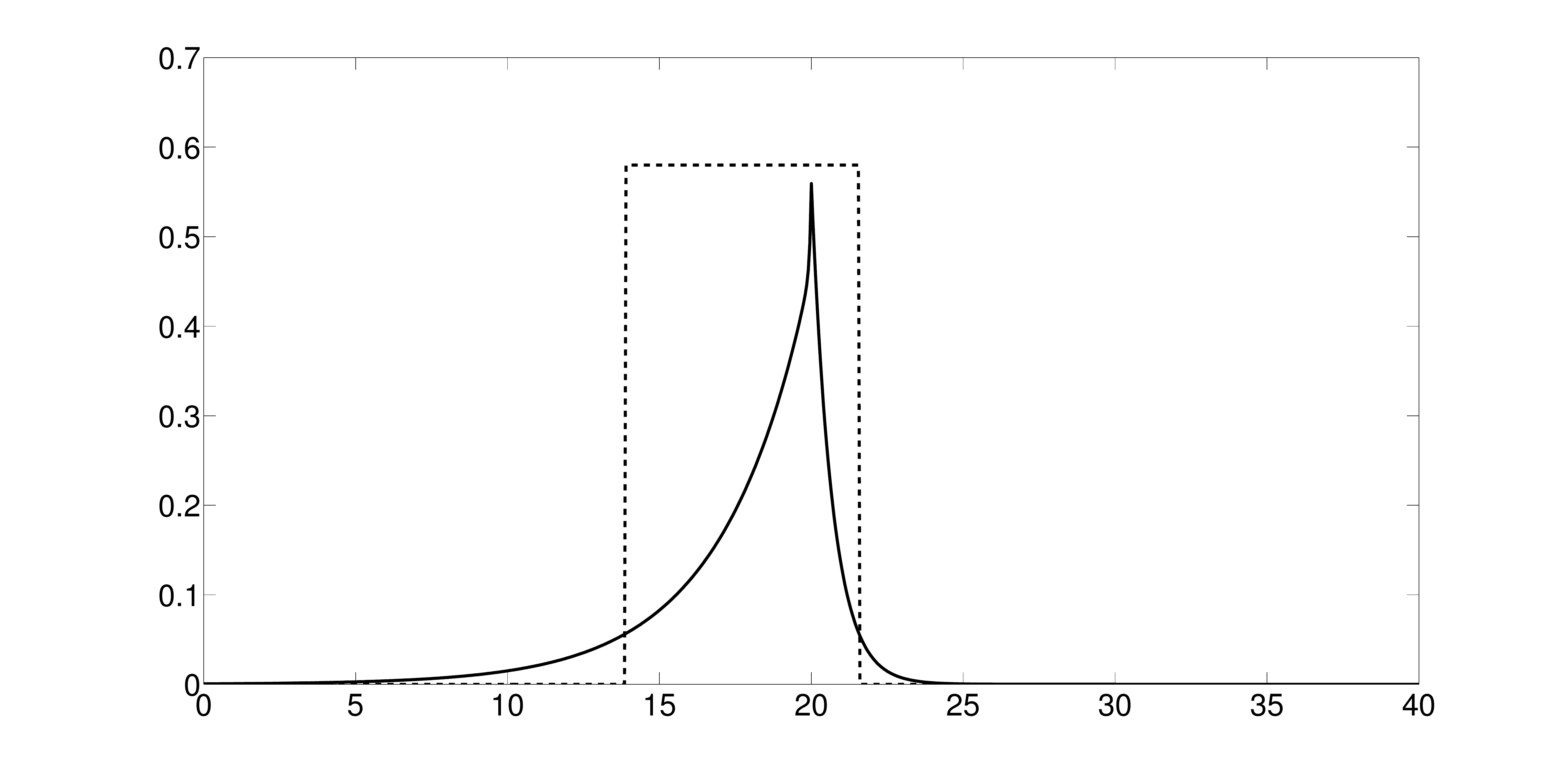} & \includegraphics[scale=0.1]{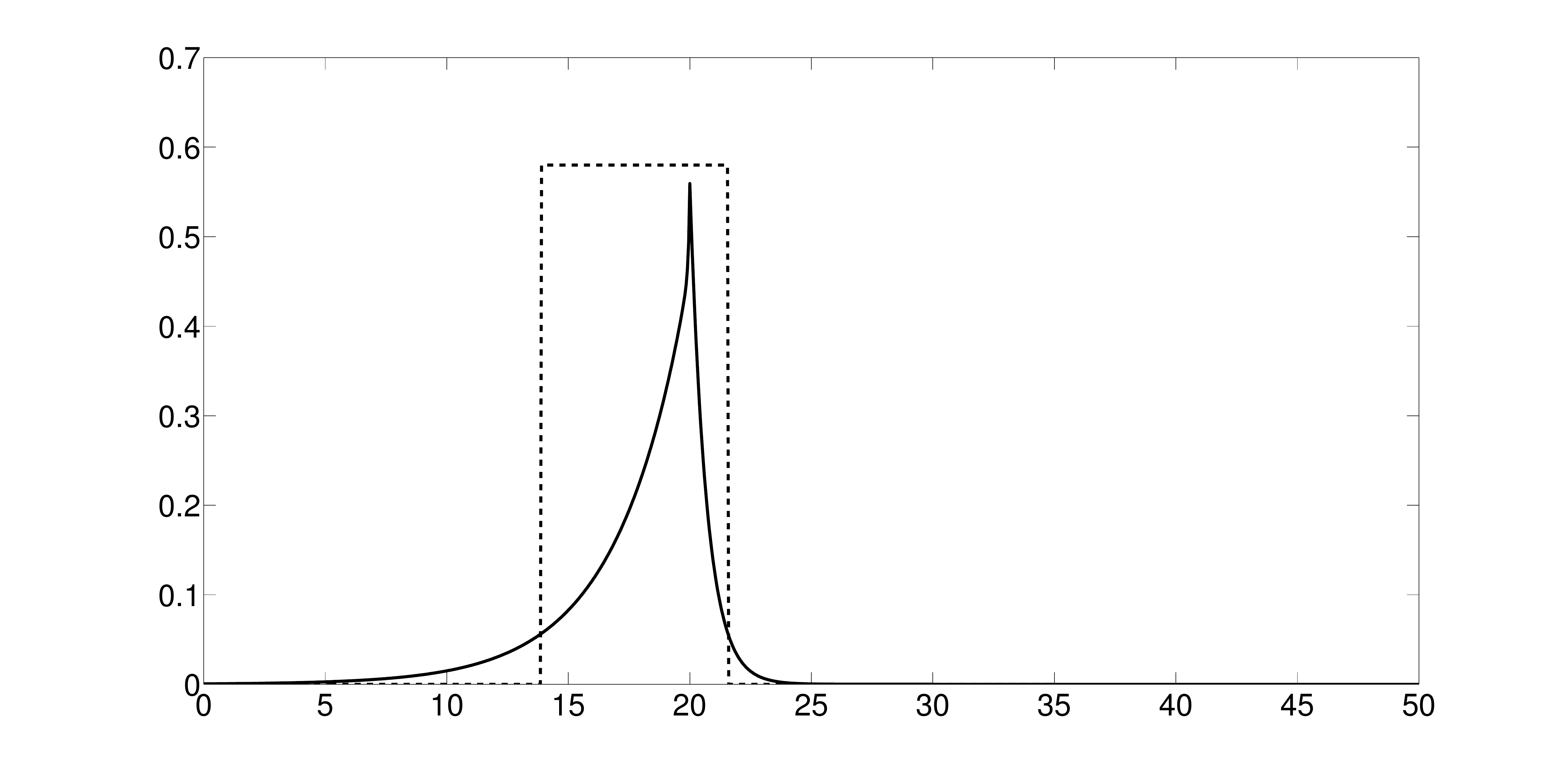}\\
      \includegraphics[scale=0.1]{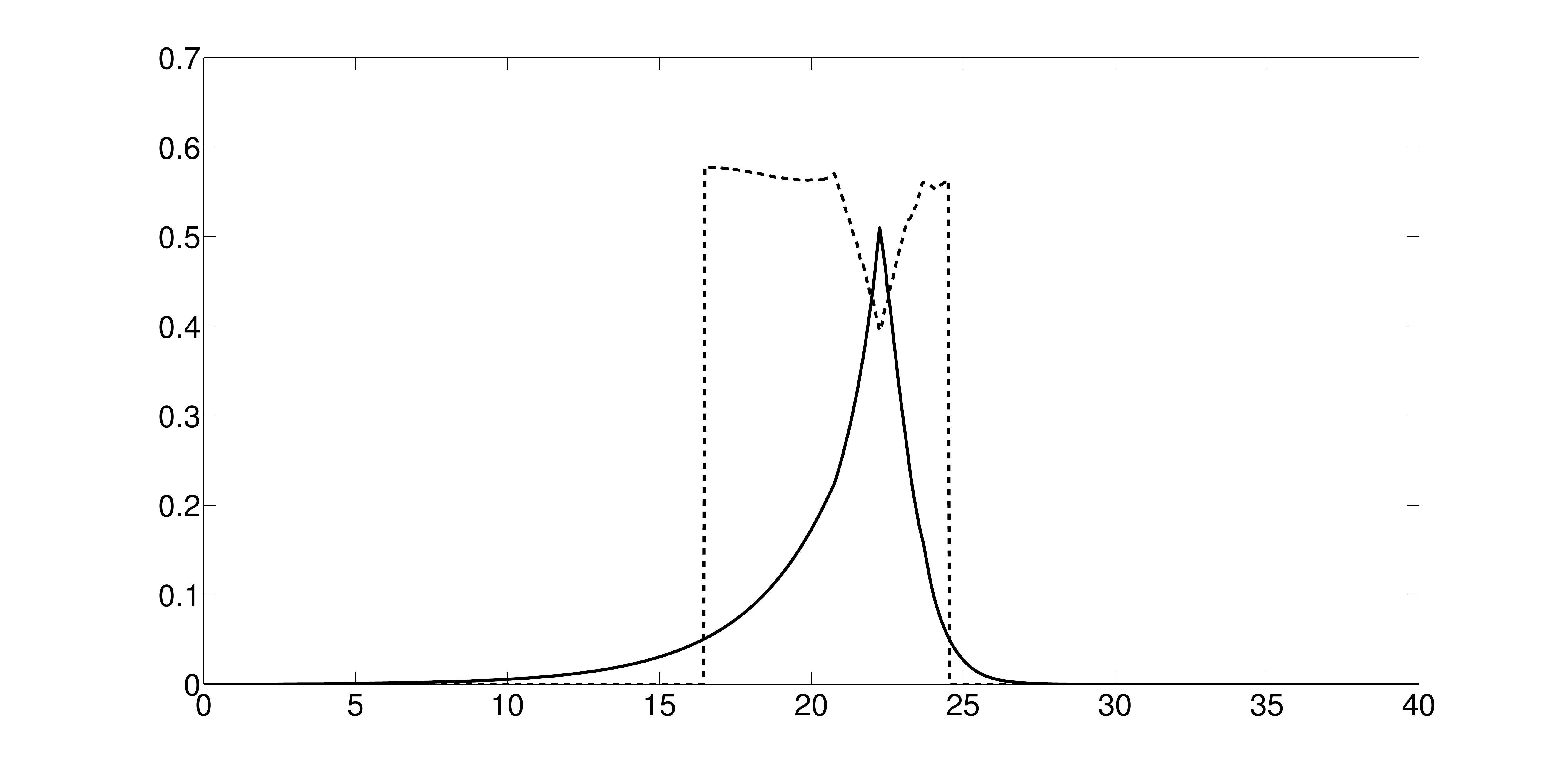} & \includegraphics[scale=0.1]{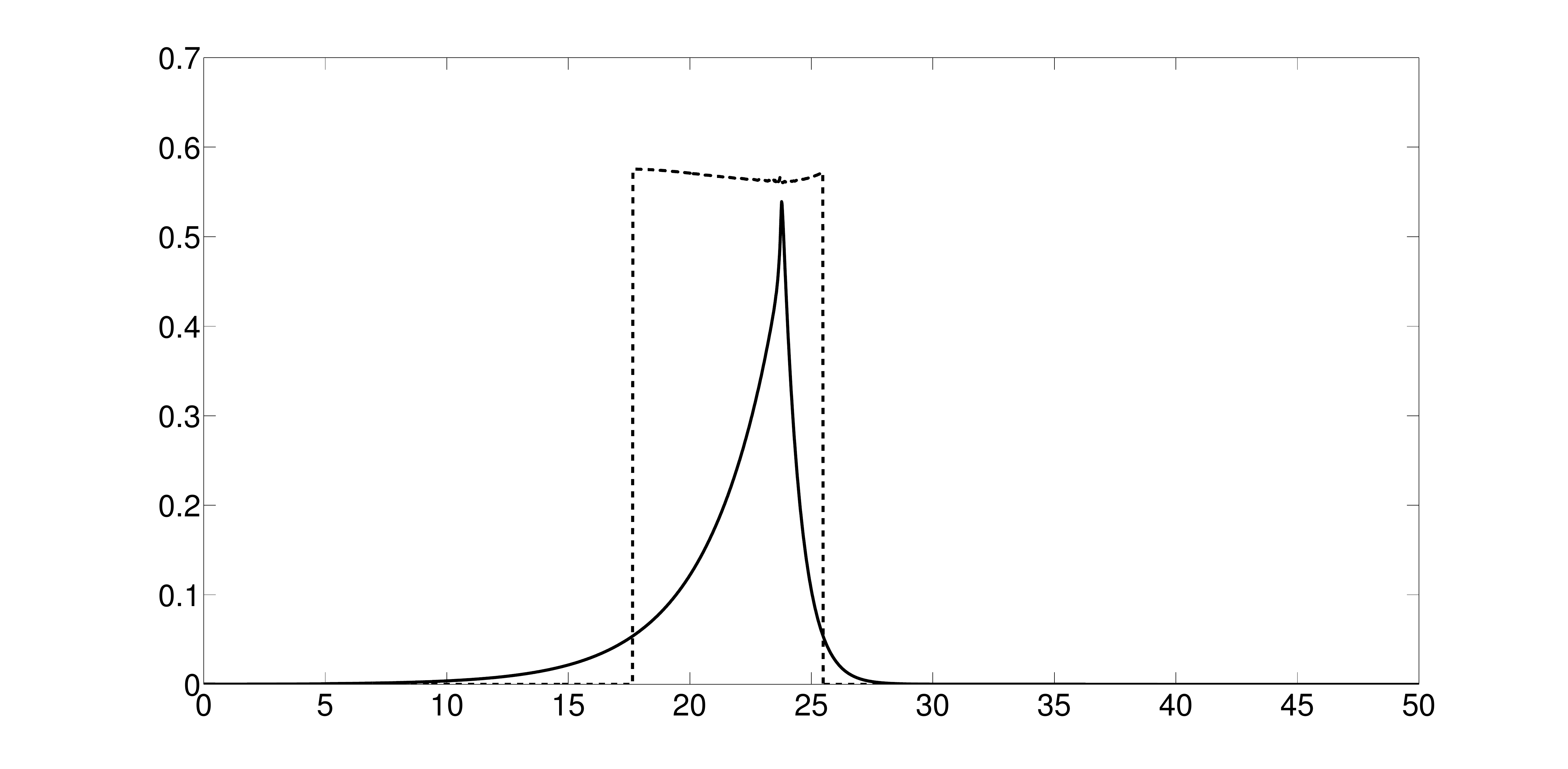}\\
      \includegraphics[scale=0.1]{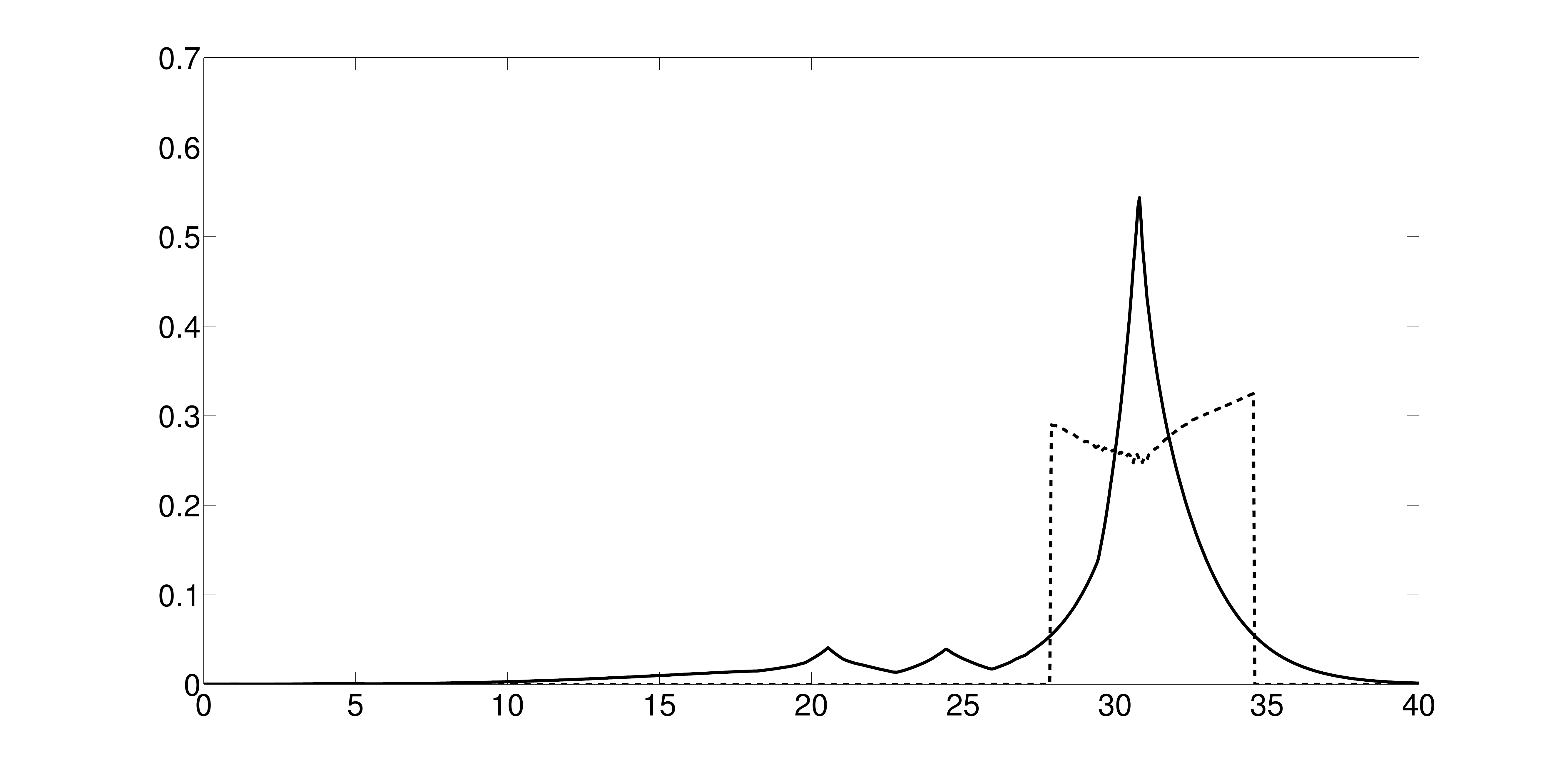}& \includegraphics[scale=0.1]{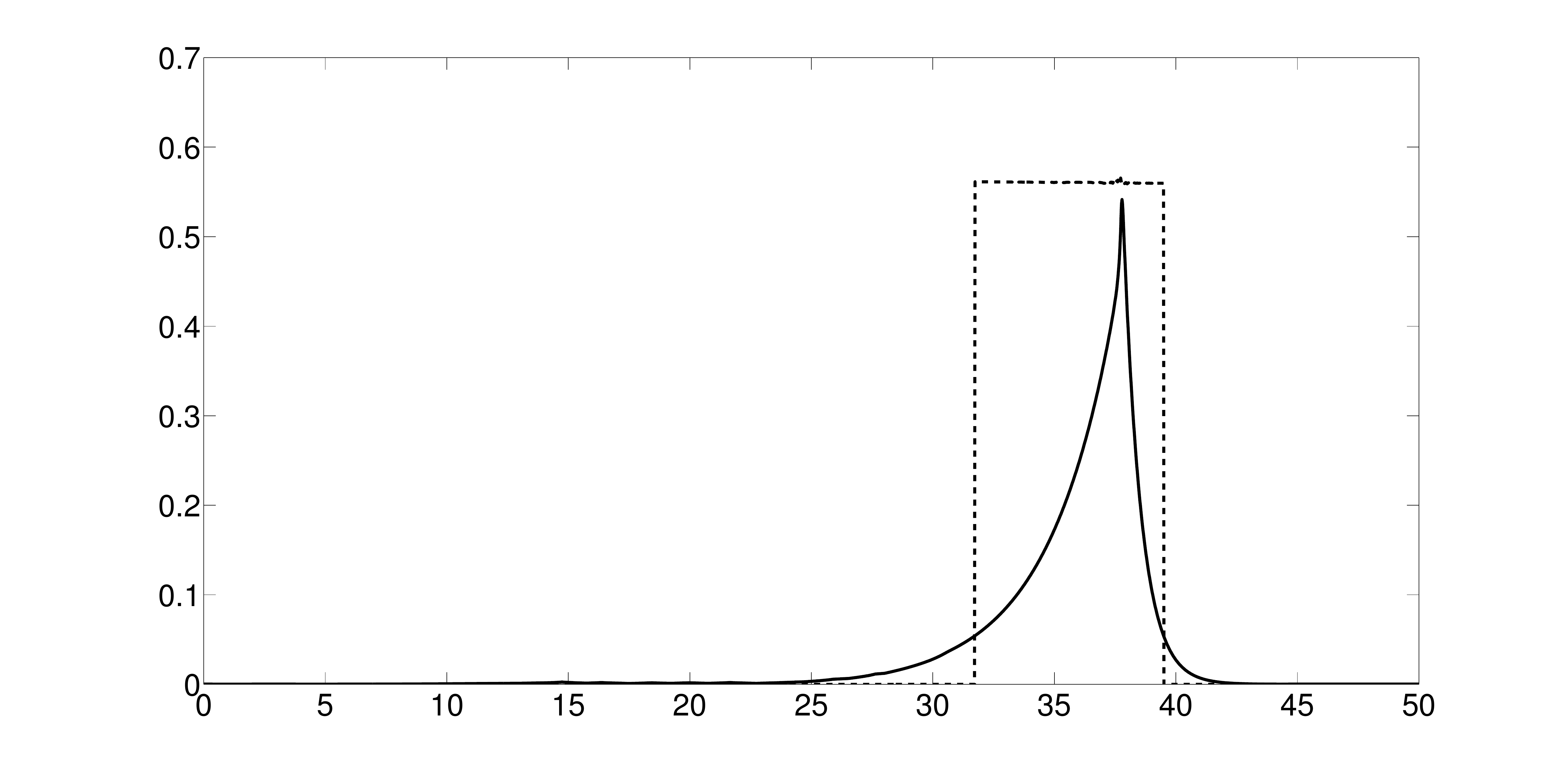}
    \end{tabular}
  \end{center}
  \caption{Effect of the mesh refinement: evolution of the macroscopic density (solid line) and the averaged velocity (dashed line), when the initial data is closed to the fast wave density profile $(c_f\approx 0.58)$. Parameters are the same as in Figure \ref{fig:gammaC_vmin05}. (Left) The space step is $\Delta x = 0.05$. Successive times are $t=0$ (top),  $t=5$ (middle), $t=30$ (bottom). (Right) The space step is  $\Delta x=0.018$. Times are the same. Notice the need for a small space step to capture the propagation of the fast wave. The slow wave is more robust for this choice of parameters.}
  \label{fig:stability_dx}
\end{figure}
%
%%%%%%%%%%%%%%%%%%%%%%%%%%%%%%%%%%%%%%%%%%%%%%%%%%%%%%%%%%%%%%%%%%%%%%%%%%%%%%%%%%%%%%%%%%%
\subsection{Observing bi-stability}
%%%%%%%%%%%%%%%%%%%%%%%%%%%%%%%%%%%%%%%%%%%%%%%%%%%%%%%%%%%%%%%%%%%%%%%%%%%%%%%%%%%%%%%%%%%

For $v_{\min}=0.5$ then the condition $\Upsilon(c)=0$ is satisfied by two values of $c$: $c_s \approx 0.214$ for the slow wave and $c_f \approx 0.58$ for the fast wave, see Figure~\ref{fig:gammaC_vmin05}. We show numerical simulations of the Cauchy problem which support the fact that the waves are both locally stable.
\begin{figure}[t]
  \begin{center}
    \begin{tabular}{cc}
      \includegraphics[scale=0.11]{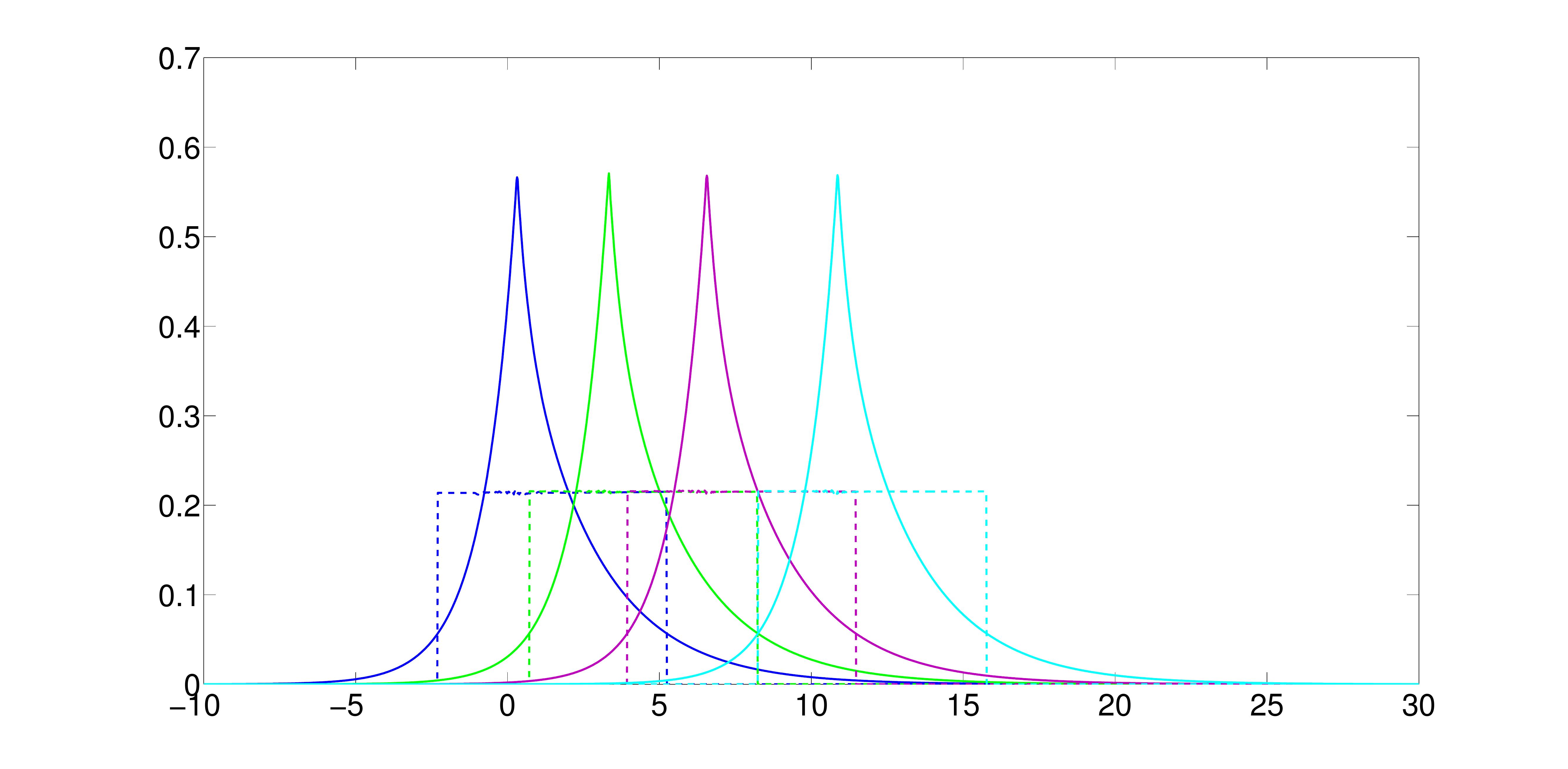}&\includegraphics[scale=0.11]{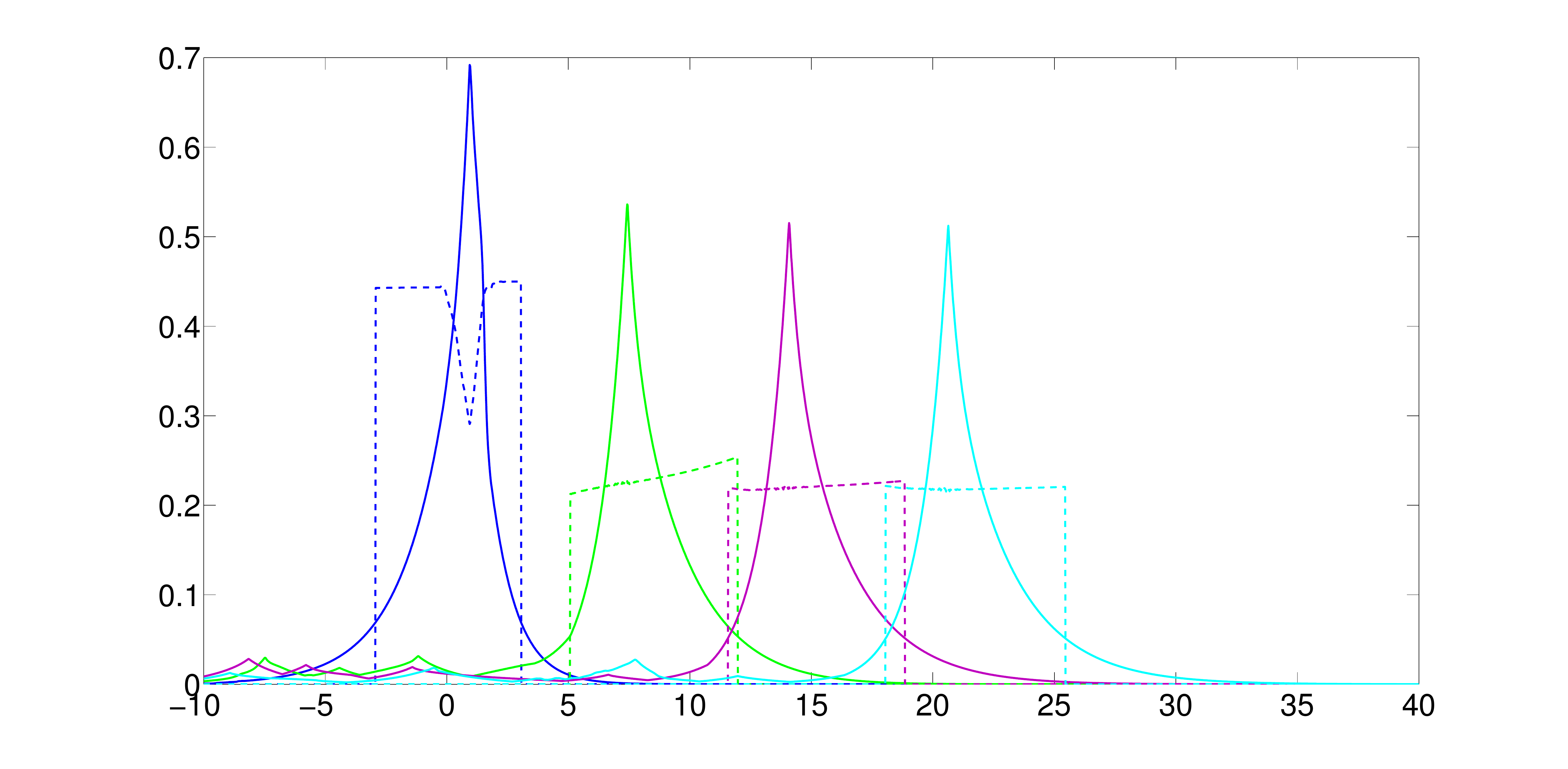}\\
      $c=0.2$&$c=0.45$\\
      \includegraphics[scale=0.11]{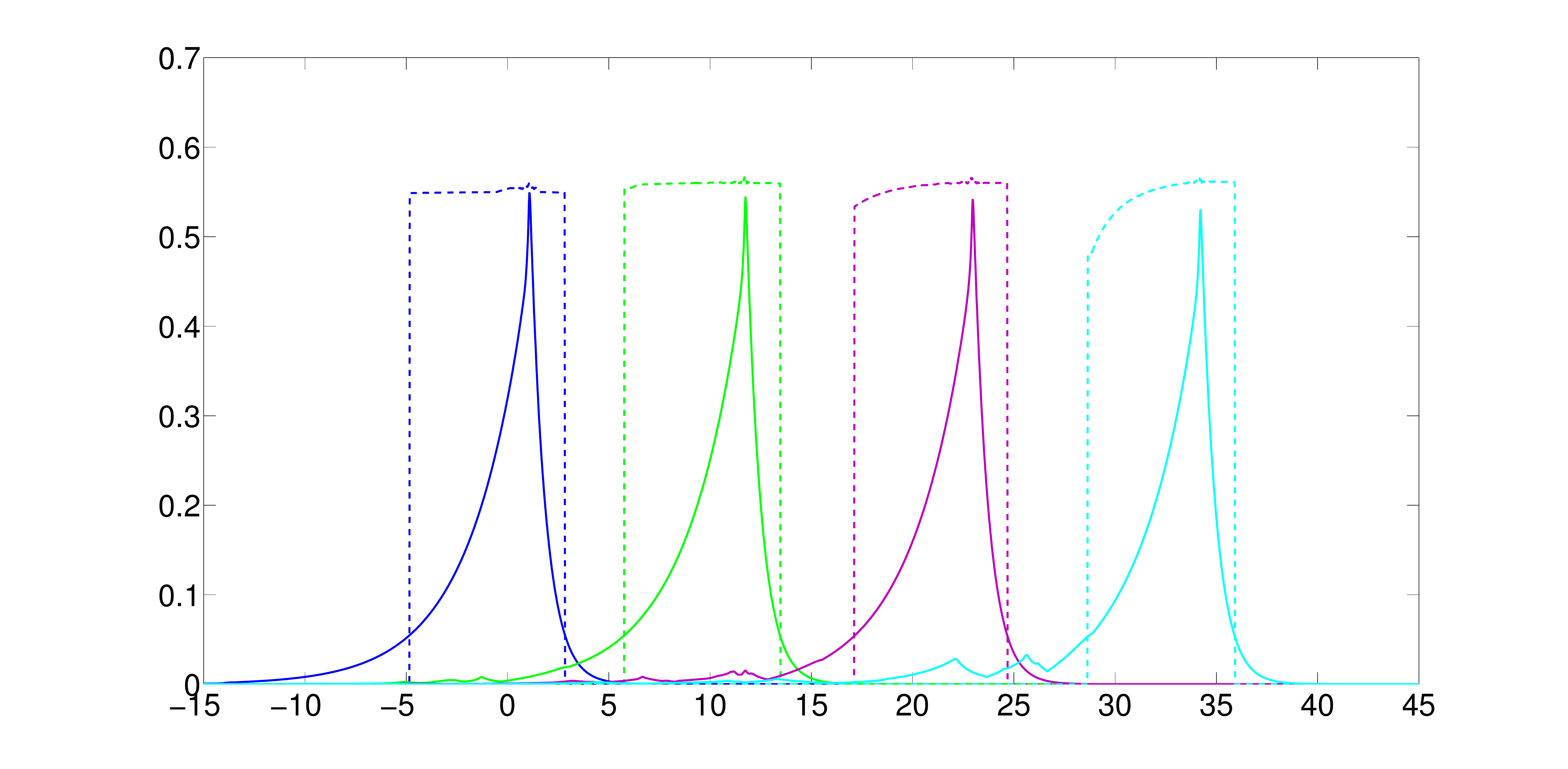}&\includegraphics[scale=0.11]{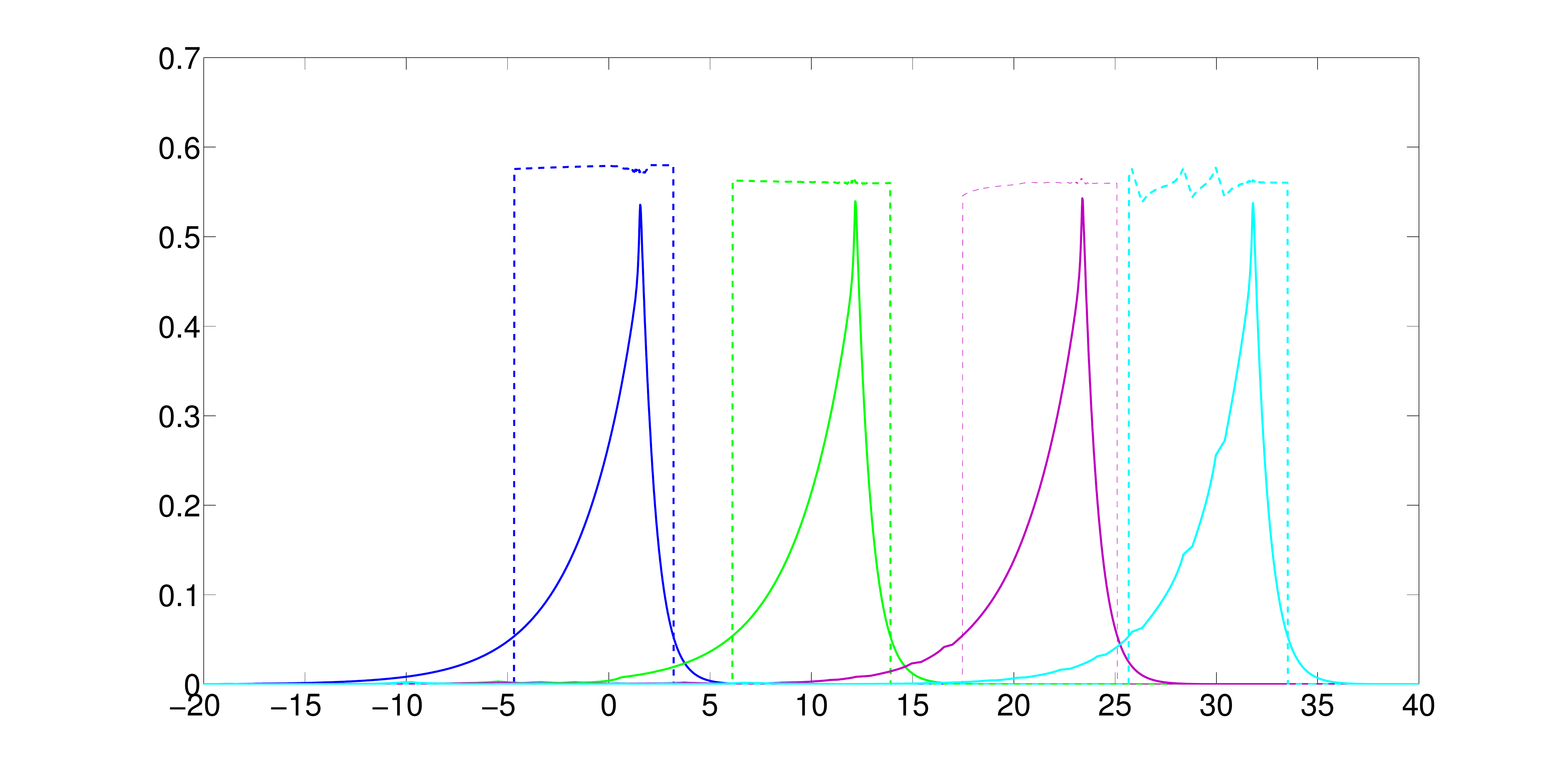}\\
      $c=0.55$&$c=0.58$
    \end{tabular}
  \end{center}
  \caption{Numerical evidence of local stability for both the slow wave and the fast wave. Parameters are the same as in Figure \ref{fig:gammaC_vmin05}. Each graph represents the space distribution of the macroscopic density (solid line) and the averaged velocity (dashed line) at different times. Each simulation is  initialized with a different profile, corresponding to the solution of the stationary uncoupled problem for parameters $c=0.214,0.45,0.55,0.58$. The first two simulations converge towards the slow wave, whereas the last two simulations converge towards the fast wave.}
  \label{fig:bistability}
\end{figure}
More precisely, we initialize the Cauchy problem with stationary solutions in shifted frame with four different speeds $c=(0.214,0.45,0.55,0.58)$ and we analyze their long time behavior. Figure~\ref{fig:bistability} presents the  distribution of the spatial density $\rho$ and the mean velocity $u$ at different times \eqref{eq:speed}. Two waves corresponding to the exact slow and fast wave are preserved. The other two initial conditions, for $c=0.45$ and $c=0.55$,  converge to each of the travelling waves, respectively. We note that smaller accuracy for faster waves is due to the low space resolution, as described in the previous section. 

%%%%%%%%%%%%%%%%%%%%%%%%%%%%%%%%%%%%%%%%%%%%%%%%%%%%%%%%%%%%%%%%%%%%%%%%%%%%%%%%%%%%%%%%%%%
\subsection{Bifurcation diagram}
%%%%%%%%%%%%%%%%%%%%%%%%%%%%%%%%%%%%%%%%%%%%%%%%%%%%%%%%%%%%%%%%%%%%%%%%%%%%%%%%%%%%%%%%%%%
%
Previously, we fixed the parameter $v_{\min}$  and we showed that various initial conditions converge asymptotically to the slow, or the fast wave given by the relation $\Upsilon(c)=0$. Now, $v_{\min}$ becomes  a free parameter in the study of local stability of admissible travelling waves. 
\begin{figure}[t]
  \begin{center}
      \includegraphics[scale=0.2]{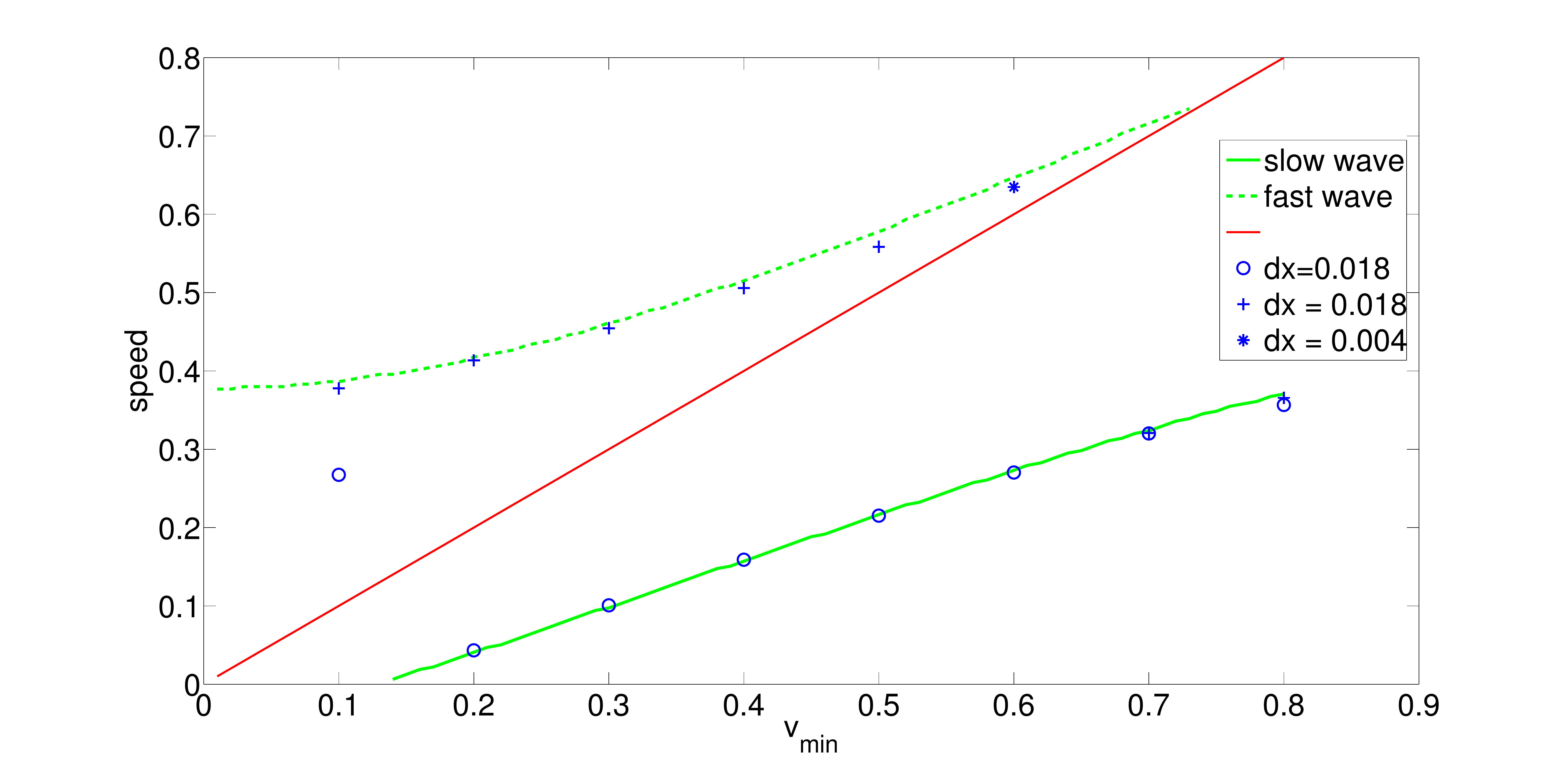}
  \end{center}
  \caption{Bifurcation diagram of the wave speed as a function of $v_{\min}$. This graph is a superposition of the branch corresponding to the slow wave (green, full line) and the branch corresponding to the fast wave (green, dashed line), obtained after a very accurate resolution of the wave speed equation $\Upsilon(c)=0$. The red line is simply the diagonal $c = v_{\min}$ for which the problem is singular. Markers correspond to the numerical values of the propagation speed  obtained using the long time simulations of the Cauchy problem~\eqref{eq:model}, with different choices of initial conditions. Parameters are the same as in Figure \ref{fig:gammaC_vmin05}. This shows the range of  $v_{\min}$ for which two waves co-exist. Interestingly, when a branch terminates, the long time simulations of the Cauchy problem jumps to the other branch in a counter-intuitive manner: by reducing the lowest velocity $v_{\min}$, the propagation speed is increased, and vice-versa.}
  \label{fig:bifurcation_diagram}
\end{figure}
The Cauchy problem is initialized close to the profile of either the slow or the fast wave, for each $v_{\min}$. A bifurcation diagram of wave speeds as a function of $v_{\min}$ is presented in Figure~\ref{fig:bifurcation_diagram}, including both exact roots of $\Upsilon(c)=0$ (for each $v_{\min}$,  green curves), and numerical ones obtained in the long time asymptotics (markers). The outcomes of this bifurcation diagram are as follows:
%\end{itemize}
%
\begin{itemize}
\item Both slow and fast waves remain stable numerically when they exist. 
\item The range of minimal velocity values $v_{\min}$, for which two waves co-exist, is wide. The smaller is the minimal velocity of the grid, the slower are the waves. Our numerical scheme captures reliably all these waves.
\item Behavior for extreme values of $v_{\min}$ is counterintuitive, but in agreement with theoretical analysis. If $v_{\min}$ is small (resp. large) enough than the slow (resp. fast)  wave disappears and all solutions stabilize on the fastest (resp. slowest)  wave. 
\item Numerical results for slow waves are in a very good agreement with theoretical curves. For $v_{\min}=0.1$ the initial wave corresponds to $c=0$ and convergence toward the fast wave is not well resolved. It might come from a too long simulation time along with numerical diffusion, which slows  velocity down. 
\item  Approximation of fast waves is quite challenging for any numerical process: as already explained, faster waves are much less aggregated in the large, and much narrow close to the peak. So, the scheme needs to balance both the transport and tumbling terms on larger domains, as well as to capture small spatial scales around the peak of the wave. Up to $v_{\min}=0.5$, our scheme manages to capture  travelling waves  with $\Delta x=0.018$. At $v_{\min}=0.6$, stability is ensured for a smaller grid size.
\end{itemize}

\section{Conclusion and outlook}

Chemotactic exponential travelling profiles were studied, both theoretically (by means of Theorem \ref{theo:kin TW} and its sketch of proof) and numerically (see Sections 3--5); in particular, unexpected bi-stability phenomena were observed, for which the accuracy of recent well-balanced kinetic and parabolic discretizations was severely tested. Overall results are satisfying, mainly because both fast and slow travelling waves, in the cases where they coexist, were captured in a stable way; however, fast waves may require a finer griding of the computational domain (see Fig. \ref{fig:bistability}). The practical bifurcation diagram agrees nicely with theoretical values, see Fig. \ref{fig:bifurcation_diagram}.

This being said, it sounds desirable to improve the global numerical strategy by getting rid of the ``splitting'' between the kinetic equation (\ref{kinetic}) and diffusion ones (\ref{para-S-N}). In a way similar to a 1D Riemann solver for a system of nonlinear conservation laws, a numerical handling of (\ref{kinetic})--(\ref{para-S-N}) as a whole set of equations is likely to bring more robustness and alleviate the griding constraints. Two angles of attack can be tried for building such a solver:
\begin{itemize}
\item a direct coupling strategy between the already existing $S$-matrix derivations and Steklov-Poincar\'e strategies presented in \cite{Gosse_Lsplines};
\item or building two-stream (diffusive) relaxation approximations of (\ref{para-S-N}) and consider an ``augmented kinetic model'' which encompasses the resulting three kinetic equations, and for which an ``augmented $S$-matrix'' might be found.
\end{itemize}
In a context of entropy-dissipating PDE's, the use of both $S$-matrices and $\mathcal L$-splines within numerical schemes allows to retrieve very high order accuracy close to steady-state, while maintaining the stencil as narrow as possible. The reason is that dissipation of entropy yields loss of information, hence irreversibility, so that distinguished ``equilibrium states'' do exist. For large-time simulations, it appears therefore sufficient to secure high accuracy only in the vicinity of such (problem-dependent) states, instead of asking for the same accuracy for a very wide class of (smooth) functions, which can actually be solutions of the considered problem only for a limited duration, at best, at the price of a more involved algebraic complexity of the algorithms.
\section*{Acknowledgments}
M.T. has benefited from the PICS Project CNR-CNRS 2015-2017 {\em Mod\`eles math\'ematiques et simulations num\'eriques pour le mouvement de cellules}.
This project received funding from the European Research Council (ERC) under the European Union's Horizon 2020 research and innovation programme (grant agreement No 639638).

\bibliography{twbib_NEW}
\bibliographystyle{siam}

\end{document}